\documentclass[review, 10pt]{article}

\usepackage[english]{babel}

\usepackage[letterpaper,top=2cm,bottom=2cm,left=3cm,right=3cm,marginparwidth=1.75cm]{geometry}

\usepackage{amsmath}
\usepackage{changes}
\usepackage{amsfonts}
\usepackage{amssymb}
\usepackage{natbib}

\usepackage{amsthm}
\usepackage{algorithm}
\usepackage{algpseudocode}
\usepackage{enumerate}   
\usepackage{mathrsfs}
\usepackage{colortbl}

\definecolor{asparagus}{rgb}{0.53, 0.66, 0.42}
\definecolor{ballblue}{rgb}{0.13, 0.67, 0.8}
\definecolor{cadmiumgreen}{rgb}{0.0, 0.42, 0.24}
\definecolor{cobalt}{rgb}{0.0, 0.28, 0.67}
\definecolor{darklavender}{rgb}{0.45, 0.31, 0.59}
\definecolor{green(pigment)}{rgb}{0.0, 0.65, 0.31}
\definecolor{blue(ncs)}{rgb}{0.0, 0.53, 0.74}
\definecolor{brandeisblue}{rgb}{0.0, 0.44, 1.0}
\definecolor{darkterracotta}{rgb}{0.8, 0.31, 0.36}

\definecolor{blue(ncs)}{rgb}{0.0, 0.53, 0.74}
\definecolor{brandeisblue}{rgb}{0.0, 0.44, 1.0}
\definecolor{ballblue}{rgb}{0.13, 0.67, 0.8}
\definecolor{blue(ryb)}{rgb}{0.01, 0.28, 1.0}
\definecolor{cobalt}{rgb}{0.0, 0.28, 0.67} % between normal and dark
\definecolor{ceruleanblue}{rgb}{0.16, 0.32, 0.75}
\definecolor{dodgerblue}{rgb}{0.12, 0.56, 1.0} % bright sky blue

\usepackage{subcaption}
\usepackage{graphicx}
\usepackage[colorlinks=true, allcolors=blue]{hyperref}

\usepackage{cleveref}

\newtheorem{theorem}{Theorem}[section]
\newtheorem{corollary}[theorem]{Corollary}
\newtheorem{lemma}[theorem]{Lemma}
\newtheorem{remark}[theorem]{Remark}
\newtheorem{definition}[theorem]{Definition}
\newtheorem{assumption}[theorem]{Assumption}

\newcommand{\quotes}[1]{``#1''}

\newcommand{\vertiii}[1]{{\left\vert\kern-0.25ex\left\vert\kern-0.25ex\left\vert #1 
    \right\vert\kern-0.25ex\right\vert\kern-0.25ex\right\vert}}

\newcommand{\dmu}{\hspace{1pt}\mathrm{d}x}

\renewcommand{\div}{\operatorname{div}}

\newcommand{\curlbf}{\boldsymbol{\operatorname{curl}}\hspace{1pt}}
\newcommand{\curl}{\operatorname{curl}}
\newcommand{\id}{\operatorname{id}}

\newcommand{\Hdiv}{\mathbf{H}(\operatorname{div},\Omega)}
\newcommand{\Hdivomega}[1]{\mathbf{H}(\operatorname{div},#1)}

\newcommand{\Hcurl}{\mathbf{H}(\curlbf,\Omega)}
\newcommand{\Hzerocurl}{\mathbf{H}_0(\curlbf,\Omega)}
\newcommand{\Hzerocurlomega}[1]{\mathbf{H}_0(\curlbf,#1)}

\newcommand{\Hzerodiv}{\mathbf{H}_0(\operatorname{div},\Omega)}
\newcommand{\Hzerodivomega}[1]{\mathbf{H}_0(\operatorname{div},#1)}

\newcommand{\Hdivzero}{\mathbf{H}(\operatorname{div}0,\Omega)}
\newcommand{\Hdivzeroomega}[1]{\mathbf{H}(\operatorname{div}0,#1)}

\newcommand{\Hzerodivzero}{\mathbf{H}_0(\operatorname{div}0,\Omega)}
\newcommand{\Hzerodivzeroomega}[1]{\mathbf{H}_0(\operatorname{div}0,#1)}

\newcommand{\Wdivzero}{W_{\operatorname{div}\hspace{-1pt}0}}
\newcommand{\VHdivzero}{V_{H,\operatorname{div}\hspace{-1pt}0}^k}

\newcommand{\VHkms}{V_H^{k,\mathrm{ms}}}

\newcommand{\VHkmsloc}{V_{H,m}^{k,\mathrm{ms}}}

\newcommand{\TauHz}{\mathcal{T}_{H,z}}

\newcommand{\cuniformity}{\gamma_{\mbox{\tiny qu}}}
\newcommand{\cshape}{\gamma_{\mbox{\tiny sr}}}

\newcommand{\PH}{\operatorname{P}_{\hspace{-1pt}H}}

\renewcommand{\dim}{\text{\normalfont dim}\,}

\newcommand{\Ltwo}{L^2(\Omega)}

\date{April 13, 2026}

\title{Stable localized orthogonal decomposition\\ in Raviart--Thomas spaces}

\begin{document}
\maketitle

\begin{center}
{\large Patrick Henning\footnote[1]{Department of Mathematics, Ruhr University Bochum, DE-44801 Bochum, Germany, \\ e-mail: \textcolor{blue}{patrick.henning@rub.de}.}},
{\large Hao Li\footnote[2]{Department of Applied Mathematics, The Hong Kong Polytechnic University, Hong Kong, \\ e-mail: \textcolor{blue}{hao94.li@polyu.edu.hk}.}}, 
{\large Timo Sprekeler\footnote[3]{Department of Mathematics, Texas A{\&}M University, College Station, TX 77843, USA, \\ e-mail: \textcolor{blue}{timo.sprekeler@tamu.edu}. 
}}\\[2em]
\end{center}

\begin{abstract}
This work proposes a computational multiscale method for the mixed formulation of a second-order linear elliptic equation subject to a homogeneous Neumann boundary condition, based on a stable localized orthogonal decomposition (LOD) in Raviart--Thomas finite element spaces. In the spirit of numerical homogenization, the construction provides low-dimensional coarse approximation spaces that incorporate fine-scale information from the heterogeneous coefficients by solving local patch problems on a fine mesh. The resulting numerical scheme is accompanied by a rigorous error analysis, and it is applicable beyond periodicity and scale-separation in spatial dimensions two and three. In particular, this novel realization circumvents the presence of pollution terms observed in a previous LOD construction for elliptic problems in mixed formulation. Finally, various numerical experiments are provided that demonstrate the performance of the method.
\end{abstract}

\section{Introduction}

In this work, we consider the finite element approximation of the Neumann problem for the prototypical second-order linear elliptic equation in divergence form, i.e.,
\begin{align}
\label{elliptic-equation-strong-form}
\begin{split}
-\div \left( \mathbf{A} \nabla p\right) &= f \qquad \mbox{in } \Omega,\\
\mathbf{A} \nabla p \cdot \mathbf{n} &= 0 \qquad \mbox{on } \partial\Omega,
\end{split}
\end{align}
posed on a bounded Lipschitz domain $\Omega\subset \mathbb{R}^d$ in dimension $d\in \{2,3\}$ with a piecewise polygonal boundary $\partial\Omega$ and outward unit normal $\mathbf{n}$. We consider \eqref{elliptic-equation-strong-form} in conjunction with the integral constraint
\begin{align}\label{p int constr}
    \int_{\Omega} p\;\mathrm{d}x = 0.
\end{align}
Here, $\mathbf{A}\in L^{\infty}(\Omega,\mathbb{R}^{d\times d}_{\mathrm{sym}})$ is a given uniformly elliptic and highly heterogeneous diffusion coefficient, and $f\in L^2(\Omega)$ is a given right-hand side that satisfies the compatibility condition $\int_{\Omega} f\,\mathrm{d}x = 0$. We highlight that we do not make any structural assumptions such as periodicity or scale-separation on the nature of the heterogeneity. Due to fast structural variations in the coefficient $\mathbf{A}$, and hence also in the solution $p$, problems of the form \eqref{elliptic-equation-strong-form} are called multiscale problems.

In a classical hydrological application, the problem \eqref{elliptic-equation-strong-form} is derived from Darcy's law \cite{Hel97} and describes the flow in a porous medium, where $\mathbf{A}$ is a highly heterogeneous permeability coefficient and the unknown $p$ is the pressure. A crucial property in hydrological simulations is the conservation of mass of the fluid. Analytically, this conservation is ensured by the continuity of fluid fluxes across interfaces of subdomains in normal direction. According to Darcy's law, these fluxes are given by $- \mathbf{A} \nabla p$. In order to reproduce this conservation on a discrete level in numerical approximations, it is common to rewrite equation \eqref{elliptic-equation-strong-form} in a suitable mixed formulation. To be precise, considering the velocity field $\mathbf{u} = \mathbf{A} \nabla p$ as our second unknown, the mixed formulation of the problem \eqref{elliptic-equation-strong-form}--\eqref{p int constr} seeks a velocity-pressure pair $(\mathbf{u},p)\in \Hdivomega{\Omega}\times L^2(\Omega)$ with $\left.\mathbf{u}\cdot \mathbf{n}\right\rvert_{\partial\Omega} = 0$ and $\int_\Omega p\,\mathrm{d}x = 0$, such that
\begin{align}\label{eqn-mixed-form-var-intro}
\begin{split}
\left(\mathbf{A}^{-1} \mathbf{u}, \mathbf{v}\right)_{L^2(\Omega)}+(\div \mathbf{v}, p)_{L^2(\Omega)} & =0 \qquad\qquad\qquad \mbox{for all }\; \mathbf{v} \in \Hdivomega{\Omega} \text{ with}\left.\mathbf{v}\cdot \mathbf{n}\right\rvert_{\partial\Omega} = 0, \\
(\div \mathbf{u}, q)_{L^2(\Omega)} & = -(f, q)_{\Ltwo} \quad\mbox{ for all }\; q \in L^2(\Omega).    
\end{split}
\end{align}
An application of the Brezzi-splitting \cite{boffi2013mixed} reveals that this mixed problem has a unique solution, and, in fact, the solution pair is given by $(\mathbf{u},p) = (\mathbf{A}\nabla p,p)$ with $p\in H^1(\Omega)$ being the unique weak solution to \eqref{elliptic-equation-strong-form}--\eqref{p int constr}; see Section \ref{Sec: Cts prob}. 

The main objective of our work is the construction and rigorous error analysis of a practical finite element scheme based on the methodology of localized orthogonal decomposition for the accurate numerical approximation of the multiscale problem \eqref{eqn-mixed-form-var-intro}, which is also referred to as numerical homogenization. The underlying coarse discretization uses Raviart--Thomas finite element spaces \cite{RT77}. 

To date, there are a wide variety of methodologies for the construction of finite element based numerical homogenization schemes, among the most popular being the multiscale finite element method (MsFEM) \cite{HW97,EH09} and its generalizations (GMsFEM) \cite{ChEfHo23,EfGaHo13}, the heterogeneous multiscale method (HMM) \cite{EE03,AEE12}, the variational multiscale method (VMS) \cite{Hug95,HFM98}, the multiscale spectral generalized finite element method (MS-GFEM) \cite{BaLi11,BaLiSiSt20,MaSc22,MaScDo22}, and the localized orthogonal decomposition (LOD) \cite{MP14,MP21,LODActa21}. This work is following the methodology of the LOD which has been the subject of extensive research in the last decade; see, e.g., \cite{HeP13,KPY18,gallistl2018numerical,Maier21,MaVe22,Ver22,HP23,FGP24,DHW24} and the references therein for some recent developments. 

Regarding mixed finite element methods for multiscale problems involving heterogeneous coefficients, MsFEM- and GMsFEM-type schemes utilizing Raviart--Thomas finite elements have been suggested in \cite{Aar04,AB06,CH03,ChEfLe15,CCE22,WaChZh21}, homogenization-based schemes have been studied in \cite{Arb11} for divergence-form equations and in \cite{GSS21} for nondivergence-form equations and VMS-type schemes have been discussed in \cite{Arb04,LM09,Mal11}. In a recent article \cite{AMS25}, a MS-GFEM realization for the mixed problem \eqref{eqn-mixed-form-var-intro} is provided based on optimal local approximation spaces obtained by solving local eigenvalue problems. Finally, LOD-type schemes for mixed problems have been proposed in \cite{hellman2016multiscale,HL24} for Darcy and Stokes problems, respectively. 

The main goal of this paper is to generalize the previous work \cite{hellman2016multiscale} in various ways. The main issue faced in the original work is that the LOD approximation space, used in practical calculations, is not based on a stable decomposition. To be precise, the LOD space is obtained by enriching Raviart--Thomas finite element functions on a coarse mesh by divergence-free functions that are in the kernel of the nodal Raviart--Thomas interpolation $\operatorname{I}_H$. However, this interpolation is not stable on $\Hdiv$ which in turn implies that the kernel of $\operatorname{I}_H$ is not closed in $\Hdiv$. This issue is formally resolved in the analysis of \cite{hellman2016multiscale} by replacing the exact solution space $\Hdiv$ by a finite-dimensional approximation space on a sufficiently fine mesh. However, this in turn results in a logarithmic pollution factor (depending on the coarse-to-fine mesh size ratio) in the error estimates, which causes a blow-up to infinity if the fine mesh size tends to zero. Hence, the numerical method becomes logarithmically instable on fine meshes, which was also confirmed numerically. Furthermore, the use of the Raviart--Thomas interpolation and the specific proof technique in \cite{hellman2016multiscale} to quantify the localization error restricted the analysis to dimension $d=2$. Our new approach not only eliminates the pollution factor via a stable LOD in Raviart--Thomas spaces, but it is also valid for dimensions $d\in\{2,3\}$. Finally, we provide a construction of arbitrary polynomial order, whereas the original construction was restricted to lowest-order Raviart--Thomas elements. With this, we also generalize the concept of higher-order LOD methods \cite{Maier21,DoHaMa23} to a new problem class.

Let us briefly outline the main ideas of this paper. As a first step, we introduce the (low-dimensional) coarse-scale space $V_H^k :=  \mathcal{RT}_{\hspace{-1pt}k}(\mathcal{T}_{H}) \cap \Hzerodiv$ using the Raviart--Thomas finite element space of order $k\in\mathbb{N}_0$ for the velocity, and the coarse-scale space $Q_H^k$ consisting of $\mathcal{T}_H$-piecewise polynomials of degree at most $k$ for the pressure. In the spirit of the LOD methodology, we decompose the solution space $\Hzerodivomega{\Omega}:=\left\{\boldsymbol{u} \in \Hdivomega{\Omega}: \left.\boldsymbol{u} \cdot \mathbf{n}\right\rvert_{\partial\Omega}=0\right\}$ for the velocity into
\begin{align}\label{orig dec intro}
    \Hzerodivomega{\Omega} = V_H^k \oplus W, 
\end{align}
where $W := \operatorname{ker}\left(\pi_H\right)$ is a fine-scale space, or detail space, defined via a (computable and local) stable quasi-interpolation operator $\pi_H: \Hzerodiv\rightarrow V_H^k$ that is a projection and satisfies the commuting property $\div\circ\, \pi_H = P_H\circ \div$, with $P_H:L^2(\Omega)\rightarrow Q_H^k$ denoting the $L^2$-orthogonal projection onto $Q_H^k$. A possible choice for such a mapping $\pi_H$ is given in \cite{ern2022equivalence}. See Sections \ref{Sec: RT} and \ref{Sec: interp} for more details.

As a second step, we use \eqref{orig dec intro} to construct a modified decomposition
\begin{align}\label{modified dec intro}
    \Hzerodivomega{\Omega} = V_H^{k,\mathrm{ms}} \oplus W,\qquad V_H^{k,\mathrm{ms}} := (\mathrm{id}-\mathcal{C})V_H^k,
\end{align}
where $\mathcal{C}:\Hzerodiv\rightarrow \Wdivzero := \{\mathbf{w}\in W:\div \mathbf{w}=0\}$ is a linear correction operator that uses fine-scale information from the coefficient $\mathbf{A}$ to enrich the coarse-scale space $V_H^k$ to an ideal multiscale space $V_H^{k,\mathrm{ms}}$ of the same dimension, and such that we have the crucial orthogonality property
\begin{align*}
    \left(\mathbf{A}^{-1}\mathbf{v}_H^{\mathrm{ms}},\mathbf{w}\right)_{L^2(\Omega)} = 0\qquad\mathrm{for}\; \mathrm{all}\;\; \mathbf{v}_H^{\mathrm{ms}} \in \VHkms, \; \mathbf{w} \in \Wdivzero.
\end{align*}
It is important to note that adding an element in $\Wdivzero$ to a given function does not change its coarse-scale behavior or its divergence. See Section \ref{Sec: Corr op and ideal space} for more details.

With the modified decomposition \eqref{modified dec intro} at hand, the ideal numerical homogenization scheme seeks a pair $(\mathbf{u}_H^{\mathrm{ms}},p_H)\in \VHkms\times Q_H^k$ with $\int_\Omega p_H\,\mathrm{d}x = 0$, such that
\begin{align}\label{ideal method intro}
\begin{split}
\left(\mathbf{A}^{-1} \mathbf{u}_H^{\mathrm{ms}}, \mathbf{v}_H^{\mathrm{ms}}\right)_{L^2(\Omega)}+(\div \mathbf{v}_H^{\mathrm{ms}}, p_H)_{L^2(\Omega)} & =0 \qquad\qquad\qquad\;\;\, \mbox{for all }\; \mathbf{v}_H^{\mathrm{ms}} \in \VHkms, \\
(\div \mathbf{u}_H^{\mathrm{ms}}, q_H)_{L^2(\Omega)} & = -(f, q_H)_{\Ltwo} \quad\mbox{ for all }\; q_H \in Q_H^k.    
\end{split}
\end{align}
A rigorous error analysis for the ideal scheme can be found in Section \ref{Sec: error ideal}.

Finally, since the ideal method, as the name suggests, is not yet practical due to the corrector problems involved in the definition of $\mathcal{C}$ being global problems, we construct a practical numerical homogenization scheme by localizing the corrector problems to small regions while preserving their approximation quality. The resulting method is stated and rigorously analyzed in Section \ref{Sec: localized}.

This paper is structured as follows. 

In Section \ref{Sec: prelim}, we discuss some preliminaries. After introducing notations used throughout this work (Section \ref{Sec: notation}), we review the well-posedness of the continuous problem \eqref{eqn-mixed-form-var-intro} (Section \ref{Sec: Cts prob}), its coarse discretization with the Raviart--Thomas finite element (Section \ref{Sec: RT}), and we state the framework and an explicit example for the stable quasi-interpolation operator $\pi_H$ (Section \ref{Sec: interp}). 

In Section \ref{Sec: Ideal}, we state and rigorously analyze the ideal multiscale method. The correction operator $\mathcal{C}$ and the ideal approximation space $V_H^{k,\mathrm{ms}}$ are constructed in Section \ref{Sec: Corr op and ideal space}, and the error analysis of the ideal method \eqref{ideal method intro} is given in \ref{Sec: error ideal}.

In Section \ref{Sec: localized}, we state and rigorously analyze the localized multiscale method, that is, the practical numerical homogenization scheme for the approximation of the multiscale problem \eqref{eqn-mixed-form-var-intro}. The localized version $\mathcal{C}^m$ of the correction operator $\mathcal{C}$ and the localized approximation space $\VHkmsloc$ are constructed in Section \ref{Sec: loc mult app}, a crucial exponential decay estimate for the localization error for dimension $d = 3$ is given in Section \ref{Sec: exp decay 3d}, and the error analysis of the localized multiscale method is given in Section \ref{sec-proof-main-thm}. 

In Section \ref{Sec: num exp}, we provide various numerical experiments that illustrate the theoretical results.

Finally, the appendix includes a proof of the well-known inf-sup stability in classical Raviart--Thomas spaces (Appendix \ref{appendix:A}), and a proof of the exponential decay estimate for the localization error for dimension $d = 2$ (Appendix \ref{appendix:section:decay-2D}).

\section{Preliminaries}\label{Sec: prelim}

\subsection{Notations}\label{Sec: notation}
We consider a bounded Lipschitz domain $\Omega \subset \mathbb{R}^d$ in dimension $d \in \{ 2,3\}$ with a piecewise polygonal boundary $\partial \Omega$. 
For any subdomain $\omega \subset \Omega$ with a Lipschitz boundary, we let $\mathbf{n}_{\omega}$ denote the outward unit normal vector on $\partial \omega$.

In the following, we use standard notation for Lebesgue and Sobolev spaces. In particular, for scalar functions $u$ and $v$, the $L^2$-inner product over $\omega$ is defined as $(u, v)_{L^2(\omega)}:=\int_\omega u v \dmu$. For $d$-dimensional vector-valued functions $\mathbf{u}$ and $\mathbf{v}$, we use the same notation for the $L^2$-inner product, but the definition formally changes to $(\mathbf{u}, \mathbf{v})_{L^2(\omega)}:=$ $\int_\omega \mathbf{u} \cdot \mathbf{v} \hspace{1pt}\dmu$. Similarly, the same notation is used for $L^2$-norms of scalar- and vector-valued functions, as the context makes the distinction clear. Vector-valued quantities, however, are denoted using boldface symbols. 

We denote the space of zero-mean $L^2$-functions by 
\begin{align*}
    L_0^2(\Omega):=\left\{ q \in L^2(\Omega): \int_{\Omega} q\, \dmu=0\right\}.
\end{align*}
Further, we introduce the following spaces of functions with a weak divergence on subdomains $\omega\subset \Omega$:
\begin{align*}
\Hdivomega{\omega}&:=\left\{\boldsymbol{u} \in L^2(\omega,\mathbb{R}^d): \div \boldsymbol{u} \in L^2(\omega)\right\}, \quad 
\Hzerodivomega{\omega}:=\left\{\boldsymbol{u} \in \Hdivomega{\omega}: \left.\boldsymbol{u} \cdot \mathbf{n}_\omega\right\rvert_{\partial\omega}=0\right\}, \\
\Hdivzeroomega{\omega}&:=\{\boldsymbol{u} \in \Hdivomega{\omega}: \div \boldsymbol{u}=0\}, 
\hspace{26pt}
 \Hzerodivzeroomega{\omega}:= \Hdivzeroomega{\omega} \cap \Hzerodivomega{\omega}. 
\end{align*}
The above spaces are equipped with the inner product $(\cdot,\cdot)_{\Hdivomega{\omega}}$ and the norm $\|\cdot\|_{\Hdivomega{\omega}}$ given by
$$
(\boldsymbol{u}, \boldsymbol{v})_{\Hdivomega{\omega}}:=(\boldsymbol{u}, \boldsymbol{v})_{L^2(\omega)}+(\div \boldsymbol{u}, \div \boldsymbol{v})_{L^2(\omega)},\qquad \|\boldsymbol{u}\|_{\Hdivomega{\omega}} := \sqrt{(\boldsymbol{u}, \boldsymbol{u})_{\Hdivomega{\omega}}}.
$$ 
The continuous dual of a Banach space $X$ is denoted by $X^{\prime}$.

Throughout this paper, the notation $a \lesssim b$ indicates that $a \leq C \, b$, where $C>0$ is a generic constant that can depend on $d$, $\Omega$, the lower and upper spectral bounds of $\mathbf{A}$, and the regularity constants of the meshes, but does not depend on the mesh size $H$ itself. In particular, $C$ does not depend on the potentially rapid oscillations in $\mathbf{A}$ or its regularity.

\subsection{Continuous problem}\label{Sec: Cts prob}

We consider the Neumann problem \eqref{elliptic-equation-strong-form} rewritten in mixed form as follows:
\begin{align}\label{eqn-mixed-form}
\begin{split}
\mathbf{A}^{-1} \mathbf{u}-\nabla p & =0 \qquad \text { in } \Omega, \\
\div \mathbf{u} & =-f \quad\,\text { in } \Omega, \\
\mathbf{u} \cdot \mathbf{n} & =0 \qquad \text { on } \partial \Omega.
\end{split}
\end{align}
We assume that $\mathbf{A} \in L^{\infty}(\Omega,\mathbb{R}^{d\times d})$ is a diffusion coefficient, possibly with rapid fine scale variations. Its value is an almost everywhere symmetric positive definite matrix, and we assume uniform ellipticity in the sense that there exist real numbers $\alpha$ and $\beta$ such that for almost every $x\in \Omega$ and for every $\boldsymbol{\xi} \in \mathbb{R}^d \backslash\{0\}$ it holds
$$
0<\alpha \leq \frac{\left(\mathbf{A}(x)^{-1} \boldsymbol{\xi} \right) \cdot \boldsymbol{\xi} }{\boldsymbol{\xi} \cdot \boldsymbol{\xi}} \leq \beta<\infty.
$$
Further, we suppose that $f \in L^2(\Omega)$ satisfies the compatibility condition
\begin{equation}\label{comp cond}
\int_{\Omega} f \,\dmu=0,
\end{equation}
i.e., $f\in L^2_0(\Omega)$. Note that the compatibility condition \eqref{comp cond} is necessary for the existence of a function $\mathbf{u}\in \Hzerodiv$ satisfying $\div \mathbf{u} = -f$ since then, by the divergence theorem,
\begin{align*}
\int_{\Omega} f \,\dmu = -\int_{\Omega}\div \mathbf{u} \,\dmu = -\int_{\partial\Omega}  \mathbf{u} \cdot \mathbf{n}\,\mbox{d}s = 0,
\end{align*}
which corresponds to the conservation of mass.
In order to state the variational (mixed) formulation of \eqref{eqn-mixed-form}, we introduce the following two bilinear forms:
\begin{align*}
\begin{split}
    a: \Hzerodiv \times \Hzerodiv \rightarrow \mathbb{R},\qquad a(\mathbf{u}, \mathbf{v})&:=\left(\mathbf{A}^{-1} \mathbf{u}, \mathbf{v}\right)_{L^2(\Omega)},\\
    b: \Hzerodiv \times L^2_0(\Omega)\;\;\; \rightarrow \mathbb{R},\qquad \;b(\mathbf{v}, q)&:=(\div \mathbf{v}, q)_{L^2(\Omega)}.
    \end{split}
\end{align*}
With this, we seek the velocity $\mathbf{u} \in \Hzerodiv$ and the pressure $p \in L_0^2(\Omega)$ such that
\begin{align}\label{eqn-mixed-form-var}
\begin{split}
a(\mathbf{u}, \mathbf{v})+b(\mathbf{v}, p) & =0 \hspace{70pt} \mbox{for all }\; \mathbf{v} \in \Hzerodiv, \\
b(\mathbf{u}, q) & = -(f, q)_{\Ltwo} \hspace{23pt}\mbox{for all }\; q \in L^2_0(\Omega).
\end{split}
\end{align}
Note that due to the compatibility condition \eqref{comp cond} and the fact that $\div \mathbf{u}\in L^2_0(\Omega)$ when $\mathbf{u}\in \Hzerodiv$, we can equivalently replace the test function space $L^2_0(\Omega)$ in the second equation of \eqref{eqn-mixed-form-var} by $L^2(\Omega)$.

For our error analysis, we also require the energy norm induced by $a(\cdot,\cdot)$, which we denote by  
$$
\vertiii{\mathbf{v}}:=\|\mathbf{A}^{-1/2} \mathbf{v}\|_{L^2(\Omega)}= \sqrt{a(\mathbf{v}, \mathbf{v})},\qquad \mathbf{v}\in \Hzerodiv,
$$ 
as well as its local version $\vertiii{\mathbf{v}}_{\omega} := \|A^{-1/2}\mathbf{v}\|_{L^2(\omega)}$ for any subdomain $\omega\subset \Omega$. We conclude this subsection with a classical well-posedness result, which guarantees the existence of a unique solution $(\mathbf{u}, p) \in \Hzerodiv\times L_0^2(\Omega)$ to the mixed problem \eqref{eqn-mixed-form-var}, as well as the well-posedness of all the discrete problems that follow later in this paper. In the latter case, the lemma below is applied with finite-dimensional subspaces $\mathcal{V} \subset \Hzerodiv$ and $\mathcal{Q} \subset L_0^2(\Omega)$. A proof of the following lemma can be found, e.g., in \cite[Theorem 4.2.3]{boffi2013mixed}.

\begin{lemma}[Well-posedness of the mixed formulation]\label{lem-inf-sup}
Let $\mathcal{V} \subset \Hzerodiv$ and $\mathcal{Q} \subset L^2_0(\Omega)$ denote respective closed subspaces. Introducing $\mathcal{V}_{\operatorname{div}\hspace{-1pt}0}:=\{\mathbf{v} \in \mathcal{V}: b(\mathbf{v}, q)=0 \;\;\mathrm{for}\; \mathrm{all}\;  q \in \mathcal{Q}\},$
suppose that
\begin{itemize}
    \item[$\mathrm{(A1)}$] $a(\cdot,\cdot)$ is coercive on $\mathcal{V}_{\operatorname{div}\hspace{-1pt}0}$ with constant $\tilde{\alpha}>0$, i.e., $a(\mathbf{v}, \mathbf{v}) \geq \tilde{\alpha}\|\mathbf{v}\|_{\Hdiv}^2$ for all $\mathbf{v} \in \mathcal{V}_{\operatorname{div}\hspace{-1pt}0}$,
    \item[$\mathrm{(A2)}$] $a(\cdot,\cdot)$ is bounded with constant $\tilde{\beta}>0$, i.e., $|a(\mathbf{v}, \mathbf{w})| \leq \tilde{\beta}\|\mathbf{v}\|_{\Hdiv}\|\mathbf{w}\|_{\Hdiv}$ for all $\mathbf{v}, \mathbf{w} \in \mathcal{V}$,
    \item[$\mathrm{(A3)}$] $b(\cdot,\cdot)$ is bounded, i.e., $\lvert b(\mathbf{v}, q)\rvert \lesssim \|\mathbf{v}\|_{\Hdiv}\|q\|_{\Ltwo}$ for all $\mathbf{v} \in \mathcal{V}$ and $q\in\mathcal{Q}$, and
    \item[$\mathrm{(A4)}$] $b(\cdot,\cdot)$ is inf-sup stable with constant $\tilde{\gamma}>0$, i.e., $\inf\limits _{q \in \mathcal{Q}\backslash\{0\}} \sup\limits_{\mathbf{v} \in \mathcal{V}\backslash\{0\}} \frac{b(\mathbf{v}, q)}{\|\mathbf{v}\|_{\Hdiv}\|q\|_{\Ltwo}} \geq \tilde{\gamma}$.
\end{itemize}
Then, there exists a unique $(\mathbf{u}_{\ast}, p_{\ast}) \in \mathcal{V} \times \mathcal{Q}$ such that 
\begin{align*}
a(\mathbf{u}_{\ast}, \mathbf{v})+b(\mathbf{v}, p_{\ast}) & =0 \hspace{70pt} \mathrm{for}\; \mathrm{all}\;\; \mathbf{v} \in \mathcal{V}, \\
b(\mathbf{u}_{\ast}, q) & = -(f, q)_{\Ltwo} \hspace{23pt}\mathrm{for}\; \mathrm{all}\;\; q \in \mathcal{Q}, 
\end{align*}
and we have the stability bounds $\|\mathbf{u}_{\ast} \|_{\Hdiv} \leq \frac{2}{ \tilde{\gamma}}\sqrt{\frac{\tilde \beta}{\tilde \alpha}}\,\|f\|_{\Ltwo}$ and $\|p_{\ast} \|_{\Ltwo} \leq \frac{\tilde{\beta}}{\tilde{\gamma}^2}\|f\|_{\Ltwo}$.
\end{lemma}
For the choice $(\mathcal{V},\mathcal{Q}) = (\Hzerodiv, L^2_0(\Omega))$, it is quickly checked that (A1)--(A3) hold with $\tilde \alpha = \alpha$ and $\tilde \beta = \beta$. To see (A4), note that for any $q\in L^2_0(\Omega)\backslash\{0\}$ there exists a unique $\varphi_q\in H^1(\Omega)\cap L^2_0(\Omega)$ such that $(\nabla \varphi_q,\nabla \psi)_{L^2(\Omega)} = -(q,\psi)_{L^2(\Omega)}$ for all $\psi\in H^1(\Omega)$. Observing that $\mathbf{v}_q :=\nabla \varphi_q\in \Hzerodiv \backslash\{\mathbf{0}\}$ with $\,\div \mathbf{v}_q = q$, we find that
\begin{align*}
    \sup_{\mathbf{v} \in \Hzerodiv\backslash\{\mathbf{0}\}} \frac{b(\mathbf{v}, q)}{\|\mathbf{v}\|_{\Hdiv}} \geq  \frac{b(\mathbf{v}_q, q)}{\|\mathbf{v}_q\|_{\Hdiv}} = \frac{\|q\|_{L^2(\Omega)}^2}{\|\mathbf{v}_q\|_{\Hdiv}}\geq \left(1+C_{\Omega}^2\right)^{-\frac{1}{2}}\|q\|_{\Ltwo}
\end{align*}
for any $q\in L^2_0(\Omega)\backslash\{0\}$, where we have used the fact that $\|\div \mathbf{v}_q\|_{L^2(\Omega)} = \|q\|_{L^2(\Omega)}$ and the bound $\|\mathbf{v}_q\|_{L^2(\Omega)} = - \frac{(q,\varphi_q)_{L^2(\Omega)}}{\|\mathbf{v}_q\|_{L^2(\Omega)}} \leq C_{\Omega}\|q\|_{L^2(\Omega)}$ with $C_{\Omega}> 0$ denoting the optimal constant for the Poincar\'{e}--Wirtinger inequality
\begin{align}\label{P-Wirtinger}
    \|\varphi\|_{L^2(\Omega)}\leq C_{\Omega} \|\nabla \varphi\|_{L^2(\Omega)}\qquad \mathrm{for}\; \mathrm{all}\;\; \varphi\in H^1(\Omega)\cap L^2_0(\Omega).
\end{align}
We summarize these observations and their consequences in a short remark.
\begin{remark}
For $(\mathcal{V},\mathcal{Q}) = (\Hzerodiv, L^2_0(\Omega))$, we have that the assumptions $\mathrm{(A1)}$--$\mathrm{(A4)}$ of Lemma \ref{lem-inf-sup} are satisfied with $(\tilde \alpha,\tilde \beta,\tilde \gamma) = (\alpha,\beta,(1+C_{\Omega}^{2})^{-\frac{1}{2}})$. Hence, by Lemma \ref{lem-inf-sup}, there exists a unique solution $(\mathbf{u}, p) \in \Hzerodiv \times L^2_0(\Omega)$ to \eqref{eqn-mixed-form-var}, and we have the stability bounds
\begin{align}\label{u stabbd}
    \|\mathbf{u}\|_{\Hdiv} \leq 2\sqrt{1+C_{\Omega}^2}\sqrt{\frac{\beta}{\alpha}} \,\|f\|_{L^2(\Omega)},\qquad \|p\|_{L^2(\Omega)}\leq (1+C_{\Omega}^2)\beta\|f\|_{L^2(\Omega)}.
\end{align}
In particular, the stability constants are independent of the variations/regularity of $\mathbf{A}$, which is crucial for multiscale problems.
\end{remark}

It is also quickly seen that for the unique solution $(\mathbf{u}, p) \in \Hzerodiv \times L^2_0(\Omega)$ to the mixed formulation \eqref{eqn-mixed-form-var}, we even have that
\begin{align*}
    p \in H^1(\Omega),\qquad \nabla p=\mathbf{A}^{-1} \mathbf{u}.
\end{align*}
Indeed, this follows from the existence and uniqueness of a solution $p \in H^1(\Omega) \cap L^2_0(\Omega)$ to the Neumann problem $(\mathbf{A} \nabla p ,\nabla v )_{\Ltwo} = (f,v)_{\Ltwo}$ for all $v\in H^1(\Omega)$. Setting $\mathbf{u}=\mathbf{A} \nabla p$, we see that $\mathbf{u}$ has a (weak) divergence in $L^2(\Omega)$, namely \,$\div \mathbf{u} = -f$, which in turn guarantees the existence of a normal trace $\mathbf{u} \cdot \mathbf{n}\in L^2(\partial\Omega)$. The normal trace vanishes since $f$ has zero mean. Hence, we see that $(\mathbf{u},p) \in \Hzerodiv \times (H^1(\Omega) \cap L^2_0(\Omega))$ solves \eqref{eqn-mixed-form-var}, and uniqueness concludes the argument. In particular, we have that
\begin{align}
\label{H1-estimate-pressure}
\| \nabla p \|_{\Ltwo} \,\,=\,\, \| \mathbf{A}^{-1} \mathbf{u} \|_{\Ltwo} \,\,\le\,\, \beta\, \| \mathbf{u} \|_{\Ltwo}  \,\,\leq\,\, C \| f\|_{\Ltwo}
\end{align}
for some constant $C=C(\alpha,\beta,C_{\Omega})>0$ independent of the multiscale variations of $\mathbf{A}$.

\subsection{Discretization with the Raviart--Thomas element}\label{Sec: RT}

The basis for our multiscale discretization is an underlying coarse discretization based on finite elements of Raviart--Thomas-type. For that, we consider a conforming simplicial mesh over $\Omega$, denoted by $\left\{\mathcal{T}_H\right\}$, where $H:=\underset{{T \in \mathcal{T}_H}}{\max} H_T$ and $H_T := \operatorname{diam}(T)$. 
Note that we think of $\mathcal{T}_H$ as a coarse mesh that does not necessarily resolve the variations of the multiscale coefficient $\mathbf{A}$. Throughout, we utilize $T$ to represent elements of $\mathcal{T}_H$ and $E$ to denote edges (for $d=2$) or faces (for $d=3$) of the elements of $\mathcal{T}_H$, with $\mathbf{n}_E$ denoting the outward normal vectors on edges and faces, respectively (where we silently assign a fixed global orientation/sign to each $\mathbf{n}_E$). 
The corresponding set of all such $(d-1)$-dimensional subsimplices, i.e., all edges for $d=2$ or all faces for $d=3$, is denoted by $\mathcal{E}_H$.
We also make the following assumptions on the coarse mesh. 
\begin{assumption}[Assumptions on mesh]\label{assu-mesh}\hfill
\begin{enumerate}
\item The mesh $\mathcal{T}_H$ is quasi-uniform in the sense that there exists a generic constant $\cuniformity > 0$ such that $\underset{{T \in \mathcal{T}_{H}}}{\min} H_T \geq \cuniformity H$.
\item The mesh $\mathcal{T}_H$ is shape-regular. In particular, we write $\cshape:=\underset{{T \in \mathcal{T}_{H}}}{\max} \frac{H_T}{\operatorname{diam}(B_T)} > 1$ to denote the shape-regularity constant, where $B_T$ denotes the largest ball contained in $T \in \mathcal{T}_{H}$.
\item The elements of $\mathcal{T}_H$ are (closed) simplices and are such that any two distinct elements $T, T^{\prime} \in \mathcal{T}_H$ are either disjoint or share a common vertex, edge, or face.
\end{enumerate}
\end{assumption}

We denote the space of all polynomials of degree at most $k$ on a subdomain $\omega\subset \Omega$ by $\mathbb{P}_k(\omega)$, and the space of $d$-dimensional vector-valued polynomials by $[\mathbb{P}_k(\omega)]^d$. With this we introduce the $\Hdiv$-conforming Raviart--Thomas finite element of order $k\in \mathbb{N}_{0}$. For each element $T \in \mathcal{T}_H$, the space of Raviart--Thomas shape functions is defined as:
$$
\begin{aligned}
\mathcal{RT}_{\hspace{-1pt}k}(T)  := \left[\mathbb{P}_k(T)\right]^d+\boldsymbol{x} \, \mathbb{P}_k(T),
\end{aligned}
$$
where $\boldsymbol{x}=\left(x_1, \ldots, x_d\right)$ is the spatial coordinate vector. 
We define the piecewise Raviart--Thomas space as
$$
\begin{aligned}
\mathcal{RT}_{\hspace{-1pt}k}(\mathcal{T}_{H}) :=\left\{\mathbf{v} \in L^2(\Omega,\mathbb{R}^d): \left.\mathbf{v}\right|_T \in \mathcal{RT}_{\hspace{-1pt}k}(T)\text{ for all } T\in \mathcal{T}_{H}\right\}
\end{aligned}.
$$
Since functions in $\mathcal{RT}_{\hspace{-1pt}k}(\mathcal{T}_{H})$ do not necessarily have a weak divergence, we have to consider the intersection $\mathcal{RT}_{\hspace{-1pt}k}(\mathcal{T}_{H}) \cap \Hdiv$ to obtain an $\Hdiv$-conforming subspace. In fact, this space can be equivalently characterized as the set of functions in $\mathcal{RT}_{\hspace{-1pt}k}(\mathcal{T}_{H})$ that have a continuous normal trace across the interior edges ($d=2$) or faces ($d=3$) of the mesh. Accordingly, we define the Raviart--Thomas finite element space of order $k\in\mathbb{N}_0$ on $\mathcal{T}_H$ (and with a vanishing normal trace on $\partial \Omega$) as
$$
\begin{aligned}
V_H^k :=\mathcal{RT}_{\hspace{-1pt}k}(\mathcal{T}_{H}) \cap \,\Hzerodiv.
\end{aligned}
$$
For the construction of a projection operator $\Pi_H : \Hzerodiv \rightarrow V_H^k$, we will also require the restriction of $V_H^k$ to vertex patches. For this, let $\mathcal{N}_H$ denote the set of vertices of $\mathcal{T}_H$ (that is, the set of all corners of mesh elements). For each vertex $z \in \mathcal{N}_H$, we define 
\begin{align*}
\TauHz :=\{T \in \mathcal{T}_H: \, z \text{ is a vertex of } T\},\qquad \omega_{z}:=\underset{{T \in \TauHz }}{\bigcup} T
\end{align*}
to denote the restriction of $\mathcal{T}_H$ to the element neighborhood of $z$, and the corresponding vertex patch, respectively. We can now define the local mixed finite element space $V_H^k\left(\omega_{z}\right)$ by 
\begin{equation}
\label{def-VHk-loc}
V_H^k(\omega_{z})\,\,:=\,\, \mathcal{RT}_{\hspace{-1pt}k}\left(\TauHz\right) \,\cap\, \Hzerodivomega{\omega_z}.
\end{equation}
The above definition directly implies that any function $\mathbf{v}_{z} \in V_H^k\left(\omega_{z}\right)$ can be canonically extended by zero on $\Omega \setminus \omega_{z}$ such that the extension is an element of $V_H^k = \mathcal{RT}_{\hspace{-1pt}k}(\mathcal{T}_H) \cap \Hzerodiv$. Whenever evaluating a function $\mathbf{v}_{z} \in V_H^k\left(\omega_{z}\right)$ outside of $\omega_z$, we will silently assume that we consider its extension by zero.

In addition to the approximation space for the velocity $\mathbf{u}$, we also require an approximation space for the pressure $p$. For that, we define the space of $\mathcal{T}_{H}$-piecewise polynomials of degree at most $k$ by
$$Q_H^k :=\left\{q \in L^2(\Omega) : \left.q\right|_T\in \mathbb{P}_k(T)\text{ for all } T\in \mathcal{T}_{H} \right\}$$
and its localized version to $\omega_{z}$ by
\begin{align}
\label{def-QHk-loc}
Q_H^k(\omega_{z})\,:=\,\
\{ q \in L^2(\omega_z):\, q \vert_{T} \in \mathbb{P}_k(T) \text{ for all } T \in \TauHz \}.
\end{align} 
In the following, we write $\operatorname{P}_H: L^2(\Omega) \rightarrow Q_H^k$ to denote the corresponding $L^2$-orthogonal projection onto the piecewise polynomials of degree at most $k$. For the construction of the multiscale method and for the derivation of error estimates, we also require a suitable (stable) projection from $\Hzerodiv$ into the Raviart--Thomas space $V_H^k \subset \Hzerodiv$. Although the classical nodal Raviart--Thomas interpolant is not sufficient for our purposes as it lacks $\Hdiv$-stability, it will appear as an ingredient in the Ern--Gudi--Smears--Vohral\'ik projection operator \cite{ern2022equivalence} which has the desired feature. For that reason, we briefly recall the canonical Raviart--Thomas interpolant
$$
\operatorname{I}_H\,:\, \Hdiv \cap L^s(\Omega) \, \rightarrow \,
\mathcal{RT}_{\hspace{-1pt}k}\left(\mathcal{T}_H\right) \cap \Hdiv,
$$
for some fixed (but arbitrary) $s>2$ which is crucial for well-posedness of the interpolant, cf. \cite[Section 2.5]{boffi2013mixed}. Note that $\Hdiv \cap L^s(\Omega) \subsetneq \Hdiv$. For any $\mathbf{v} \in \Hdiv \cap L^s(\Omega)$, the interpolant $\operatorname{I}_H \mathbf{v}$ is uniquely defined by the following element-wise conditions for each $T \in \mathcal{T}_H$:
\begin{equation}\label{def-RT-inte}
\begin{aligned}
\left(\left.\left(\operatorname{I}_H \mathbf{v} - \mathbf{v}\right)\right|_T \cdot \mathbf{n}_T, q_E \right)_{L^2(E)}  &= 0  \qquad  
\mbox{for all }\;
q_E \in \mathbb{P}_k(E),
\,\,
E \in \mathcal{E}_H, \, E\subset \partial T,\\
\left( \operatorname{I}_H\hspace{-1pt}\mathbf{v} - \mathbf{v} , \mathbf{q}_T \right)_{L^2(T)}  &= 0  \qquad  
\mbox{for all }\;
\mathbf{q}_T \in [\mathbb{P}_{k-1}(T)]^d.
\end{aligned}
\end{equation}
Here, $\left.\mathbf{v}\right|_T \cdot \mathbf{n}_T$ denotes the normal trace of $\mathbf{v}$ on $\partial T$. The second condition in \eqref{def-RT-inte} is only needed for $k\ge1$. The first condition states that the normal trace of $\operatorname{I}_H \mathbf{v}$ on $E$ coincides with the $L^2$-orthogonal projection of $\mathbf{v} \cdot \mathbf{n}$ on $\mathbb{P}_k(E)$. In particular, despite being a local construction, this ensures that $\operatorname{I}_H \mathbf{v}$ has a continuous normal trace across element interfaces. Together with the trivial observation that $\operatorname{I}_H \mathbf{v}$ has a piecewise divergence, we conclude that the image of $\operatorname{I}_H$ is in fact a subset of $
\Hdiv$. Also note that if $\mathbf{v}$ has a vanishing trace on $\partial \Omega$, so does $\operatorname{I}_H \mathbf{v}$.
 
One of the most important properties of interpolation operators defined on the space $\Hdiv$ (or $\Hcurl$) is the so-called \quotes{commuting diagram property} that ensures compatibility between differential operators and interpolation operators; cf. \cite{arnold2006finite, arnold2010finite}. This property is typically crucial for the convergence analysis of mixed finite element methods. In our setting, the relevant commuting identity is
$$
\div\left(\operatorname{I}_H \mathbf{v}\right)=\operatorname{P}_{\hspace{-1pt}H}\left(\div \mathbf{v}\right) \qquad \mbox{for all }\; \mathbf{v} \in \Hdiv \cap L^s(\Omega),
$$
i.e., we have an explicit relation between the divergence of the interpolation and the $L^2$-projection of the divergence.

The pair of spaces $(V_H^k,Q_H^k)$ for velocity and pressure is known to yield inf-sup-stability for $b(\cdot,\cdot)$, i.e., there is a constant $\rho>0$ that does not depend on the mesh size $H$, such that 
\begin{align}
\label{inf-sup-classical-spaces}
\inf_{q_H \in Q_H^k \cap L^2_0(\Omega)\backslash\{0\}} \sup _{\mathbf{v}_H \in V_H^k\backslash\{\mathbf{0}\} } \frac{b(\mathbf{v}_H, q_H)}{\|q_H\|_{\Ltwo}\|\mathbf{v}_H\|_{\Hdiv}} \geq \rho.
\end{align}
The result is classical (cf. \cite[Section 7.1.2]{boffi2013mixed} for the case of a boundary piece with Dirichlet values), but we could not find a good reference for the pure Neumann case, which is why we added a short proof for completeness in the appendix.

\subsection{Stable quasi-interpolation}\label{Sec: interp}

The construction of the multiscale method as well as its error analysis require a computable, local, $\Hdiv$-stable projection 
\begin{align*}
    \pi_H: \Hzerodiv \rightarrow V_H^k =\mathcal{RT}_{\hspace{-1pt}k}(\mathcal{T}_{H}) \cap \,\Hzerodiv
\end{align*}
that commutes with the local $L^2$-projection on $Q_H^k$ the same way as the canonical Raviart--Thomas interpolant. A suitable construction is provided in \cite{ern2022equivalence}, which is also the construction that we use later in our numerical experiments and which is summarized in Section \ref{sec-example-stable-int} below. However, since our analysis does not depend on that specific choice but rather on a set of axiomatic properties, it is sufficient to work with an abstract set of assumptions on the projection $\pi_H$. To be precise, we assume that the following properties are satisfied:

\begin{assumption}[Assumptions on the quasi-interpolation operator]\label{assu-stab-inte}\hfill
\begin{enumerate}
\item[(i)] Minimal regularity: The mapping $\pi_H : \Hzerodiv \rightarrow \mathcal{RT}_{\hspace{-1pt}k}(\mathcal{T}_{H}) \cap \,\Hzerodiv$ is well-defined.
\item[(ii)] Stability and locality: There exists a constant $C>0$, depending only on the spatial dimension $d$ and the shape-regularity parameter $\cshape$ of the mesh $\mathcal{T}_{H}$, such that for any $\mathbf{v} \in \Hzerodiv$ it holds
\begin{equation}\label{h-div-inte-stab}
\|\pi_H\mathbf{v}\|_{L^2(T)}^2 + H^2\|\div \pi_H\mathbf{v}\|_{L^2(T)}^2 \leq C \left(\|\mathbf{v}\|_{L^2(N^1(T))}^2 + H^2\|\div \mathbf{v}\|_{L^2(N^1(T))}^2\right)
\end{equation}
for all \( T \in \mathcal{T}_H \), where $N^1(T):= \bigcup \hspace{1pt}\{ K \in \mathcal{T}_H  \, : \, K \cap T \not= \emptyset \}$ denotes the element patch (consisting of $T$ and all its neighbors). In particular, the constant $C$ is independent of the mesh size $H$.
\item[(iii)] Projection property: The mapping
\( \pi_H : \Hzerodiv \rightarrow \mathcal{RT}_{\hspace{-1pt}k}(\mathcal{T}_{H}) \cap \,\Hzerodiv \) is a projection, that is, $\pi_H$ is surjective and fulfills $\pi_H \circ \pi_H = \pi_H$.
\item[(iv)] Commuting property: It holds $$
\div (\pi_H \mathbf{v})=\mathrm{P}_H (\div \mathbf{v}) \qquad \mathrm{for}\; \mathrm{all}\;\; \mathbf{v} \in \Hzerodiv,
$$
where we recall that $\mathrm{P}_H : L^2(\Omega) \rightarrow Q_H^k$ denotes the $L^2$-orthogonal projection onto $Q_H^k$.
\end{enumerate}
\end{assumption}

\subsubsection{Construction of a stable quasi-interpolation operator}\label{sec-example-stable-int}

In this subsection, we briefly describe the construction of the stable locally-defined quasi-interpolation operator $\Pi_H$ from $\Hzerodiv$ to $V_H^k= \mathcal{RT}_{\hspace{-1pt}k}(\mathcal{T}_H) \cap \Hzerodiv$ from \cite{ern2022equivalence}. Before we start, note that the element-wise Raviart--Thomas interpolant given by \eqref{def-RT-inte} is still well-defined for piecewise smooth functions $\mathbf{v} \in C^1(\mathcal{T}_H) := \{ \mathbf{v} \in L^2(\Omega)\, : \, \mathbf{v}\vert_{T} \in C^1(T,\mathbb{R}^d) \mbox{ for all } T \in \mathcal{T}_H\}$. However, in this case, the interpolation is not necessarily $\Hdiv$-conforming. To distinguish it from $\operatorname{I}_H$, we shall denote it by 
$$\operatorname{I}_H^{\operatorname{nc}} : C^1(\mathcal{T}_H) \rightarrow \mathcal{RT}_{\hspace{-1pt}k}(\mathcal{T}_{H}).$$
The interpolation $\operatorname{I}_H^{\operatorname{nc}} \mathbf{v}$ still admits a $\mathcal{T}_H$-piecewise divergence.

\begin{definition}\label{def-quas-inte}
Let $\mathbf{v} \in \Hzerodiv$ be arbitrary. We start with defining the element-wise $L^2$-best approximations $\boldsymbol{\tau}_H(\mathbf{v}) \in \mathcal{RT}_{\hspace{-1pt}k}(\mathcal{T}_{H})$ under the commuting property constraint by 
\begin{equation}\label{quas-inte-step1}
\left.\boldsymbol{\tau}_H(\mathbf{v})\right|_T\,\,:=\,\,\mathbf{v}_T^*\,\,:=\,\,\underset{\substack{\mathbf{v}_T \in \mathcal{RT}_k(T) \\ \div \mathbf{v}_T=\operatorname{P}_H(\div \mathbf{v})}}{\arg \min }\left\|\mathbf{v}-\mathbf{v}_T\right\|_{L^2(T)}\qquad \mathrm{for}\; \mathrm{all}\;\; T \in \mathcal{T}_H.
\end{equation}
Since $\boldsymbol{\tau}_H(\mathbf{v})$ is still non-conforming, we first localize it to vertex patches $\omega_z$ (for $z \in \mathcal{N}_H$) by forming $\psi_{z}\hspace{1pt} \boldsymbol{\tau}_H(\mathbf{v})$. Here, $\psi_{z}\in  \{ v_H \in C^0(\Omega)  :  \left.v_H\right\rvert_{T} \in \mathbb{P}_1(T) \mbox{ for all } T \in \mathcal{T}_H\}$ with $\mbox{\normalfont supp}\hspace{1pt}\psi_z =\omega_z$ denote the piecewise linear Lagrange nodal basis functions uniquely defined via $\psi_z(z^{\prime})=\delta_{zz^{\prime}}$ for $z,z^{\prime} \in \mathcal{N}_H$. This is an admissible localization since the $\psi_z$ form a partition of unity, i.e., $\sum_{z\in \mathcal{N}_H} \psi_z \equiv 1$.

Afterwards, the localized approximations $\psi_{z}\hspace{1pt} \boldsymbol{\tau}_H(\mathbf{v})$ are smoothed on each patch in the following way. For each $z \in \mathcal{N}_H$, we define the (local) smoothed interpolation $\boldsymbol{\sigma}_{z}(\mathbf{v}) \in V_H^k\left(\omega_{z}\right)$ by
\begin{equation}\label{quas-inte-step2}
\boldsymbol{\sigma}_{z}(\mathbf{v}):=\underset{\substack{\mathbf{v}_{z} \in V_H^k\left(\omega_{z}\right) \\ \div \mathbf{v}_{z}=\operatorname{P}_H\left(\psi_{z} \div \mathbf{v}+\nabla \psi_{z} \cdot \boldsymbol{\tau}_H(\mathbf{v})\right)}}{\arg \min }\left\|\mathbf{v}_{z}-\operatorname{I}_H^{\operatorname{nc}}\hspace{-1pt}\left(\, \psi_{z}\hspace{1pt} \boldsymbol{\tau}_H(\mathbf{v})\,\right)\right\|_{L^2(\omega_{z})}.
\end{equation}
Recall here the local $\Hdiv$-conforming Raviart--Thomas space $V_H^k\left(\omega_{z}\right)$ from \eqref{def-VHk-loc}. Since $\boldsymbol{\sigma}_{z}(\mathbf{v})$ admits an $\Hdiv$-conforming extension by zero to $\Omega$, we can define $\operatorname{\Pi}_H(\mathbf{v}) \in V_H^k$ by
\begin{equation*}
\operatorname{\Pi}_H(\mathbf{v}):=\sum_{z \in \mathcal{N}_H} \boldsymbol{\sigma}_{z}(\mathbf{v}).
\end{equation*}
Note that the vanishing normal trace of $\operatorname{\Pi}_H(\mathbf{v})$ on $\partial \Omega$ is already incorporated in $\boldsymbol{\sigma}_{z}(\mathbf{v})$ through the definition of $V_H^k\left(\omega_{z}\right)$.
\end{definition}
The justification that the construction of $\operatorname{\Pi}_H(\mathbf{v})$ is well-defined is provided in \cite[Section 4.1]{ern2022equivalence}. Furthermore, the quasi-interpolation operator satisfies the following stability properties: 
\begin{lemma} 
The interpolation operator $\operatorname{\Pi}_H : \Hzerodiv \rightarrow V_H^k$ from Definition \ref{def-quas-inte} is well-defined, it is a projection (i.e., $\Pi_H = \Pi_H \circ \Pi_H$ and $\mbox{\normalfont im}\hspace{1pt}(\Pi_H)=V_H^k$), and it has the commuting property 
$$
\begin{aligned}
\div \operatorname{\Pi}_H(\mathbf{v}) & =\operatorname{P}_H(\div \mathbf{v}) & & \quad\mathrm{for}\; \mathrm{all}\;\; \mathbf{v} \in \Hzerodiv
\end{aligned}
$$
with the $L^2$-projection $\operatorname{P}_H : L^2(\Omega) \rightarrow Q_H^k$. 
Furthermore, for any $\mathbf{v} \in \Hzerodiv$ and any $T \in \mathcal{T}_H$, we have the stability bound
\begin{equation*}
\left\|\operatorname{\Pi}_H(\mathbf{v})\right\|_{L^2(T)}^2+H^2\left\|\div \operatorname{\Pi}_H(\mathbf{v})\right\|_{L^2(T)}^2 \leq C \left(\|\mathbf{v}\|_{L^2(N^1(T))}^2+H^2\|\div \mathbf{v}\|_{L^2(N^1(T))}^2\right),
\end{equation*}
and the error estimate
\begin{eqnarray}
\label{loc-error-estimate-Ern-interpol}
\lefteqn{ \left\| \mathbf{v} - \operatorname{\Pi}_H(\mathbf{v}) \right\|_{L^2(T)}^2+H^2\left\| \div \mathbf{v} - \div \operatorname{\Pi}_H(\mathbf{v}) \right\|_{L^2(T)}^2 }\\
\nonumber&\leq& C  \underset{K\subset N^1(T)}{\sum_{K \in \mathcal{T}_H}}
\left(\min_{\mathbf{v}_K \in \mathcal{RT}_{\hspace{-1pt}k}(K)} \|\mathbf{v} - \mathbf{v}_K \|_{L^2(K)}^2+H^2\|\div \mathbf{v} - \operatorname{P}_H(\div \mathbf{v})\|_{L^2(K)}^2 \right).
\end{eqnarray}
In both bounds, $C>0$ is a generic constant depending only on the spatial dimension $d$, the shape-regularity parameter $\cshape$ of the mesh $\mathcal{T}_H$, and the polynomial degree $k$.
\end{lemma}
For a proof, we refer to \cite[Theorem 3.2]{ern2022equivalence}.

\begin{remark}[Local interpolation error bound for $\pi_H$]
\label{remark-error-estimates-piH}
Note that the error estimate \eqref{loc-error-estimate-Ern-interpol} for $\Pi_H$ implies that the same estimate must hold (up to slightly increased locality) for any $\pi_H$ satisfying Assumption \ref{assu-stab-inte}. Indeed, for any $\mathbf{v}\in \Hzerodiv$, we have by the projection properties that
\begin{eqnarray*}
\lefteqn{ \left\| \mathbf{v} - \operatorname{\pi}_H(\mathbf{v}) \right\|_{L^2(T)}^2 
+H^2\left\| \div \mathbf{v} - \div \operatorname{\pi}_H(\mathbf{v}) \right\|_{L^2(T)}^2 } \\
&=&
\left\| (\id - \operatorname{\pi}_H) (\mathbf{v} - \operatorname{\Pi}_H (\mathbf{v})) \right\|_{L^2(T)}^2 
+H^2\left\| \div\; (\id - \operatorname{\pi}_H) (\mathbf{v} - \operatorname{\Pi}_H (\mathbf{v})) \right\|_{L^2(T)}^2
\\
&\overset{\eqref{h-div-inte-stab}}{\lesssim}& \left\| \mathbf{v} - \operatorname{\Pi}_H (\mathbf{v}) \right\|_{L^2(N^1(T))}^2 + H^2 \| \div (\mathbf{v} - \operatorname{\Pi}_H (\mathbf{v})) \|_{L^2(N^1(T))}^2\\
&\overset{\eqref{loc-error-estimate-Ern-interpol}}\lesssim& 
\underset{K\subset N^2(T)}{\sum_{K \in \mathcal{T}_H}} \left(
\min_{\mathbf{v}_K \in \mathcal{RT}_{\hspace{-1pt}k}(K)} \|\mathbf{v} - \mathbf{v}_K \|_{L^2(K)}^2+H^2\|\div \mathbf{v} - \operatorname{P}_H(\div \mathbf{v})\|_{L^2(K)}^2\right),
\end{eqnarray*}
where $N^2(T):=\bigcup_{K \subset N^1(T)} N^1(K)$.
\end{remark}

\subsubsection{Implementation of a stable quasi-interpolation operator}\label{sec-example-int-impl}

Since the stable quasi-interpolation operator is not only required as an analytical tool, but also as an explicit component of the numerical method, we briefly want to discuss the practical implementation of the operator $\Pi_H$ from Definition \ref{def-quas-inte}.

To this end, we observe that the minimization problems \eqref{quas-inte-step1} and \eqref{quas-inte-step2} are equivalent to (low-dimensional) mixed finite element problems on the element $T$ and the vertex patch  $\omega_{z}$, respectively. Specifically, by introducing a Lagrange multiplier $p_T \in \mathbb{P}_k(T)$, problem \eqref{quas-inte-step1} is reformulated as finding the unique saddle point $\left(\mathbf{v}_T^*, p_T^*\right)\in V_H^k(T) \times Q_H^k(T) := \mathcal{RT}_{\hspace{-1pt}k}(T) \times \mathbb{P}_k(T)$ of the Lagrangian
\begin{equation*}
\mathcal{L}\left(\mathbf{v}_T, p_T\right):=\frac{1}{2}\left(\mathbf{v}-\mathbf{v}_T, \mathbf{v}-\mathbf{v}_T\right)_{L^2(T)}+\left(p_T, \div \mathbf{v}_T-\operatorname{P}_H(\div \mathbf{v})\right)_{L^2(T)},
\end{equation*}
which is equivalent to solving the following problem: Find $\mathbf{v}_T^* \in V_H^k(T)$ and $p_T^* \in Q_H^k(T)$ such that
\begin{align}\label{quas-inte-step1-reformulation}
\begin{split}
\left(\mathbf{v}_T^*, \boldsymbol{\varsigma}_H\right)_{L^2(T)} + \left(\div \boldsymbol{\varsigma}_H, p_T^*\right)_{L^2(T)}&=(\mathbf{v}, \boldsymbol{\varsigma}_H)_{L^2(T)} \qquad\qquad\;\;\;\mbox{for all }\; \boldsymbol{\varsigma}_H \in V_H^k(T), \\ 
\left(\div \mathbf{v}_T^*, q_H\right)_{L^2(T)} &= \left(\operatorname{P}_H(\div \mathbf{v}), q_H\right)_{L^2(T)}  \quad\mbox{for all }\; q_H \in Q_H^k(T).
\end{split}
\end{align}
We set $\left.\boldsymbol{\tau}_H(\mathbf{v})\right|_T\,:=\,\mathbf{v}_T^*$ for $T\in \mathcal{T}_H$.

For the second problem \eqref{quas-inte-step2}, recall the definitions of $V_H^k(\omega_{z})$ and $Q_H^k(\omega_{z})$ from \eqref{def-VHk-loc} and \eqref{def-QHk-loc} respectively. 
Similar as in the derivation of \eqref{quas-inte-step1-reformulation}, problem \eqref{quas-inte-step2} is now equivalent to solving the following problem: Find  $\mathbf{v}_{z}^* \in V_H^k(\omega_{z})$ and a Lagrange multiplier $p_z^* \in$ $Q_H^k(\omega_{z})$ such that
\begin{align}\label{quas-inte-step2-reformulation}
\begin{split}
\left(\mathbf{v}_{z}^*, \boldsymbol{\varsigma}_H\right)_{L^2(\omega_{z})} + \left(\div \boldsymbol{\varsigma}_H, p_z^*\right)_{L^2(\omega_{z})}=(\operatorname{I}_H^{\operatorname{nc}}\left(\psi_{z} \boldsymbol{\tau}_H(\mathbf{v}) \right), \boldsymbol{\varsigma}_H)_{L^2(\omega_{z})} & \quad\mbox{for all }\; \boldsymbol{\varsigma}_H \in V_H^k(\omega_{z}), \\ 
\left(\div \mathbf{v}_{z}^*, q_H\right)_{L^2(\omega_{z})} = \left(\operatorname{P}_H\left(\psi_z \div \mathbf{v}+\nabla \psi_z \cdot \boldsymbol{\tau}_H(\mathbf{v}) \right), q_H\right)_{L^2(\omega_{z})} & \quad\mbox{for all }\; q_H \in Q_H^k(\omega_{z}).
\end{split}
\end{align}
With $\boldsymbol{\sigma}_{z}(\mathbf{v})=\mathbf{v}_{z}^*$ and its extension by zero to $\Omega$, the quasi-interpolation $\operatorname{\Pi}_H(\mathbf{v})$ is obtained by $\operatorname{\Pi}_H(\mathbf{v})=\sum\limits_{z \in \mathcal{N}_H} \boldsymbol{\sigma}_{z}(\mathbf{v})$. Note that \eqref{quas-inte-step1-reformulation} and \eqref{quas-inte-step2-reformulation} are very small local problems that are inexpensive to solve.

\section{Ideal multiscale method}\label{Sec: Ideal}

We are now prepared to introduce the ideal multiscale method, i.e., we present the construction of a Localized Orthogonal Decomposition (LOD) space which is obtained by enriching the coarse Raviart--Thomas approximation space $V_H^k$ by fine-scale details from the kernel of the quasi-interpolation operator. The resulting low-dimensional LOD space is designed so that it captures the multiscale features of the true multiscale solution, ultimately serving as a substitute for the classical approximation space $V_H^k$.

Let $\pi_H$ be a stable interpolation operator satisfying Assumption \ref{assu-stab-inte}, e.g., the particular construction from Definition \ref{def-quas-inte}. Using $\pi_H$, we can decompose the space $\Hzerodiv$ into the finite-, low-dimensional coarse space $V_H^k=\operatorname{im}\left(\pi_H\right)$ and the detail space defined as the kernel of $\pi_H$, that is, 
\begin{align*}
    W:=\operatorname{ker}\left(\pi_H\right) \subset \Hzerodiv.
\end{align*}
This results in the direct sum decomposition 
\begin{align}
\label{initial-decomposition}
\Hzerodiv=V_H^k \oplus W.
\end{align}
Note that $W$ is closed, as it is the kernel of a continuous operator. The continuity of $\pi_H$ is guaranteed by the second property in Assumption \ref{assu-stab-inte}. 

Starting from the decomposition $\Hzerodiv=V_H^k \oplus W$, we construct a modified splitting, where $V_H^k$ is replaced by an ideal multiscale space that incorporates fine-scale features. The ideal multiscale space will be defined as $\left(\id -\operatorname{\mathcal{C}}\right)V_H^k$, where $\id$  denotes the identity operator and $\operatorname{\mathcal{C}}$ is a linear correction operator. The correction operator is constructed using information from the coefficient $\mathbf{A}$ and has a divergence-free range to ensure the resulting method is still mass-conserving.

\subsection{Correction operator and ideal multiscale space}\label{Sec: Corr op and ideal space}

Recalling $\Hzerodivzero = \Hdivzero \cap \Hzerodiv$ as the subspace of $\Hdiv$ consisting of divergence-free functions with a vanishing normal trace on $\partial \Omega$, we start by restricting the decomposition \eqref{initial-decomposition} to $\Hzerodivzero$. To this end, we define the spaces
\begin{align}
\label{def-Wdivzero}
\VHdivzero :=\operatorname{im}\left(\left.\operatorname{\pi}_H\right|_{\Hzerodivzero}\right), \qquad  \Wdivzero:=\operatorname{ker}\left(\left.\operatorname{\pi}_H\right|_{\Hzerodivzero}\right).
\end{align}
Note here that $\VHdivzero\subset \Hzerodivzero$ because of the commuting property (Assumption \ref{assu-stab-inte}) which implies for any $\mathbf{v} \in \Hzerodivzero$ that
$$
\div (\pi_H \mathbf{v})=\mathrm{P}_H (\div \mathbf{v}) = \mathrm{P}_H (0) = 0.
$$
Hence, $\pi_H \mathbf{v}$ is divergence-free if $\mathbf{v}$ is divergence-free. Consequently, since $\pi_H$ is a projection, we obtain that
\begin{align*}
\Hzerodivzero=\VHdivzero \oplus \Wdivzero.
\end{align*}
With the above decomposition, we can now introduce the ideal correction operator. First, we define local (element-wise) correctors, followed by the construction of a global corrector based on these local correctors.

\begin{definition}[Ideal correction operators] 
\label{def:ideal-corrector}
Let $\mathbf{v} \in \Hzerodiv$ be given. For each $T \in \mathcal{T}_H$, we define the restriction of $a(\cdot,\cdot)$ to $T$ by
$$
a_T(\mathbf{v}, \mathbf{w}):=(\mathbf{A}^{-1} \mathbf{v}, \mathbf{w})_{L^2(T)}
\qquad \mathrm{for}\;\; \mathbf{w} \in \Wdivzero.
$$
With this, the ideal element correction operator $\operatorname{\mathcal{C}}_{T}: \Hzerodiv \rightarrow \Wdivzero$ is defined by
\begin{equation}\label{ideal-correction-operator}
a\left(\operatorname{\mathcal{C}}_{T} \hspace{-1pt}\mathbf{v}, \mathbf{w}\right)=a_T\left(\mathbf{v}, \mathbf{w}\right)\qquad \mathrm{for}\; \mathrm{all}\;\; \mathbf{w} \in \Wdivzero. 
\end{equation}
The ideal global correction operator $\operatorname{\mathcal{C}}: \Hzerodiv \rightarrow \Wdivzero$ is then defined as the sum of the local contributions, i.e., $$\operatorname{\mathcal{C}}:=\underset{T \in \mathcal{T}_H}{\sum} \operatorname{\mathcal{C}}_{T}.$$ Note that the global corrector satisfies
\begin{equation}\label{ideal-global-correction-operator}
a\left(\operatorname{\mathcal{C}}\hspace{-1pt}\mathbf{v}, \mathbf{w}\right)=a\left(\mathbf{v}, \mathbf{w}\right)\qquad \mathrm{for}\; \mathrm{all}\;\; \mathbf{w} \in \Wdivzero. 
\end{equation}
\end{definition}
By the Lax--Milgram theorem and the coercivity and boundedness of $a(\cdot, \cdot)$ on the (closed) space $\Wdivzero$, the ideal correction operators are well-defined. 

Using the corrector $\operatorname{\mathcal{C}}$, we construct the LOD multiscale space $\VHkms$ by enriching the coarse space $V_H^k$ with \quotes{details} from $\Wdivzero$. Observe that functions $\mathbf{w}\in\Wdivzero$ cannot change the coarse scale behaviour due to $\pi_H\mathbf{w}=\mathbf{0}$ and they cannot change the divergence due to $\div \mathbf{w} =0$. Having both properties simultaneously is key to a reasonable multiscale space with suitable approximation properties. In fact, as we will see later, the property $\div \mathbf{w}=0$ is crucial to guarantee that the resulting multiscale approximation $\mathbf{u}_H^{\mathrm{ms}}$ preserves local mass conservation on the coarse mesh, i.e., the property $\mathrm{div}\,\mathbf{u}_H^{\mathrm{ms}} = -P_H(f)$. In other words, this condition ensures that the divergence of the numerical flux remains consistent with the given data, which is an essential qualitative feature of the exact solution.

With this, we define the LOD space by
\begin{equation}\label{VHkms defn} \VHkms := \left(\id - \operatorname{\mathcal{C}}\right)V_H^k, \end{equation}
where we recall $\id$ as the identity operator. Since it holds
\begin{equation*} 
\operatorname{\pi}_H\hspace{-2pt}\left((\id - \operatorname{\mathcal{C}})\mathbf{v}_H\right) = \mathbf{v}_H \qquad \mbox{for all } \;\mathbf{v}_H \in V_H^k,
\end{equation*}
we conclude that $(\id-\mathcal{C}): V_H^k \rightarrow \VHkms$ is bijective with inverse $\pi_H: \VHkms \rightarrow V_H^k$. Consequently, the ideal multiscale space $\VHkms$ has the same dimension as $V_H^k$. We summarize some important properties of $\VHkms$ in the following corollary. 

\begin{corollary}\label{corollary-decompositions}
Let $\operatorname{\mathcal{C}}: \Hzerodiv \rightarrow \Wdivzero$ denote the correction operator defined in Definition \ref{def:ideal-corrector}, and let $\VHkms$ be the LOD space defined in \eqref{VHkms defn}. It holds $\dim\VHkms=\dim V_H^k$ and we have the direct decomposition 
\begin{align}
\label{multiscale-decomposition-1}
\VHkms \oplus W \,\, =\,\, \Hzerodiv.
\end{align}
Furthermore, it holds
\begin{align*}
\label{multiscale-decomposition-2}
\VHkms \oplus \Wdivzero \,\,\subsetneq\,\, \Hzerodiv,
\end{align*}
where $\VHkms$ and $\Wdivzero$ are $a$-orthogonal in the sense that
\begin{equation}\label{a-orthogonality}
a\left(\mathbf{v}_H^{\mathrm{ms}}, \mathbf{w}\right)=0 \qquad \mathrm{for}\; \mathrm{all}\;\; \mathbf{v}_H^{\mathrm{ms}} \in \VHkms, \; \mathbf{w} \in \Wdivzero.
\end{equation}
\end{corollary}
\begin{proof}
We already verified that the dimensions of $\VHkms$ and $V_H^k$ are the same. To verify the decomposition \eqref{multiscale-decomposition-1}, note that an arbitrary function $\mathbf{v} \in \Hzerodiv$ can be written as
\begin{eqnarray*}  
  \mathbf{v} &=& \pi_H \mathbf{v} + (\mathbf{v} - \pi_H \mathbf{v} )  - \mathcal{C}(\pi_H \mathbf{v}) +  \mathcal{C}(\pi_H \mathbf{v}) \\
  &\overset{\pi_H\circ \mathcal{C} = 0}{=}& \underbrace{\pi_H \mathbf{v} - \mathcal{C}(\pi_H \mathbf{v})}_{\in \VHkms} \,\,+\,\, \underbrace{(\mathbf{v} + \mathcal{C}(\pi_H \mathbf{v}) - \pi_H ( \mathbf{v} + \mathcal{C}(\pi_H \mathbf{v}) )}_{\in W}.
\end{eqnarray*}
The $a$-orthogonality of $\VHkms$ and $\Wdivzero$ follows readily from \eqref{ideal-global-correction-operator}. Finally, since 
\begin{align*}
    \div(\VHkms \oplus \Wdivzero)= \div( \left(\id - \operatorname{\mathcal{C}}\right)V_H^k \oplus \Wdivzero) = \div V_H^k \subsetneq \div \Hzerodiv,
\end{align*}
the decomposition $\VHkms \oplus \Wdivzero$ must be a strict subset of $\Hzerodiv$.
\end{proof}
Using the ideal multiscale space $\VHkms$, the corresponding ideal multiscale approximation is given as follows.

\begin{definition}[Ideal multiscale approximation]
The ideal multiscale approximation is the solution $(\mathbf{u}_H^{\mathrm{ms}},p_H) \in \VHkms\times\left( Q_H^k \cap L^2_0(\Omega)\right)$ to the problem
\begin{equation}\label{ideal-mult-prob}
    \begin{aligned}
a\left(\mathbf{u}_{H}^{\mathrm{ms}}, \mathbf{v}_{H}^{\mathrm{ms}}\right)+b\left( \mathbf{v}_{H}^{\mathrm{ms}}, p_H\right) &=0 \hspace{65pt} \mathrm{for}\; \mathrm{all}\;\; \mathbf{v}_{H}^{\mathrm{ms}} \in \VHkms,\\ 
b\left(\mathbf{u}_{H}^{\mathrm{ms}}, q_H\right) &= -\left(f, q_H\right) \hspace{30pt} \mathrm{for}\; \mathrm{all}\;\; q_H \in Q_H^k.
\end{aligned}
\end{equation}
\end{definition}
 
The existence and uniqueness of a solution to the ideal multiscale problem \eqref{ideal-mult-prob} are established by the following lemma, which can be proved analogously to the corresponding result in \cite[Lemma 8]{hellman2016multiscale}. It is an application of Lemma \ref{lem-inf-sup} and follows from the inf-sup stability of $b(\cdot,\cdot)$ on the classical pair $V_H^k \times Q_H^k$ together with the $\Hdiv$-continuity of the corrector operator $\mathcal{C}$.

\begin{lemma}[Well-posedness of the ideal multiscale problem]\label{lemma:inf-sup-multiscale}
For the ideal problem \eqref{ideal-mult-prob}, we have the inf-sup stability
$$
\sup _{\mathbf{v}_{H}^{\mathrm{ms}} \in \VHkms \backslash\{\mathbf{0}\}} \frac{b(\mathbf{v}_{H}^{\mathrm{ms}}, q_H)}{\|q_H\|_{\Ltwo}\|\mathbf{v}_{H}^{\mathrm{ms}}\|_{\Hdiv}}\geq \rho\left(1+\alpha^{-1} \beta\right)^{-1}  \qquad\mathrm{for}\; \mathrm{all}\;\;q_H\in Q_H^k \cap L^2_0(\Omega)\backslash\{0\},
$$
where $\rho>0$ denotes the inf-sup constant from \eqref{inf-sup-classical-spaces}. In particular, there exists a unique solution $(\mathbf{u}_H^{\mathrm{ms}}, p_H) \in \VHkms\times \left( Q_H^k \cap L^2_0(\Omega)\right)$ to \eqref{ideal-mult-prob}.
\end{lemma}

\subsection{Error estimate for ideal multiscale problem}\label{Sec: error ideal}

In this subsection, we demonstrate that the flux solution of the ideal multiscale problem converges to the reference solution in the energy norm at least with a linear rate with respect to $H$, but possibly higher order depending on the regularity of the source term $f$. Notably, this convergence is independent of the variations in $\mathbf{A}$, indicating the absence of a pre-asymptotic regime caused by the multiscale features.

\begin{theorem}[Error estimate for the ideal multiscale problem]\label{thm-error-ideal-problem}
Let $(\mathbf{u}, p) \in \Hzerodiv \times L^2_0(\Omega)$ and $(\mathbf{u}_{H}^{\mathrm{ms}}, p_H) \in \VHkms \times \left(Q_H^k \cap L^2_0(\Omega)\right)$ denote the solutions to \eqref{eqn-mixed-form-var} and \eqref{ideal-mult-prob}, respectively. Introducing 
$$\VHkms(f) \,:=\,\{\mathbf{v} \in \VHkms \,:\, \div \mathbf{v}=-\operatorname{P}_H \hspace{-2pt} f \},$$
we have the error bounds
\begin{align}
\vertiii{\mathbf{u}-\mathbf{u}_{H}^{\mathrm{ms}}} &\;=\; \inf_{\mathbf{v}_f \in \VHkms(f)} \vertiii{\mathbf{u}-\mathbf{v}_f}
\qquad\lesssim\;
H\left\| f-\operatorname{P}_H \hspace{-2pt}f \right\|_{\Ltwo}, \label{main-estimate-ideal-method} \\
\|\div \,(\mathbf{u}-\mathbf{u}_{H}^{\mathrm{ms}})\hspace{1pt}\|_{\Ltwo}  &\;=\; \inf_{q_H \in Q_H^k}\|\div \mathbf{u}-q_H\|_{\Ltwo} 
\;=\quad\;\left\| f-\operatorname{P}_H \hspace{-2pt}f \right\|_{\Ltwo}, \label{sub-estimate-2-ideal-method}
\end{align}
and
\begin{eqnarray}
\label{sub-estimate-1-ideal-method}
\|p- p_H\|_{\Ltwo} 
&\lesssim& H \left\| f-\operatorname{P}_H \hspace{-2pt}f \right\|_{\Ltwo} \,+\,
\left\| p-\operatorname{P}_H \hspace{-1pt}p \right\|_{\Ltwo}.
\end{eqnarray}
Further, if $f\in Q_H^k$, then the method is exact for the velocity and it holds $\mathbf{u} =  (\id- \mathcal{C})\hspace{1pt}\pi_H(\,\mathbf{u}\,) = \mathbf{u}_{H}^{\mathrm{ms}}$.
\end{theorem}

Before presenting the proof of the theorem, let us briefly discuss the results. First of all, estimate \eqref{main-estimate-ideal-method} guarantees that $\mathbf{u}_{H}^{\mathrm{ms}}$ converges at least with linear rate in $L^2(\Omega)$ to the exact solution $\mathbf{u}$. In particular, this convergence is independent of the regularity of $\mathbf{u}$ and the variations of $\mathbf{A}$. If $f$ admits higher regularity, e.g., $f \in H^{k+1}(T)$ for all $T\in \mathcal{T}_H$, then the rate of convergence increases to $O(H^{k+2})$, again, independent of the regularity of $\mathbf{u}$ or $\mathbf{A}$. The bound \eqref{sub-estimate-2-ideal-method} ensures that the convergence of the error in the $\Hdiv$-norm only drops by one order and we can still expect the convergence rate $O(H^{k+1})$ for regular $f$. If $f\in Q_H^k$, then the ideal multiscale method is exact for the velocity, that is, $\mathbf{u}=\mathbf{u}_{H}^{\mathrm{ms}}$.

In contrast, the flavor of the bound \eqref{sub-estimate-1-ideal-method} on the approximation error of the pressure is different, as it requires additional regularity of the exact solution $p$ to obtain higher rates of convergence. In the minimal regularity setting, that is, $p\in H^1(\Omega)$, the estimate \eqref{sub-estimate-1-ideal-method} implies a linear convergence in the mesh size $H$. Since $\| \nabla p \|_{L^2(\Omega)} = \| \mathbf{A}^{-1} \mathbf{u} \|_{L^2(\Omega)} \lesssim \| f \|_{L^2(\Omega)} $, this linear convergence is independent of the variations of $\mathbf{A}$. Hence, for $p\in H^1(\Omega)$ we have $\|p- p_H\|_{L^2(\Omega)} \lesssim H \| f \|_{L^2(\Omega)} $. However, in a realistic multiscale setting, this estimate cannot be improved since higher derivatives of $p$ (if they exist) would scale with the variations of $\mathbf{A}$ in a negative way. For example, if $\mathbf{A}$ oscillates at a scale of order $\varepsilon \ll H$, then we would expect $\| p \|_{H^{1+m}(\Omega)} \lesssim \varepsilon^{-m}$ and there is formally nothing to gain for $m> 0$. Yet, this error can still appear small, if there is a smooth homogenized solution nearby, say some $p^0 \in H^{k+1}(\Omega)$ with $\| p - p^0 \|_{L^2(\Omega)} \lesssim \varepsilon \ll H$. In this case, the error can behave as $\| p -\operatorname{P}_H \hspace{-1pt}p \|_{\Ltwo} \lesssim \varepsilon + H^{k+1}$.

The proof of Theorem \ref{thm-error-ideal-problem} given below loosely adopts and extends the arguments from \cite[Lemma 9]{hellman2016multiscale}. 
\begin{proof}
\textbf{Step 1.} We start by proving \eqref{sub-estimate-2-ideal-method}.

First, we note that the second equation in \eqref{ideal-mult-prob} implies that
\begin{align*}
(\div \mathbf{u}_{H}^{\mathrm{ms}}, q_H )_{\Ltwo} &=(-f, q_H)_{\Ltwo} = ( -\operatorname{P}_H \hspace{-1pt}f , q_H)_{\Ltwo}\qquad \mbox{for all } \, q_H \in Q_H^k.
\end{align*}
Since $\div \mathbf{u}_{H}^{\mathrm{ms}} \in \div \VHkms \subset Q_H^k$ and $\operatorname{P}_H \hspace{-1pt}f \in Q_H^k$, we deduce that $\,\div \mathbf{u}_{H}^{\mathrm{ms}} = -\operatorname{P}_H \hspace{-2pt}f$. This yields
$$
\div \left(\mathbf{u}-\mathbf{u}_{H}^{\mathrm{ms}}\right) \,\,=\,\, \div \mathbf{u}- \div \mathbf{u}_{H}^{\mathrm{ms}} =  -f - (-P_Hf) \,\,=\,\, \operatorname{P}_H \hspace{-2pt}f-f ,
$$
which proves \eqref{sub-estimate-2-ideal-method} by the properties of the $L^2$-projection $\operatorname{P}_H:L^2(\Omega)\rightarrow Q_H^k$.

\textbf{Step 2.} Our next goal is to prove \eqref{main-estimate-ideal-method}. 

\textbf{Step 2a.} To this end, we begin by showing that
\begin{align}\label{Step2a}
    \vertiii{\mathbf{u}-\mathbf{u}_{H}^{\mathrm{ms}}} \leq \vertiii{\mathbf{u}-\mathbf{v}_f}\qquad \mbox{for all } \mathbf{v}_f \in \VHkms(f),
\end{align}
which then implies the first equality in \eqref{main-estimate-ideal-method} in view of the fact that $\mathbf{u}_{H}^{\mathrm{ms}} \in \VHkms(f)$.

Using again that $\,\div \mathbf{u}_{H}^{\mathrm{ms}} = -\operatorname{P}_H \hspace{-2pt}f$, we conclude for arbitrary  $\mathbf{v}_f \in \VHkms(f)$ that $\div (\mathbf{u}_{H}^{\mathrm{ms}}- \mathbf{v}_f)=0$, or equivalently,
\begin{align}
\label{error-in-kernel-B}
b(\mathbf{u}_{H}^{\mathrm{ms}}- \mathbf{v}_f, q) = 0 \qquad \mbox{for all } q\in L^2(\Omega). 
\end{align}
Observing the Galerkin orthogonality
\begin{align} \label{ideal-galerkin-orthogonality-1}
\begin{split}
a\left(\mathbf{u} - \mathbf{u}_{H}^{\mathrm{ms}}, \mathbf{v}\right)+b\left(\mathbf{v}, p - p_H\right) =0 & \qquad  \mbox{for all } \mathbf{v} \in \VHkms,\\ 
b\left(\mathbf{u} - \mathbf{u}_{H}^{\mathrm{ms}}, q_H\right) =0  & \qquad \mbox{for all } q_H \in Q_H^k,
\end{split}
\end{align}
we combine \eqref{error-in-kernel-B} with the first equation in \eqref{ideal-galerkin-orthogonality-1} for $\mathbf{v}= \mathbf{u}_{H}^{\mathrm{ms}}- \mathbf{v}_f$ to find that
\begin{align*}
a\left(\mathbf{u} - \mathbf{u}_{H}^{\mathrm{ms}}, \mathbf{u}_{H}^{\mathrm{ms}}- \mathbf{v}_f \right) \,\,=\,\, 0
\qquad \mbox{for all } \mathbf{v}_f \in \VHkms(f).
\end{align*}
Hence, for any $\mathbf{v}_f \in \VHkms(f)$, we have that
\begin{eqnarray*}
\vertiii{\mathbf{u}-\mathbf{u}_{H}^{\mathrm{ms}}}^2 &=&  a\left(\mathbf{u} - \mathbf{u}_{H}^{\mathrm{ms}}, \mathbf{u}-\mathbf{v}_f\right) - a\left(\mathbf{u} - \mathbf{u}_{H}^{\mathrm{ms}}, \mathbf{u}_{H}^{\mathrm{ms}}-\mathbf{v}_f\right) \\
&=&  a\left(\mathbf{u} - \mathbf{u}_{H}^{\mathrm{ms}}, \mathbf{u}-\mathbf{v}_f\right)\\
&\leq& \vertiii{\mathbf{u}-\mathbf{u}_{H}^{\mathrm{ms}}}\, \vertiii{\mathbf{u}-\mathbf{v}_f},
\end{eqnarray*}
and the claimed inequality \eqref{Step2a} follows. 

\textbf{Step 2b.} Introducing $(\,\mathbf{u}(\operatorname{P}_H \hspace{-1pt}f),p(\operatorname{P}_H \hspace{-1pt}f)\,)\in \Hzerodiv \times L^2_0(\Omega)$ to be the exact solution of problem \eqref{eqn-mixed-form-var} with the source function $f$ replaced by $\operatorname{P}_H\hspace{-2pt}f$, we want to verify that $\mathbf{u}(\operatorname{P}_H \hspace{-1pt}f)\in \VHkms(f)$. 

Clearly, it holds 
\begin{align}\label{divuPHf}
    \div \mathbf{u}(\operatorname{P}_H \hspace{-1pt}f) = -\operatorname{P}_H\hspace{-1pt} f
\end{align}
so that we only need to show that $\mathbf{u}(\operatorname{P}_H \hspace{-1pt}f)\in \VHkms$. For that, we define the auxiliary function
\begin{align*}
\mathbf{u}_{H}^{\mathrm{ms}}(\operatorname{P}_H \hspace{-1pt}f) := (\id- \mathcal{C})\hspace{1pt}\pi_H(\,\mathbf{u}(\operatorname{P}_H \hspace{-1pt}f)\,) \,\, \in \,\,\VHkms.
\end{align*}
Noting that $\div \circ \,\mathcal{C}=0$, and using the commuting property from Assumption \ref{assu-stab-inte}, we observe that
\begin{align}\label{divuHmsPHf}
\div \mathbf{u}_{H}^{\mathrm{ms}}(\operatorname{P}_H \hspace{-1pt}f) \, =\, \div \pi_H(\,\mathbf{u}(\operatorname{P}_H \hspace{-1pt}f)\,)
\,=\, \operatorname{P}_H(\div \mathbf{u}(\operatorname{P}_H \hspace{-1pt}f) ) \,=\, -(\operatorname{P}_H \circ \operatorname{P}_H)f\,=\, -\operatorname{P}_H \hspace{-2pt}f.
\end{align}
Since $\Hzerodiv = \VHkms \oplus W$, we can write  $\mathbf{u}(\operatorname{P}_H \hspace{-1pt}f)\in \Hzerodiv$ as
\begin{align*}
\mathbf{u}(\operatorname{P}_H \hspace{-1pt}f) = \mathbf{u}_{H}^{\mathrm{ms}}(\operatorname{P}_H \hspace{-1pt}f) + \mathbf{w}_f \qquad \mbox{for some } \mathbf{w}_f \in W.
\end{align*}
Noting that by \eqref{divuPHf} and \eqref{divuHmsPHf} it holds \,$\div \mathbf{w}_f =0$, we see that $\mathbf{w}_f = \mathbf{u}(\operatorname{P}_H \hspace{-1pt}f) - \mathbf{u}_{H}^{\mathrm{ms}}(\operatorname{P}_H \hspace{-1pt}f) \in \Wdivzero$ (recall the definition of $\Wdivzero$ from \eqref{def-Wdivzero}). Hence, using the definition of $\mathbf{u}(\operatorname{P}_H \hspace{-1pt}f)$ and the $a$-orthogonality of $\Wdivzero$ and $\VHkms$ (see Corollary \ref{corollary-decompositions}), we find that
\begin{eqnarray*}
a(\mathbf{w}_f,\mathbf{w}_f) \,\,\,=\,\,\, a(\mathbf{u}(\operatorname{P}_H \hspace{-1pt}f) - \mathbf{u}_{H}^{\mathrm{ms}}(\operatorname{P}_H \hspace{-1pt}f) ,\mathbf{w}_f )
&=& a(\mathbf{u}(\operatorname{P}_H \hspace{-1pt}f)  , \mathbf{w}_f )
\,\,\,=\,\,\, - b( \mathbf{w}_f  , p(\operatorname{P}_H \hspace{-1pt}f) )
\overset{\div \mathbf{w}_f = 0}{=} 0,
\end{eqnarray*}
i.e., $\mathbf{w}_f = \mathbf{u}(\operatorname{P}_H \hspace{-1pt}f) - \mathbf{u}_{H}^{\mathrm{ms}}(\operatorname{P}_H \hspace{-1pt}f) = \mathbf{0}$, and hence, $\mathbf{u}(\operatorname{P}_H \hspace{-1pt}f) = \mathbf{u}_{H}^{\mathrm{ms}}(\operatorname{P}_H \hspace{-1pt}f)\in \VHkms$. 

\textbf{Step 2c.} Combining the results from Steps 2a and 2b, we arrive at the bound
\begin{align}\label{umuHms bd1}
\vertiii{\mathbf{u}-\mathbf{u}_{H}^{\mathrm{ms}}} \le \vertiii{\mathbf{u}-\mathbf{u}(\operatorname{P}_H \hspace{-1pt}f) },
\end{align}
and it remains to show that 
\begin{align}\label{Step 2c}
    \vertiii{\mathbf{u}-\mathbf{u}(\operatorname{P}_H \hspace{-1pt}f) } \lesssim H \, \|  f - \operatorname{P}_H \hspace{-1pt}f \|_{\Ltwo}.
\end{align}
To this end, we observe that the pair $\left(\mathbf{u}-\mathbf{u}(\operatorname{P}_H \hspace{-1pt}f) ,p-p(\operatorname{P}_H \hspace{-1pt}f)\right) \in \Hzerodiv\times L^2_0(\Omega)$ is the unique solution to the mixed problem 
\begin{equation*}
    \begin{aligned}
a( \mathbf{u}-\mathbf{u}(\operatorname{P}_H \hspace{-1pt}f) , \mathbf{v} ) \,\,+\,\, b ( \mathbf{v}, p-p(\operatorname{P}_H \hspace{-1pt}f) ) &\,\,=\,\,0 \hspace{95pt} \mbox{for all } \, \mathbf{v} \in \Hzerodiv,\\ 
b( \mathbf{u}-\mathbf{u}(\operatorname{P}_H \hspace{-1pt}f) , q ) &\,\,=\,\, -(f - \operatorname{P}_H \hspace{-1pt}f, q )_{L^2(\Omega)} \hspace{15pt} \mbox{for all } \, q \in L^2(\Omega).
\end{aligned}
\end{equation*}
Consequently, we have with $\mathbf{v}=\mathbf{u}-\mathbf{u}(\operatorname{P}_H \hspace{-1pt}f)$ and $q=p-p(\operatorname{P}_H \hspace{-1pt}f)$ that
\begin{eqnarray*}
a( \, \mathbf{u}-\mathbf{u}(\operatorname{P}_H \hspace{-1pt}f),  \mathbf{u}-\mathbf{u}(\operatorname{P}_H \hspace{-1pt}f) \,)
&=& ( f - \operatorname{P}_H \hspace{-1pt}f , p-p(\operatorname{P}_H \hspace{-1pt}f) )_{L^2(\Omega)} \\
&=& ( f - \operatorname{P}_H \hspace{-1pt}f , \, p-p(\operatorname{P}_H \hspace{-1pt}f) \, - \, \operatorname{P}_H \hspace{-1pt}\left( p-p(\operatorname{P}_H \hspace{-1pt}f) \right) )_{L^2(\Omega)} \\
&\lesssim& H\,\|  f - \operatorname{P}_H \hspace{-1pt}f \|_{\Ltwo} \, 
  \| \nabla (p-p(\operatorname{P}_H \hspace{-1pt}f)) \|_{\Ltwo} \\
&\overset{\eqref{H1-estimate-pressure}}{=}&  H\, \|  f - \operatorname{P}_H \hspace{-1pt}f \|_{\Ltwo} \, 
 \| \mathbf{A}^{-1} (\mathbf{u}-\mathbf{u}(\operatorname{P}_H \hspace{-1pt}f)) \|_{\Ltwo}\\
 &\lesssim& H\,\|  f - \operatorname{P}_H \hspace{-1pt}f \|_{\Ltwo} \, 
  \vertiii{\mathbf{u}-\mathbf{u}(\operatorname{P}_H \hspace{-1pt}f) }.
\end{eqnarray*}
Dividing by $\vertiii{\mathbf{u}-\mathbf{u}(\operatorname{P}_H \hspace{-1pt}f) }$ yields \eqref{Step 2c}, which in combination with \eqref{umuHms bd1} concludes the proof of \eqref{main-estimate-ideal-method}.

\textbf{Step 3.} Finally, we prove the error bound \eqref{sub-estimate-1-ideal-method}.

To this end, we split the error as
\begin{align*}
\| p - p_H \|_{\Ltwo} \le \| p - \operatorname{P}_H\hspace{-1pt}p \|_{\Ltwo} + \| p_H - \operatorname{P}_H\hspace{-1pt}p \|_{\Ltwo}.
\end{align*}
To estimate the latter term on the right-hand side, we use the inf-sup stability in Lemma \ref{lemma:inf-sup-multiscale} to obtain that
\begin{align*}
    \| p_H - \operatorname{P}_H\hspace{-1pt}p \|_{\Ltwo} &\,\lesssim\, \sup_{\mathbf{v}_{H}^{\mathrm{ms}}\in V_H^{k,\mathrm{ms}}\backslash\{\mathbf{0}\}}\frac{\lvert b(\mathbf{v}_{H}^{\mathrm{ms}},p_H-\operatorname{P}_H\hspace{-1pt}p)\rvert}{\|\mathbf{v}_{H}^{\mathrm{ms}}\|_{\Hdiv}}\,\\ &\leq\, \sup_{\mathbf{v}_{H}^{\mathrm{ms}}\in V_H^{k,\mathrm{ms}}\backslash\{\mathbf{0}\}}\frac{\lvert a(\mathbf{u}-\mathbf{u}_{H}^{\mathrm{ms}},\mathbf{v}_{H}^{\mathrm{ms}})\rvert+\lvert b(\mathbf{v}_{H}^{\mathrm{ms}},p-\operatorname{P}_H\hspace{-1pt}p)\rvert}{\|\mathbf{v}_{H}^{\mathrm{ms}}\|_{\Hdiv}} \\
    &\,\lesssim\, \vertiii{\mathbf{u}-\mathbf{u}_{H}^{\mathrm{ms}}} + \|p -\operatorname{P}_H\hspace{-1pt}p\|_{\Ltwo},
\end{align*}
where we used that $\lvert a(\mathbf{w},\mathbf{z})\rvert \leq \vertiii{w}\vertiii{z} \lesssim \vertiii{w} \|z\|_{\Hdiv}$ for any $\mathbf{w},\mathbf{z}\in \Hzerodiv$. Finally, using \eqref{main-estimate-ideal-method} to bound the term $\vertiii{\mathbf{u}-\mathbf{u}_{H}^{\mathrm{ms}}}$ completes the proof. 
\end{proof}

\section{Localized multiscale method}\label{Sec: localized}

The optimal corrector problems \eqref{ideal-correction-operator} require nearly the same computational effort as the original multiscale problem \eqref{eqn-mixed-form} since they are formulated on the full domain $\Omega$ and are hence global problems. To address this issue, it is crucial to confine these problems to small regions while maintaining their high approximation quality. This strategy enables efficient, low-cost computations that can be executed entirely in parallel.

In the following, we will introduce the localized method as it can be used in practice. In particular, we state the localized orthogonal decomposition (LOD) in $\Hdiv$-spaces. In our setting, the LOD refers to a localized version of the $a$-orthogonal splitting $\VHkms \oplus \Wdivzero$ as stated in Corollary \ref{corollary-decompositions}. The localization is achieved by truncating the element corrector problems \eqref{ideal-correction-operator} to small neighborhoods of coarse elements $T \in \mathcal{T}_H$. As in the classical LOD for $H^1$-conforming spaces, the truncation will be justified by the exponential decay of the element correctors away from the coarse element $T$. Using this approach, we construct an approximate basis for the multiscale space $\VHkms$ that is spatially localized, relying on the localized correction operator as an analog to the ideal correction operator. This localization significantly reduces the computational effort required to assemble the multiscale space.

\begin{remark}[Contractible $\Omega$]
In order to ensure that all localized problems appearing in this section are well-defined, and that we have local regular decompositions of $\Hzerodivomega{\omega}$ for suitable patch neighborhoods $\omega$, we require that these patches are always simply-connected. To avoid complicated assumptions on the mesh geometry, we will simply assume in this section that, in addition to the previous assumptions, the domain $\Omega$ is contractible.
\end{remark}

For a proper definition of the localized correctors, we need to restrict the detail space $\Wdivzero$ to suitable local subdomains. These subdomains are formed by element patches $N^m(T)$. To be precise, for any coarse element $T \in \mathcal{T}_H$ (which we recall as closed sets), we define $N^m(T)$ for a given number of layers $m \in \mathbb{N}_0$ recursively as follows: 
$$
\begin{aligned}
N^0(T):=T, \qquad N^m(T):=\bigcup\left\{T^{\prime} \in \mathcal{T}_H: T^{\prime} \cap N^{m-1}(T) \neq \emptyset\right\}, \qquad m \in \mathbb{N}.
\end{aligned}
$$
For a given patch $N^m(T)$, the restriction of $W$ to $N^m(T)$ is given by
$$
W(N^m(T))\,\,:=\,\,\{\mathbf{w} \in W: \mathbf{w}=\mathbf{0} \text { in } \Omega \backslash N^m(T)\}
$$
and the restriction of $\Wdivzero$ by
$$
\Wdivzero(N^m(T))\,\,:=\,\,\{\mathbf{w} \in W(N^m(T)): \div \mathbf{w}=0\}.
$$

We also introduce the corresponding cut-off functions, which play a central role in the proof. For $T \in \mathcal{T}_H$ and $m\in \mathbb{N}$, let $\eta_T^m \in \{ v\in C^0(\overline{\Omega}) \, : \, \left. v\right\rvert_T \in \mathbb{P}_1(T) \mbox{ for all } T \in \mathcal{T}_H\}$ (i.e., $\eta_T^m$ is a $\mathbb{P}_1$ Lagrange finite element function) be the function that is uniquely determined by the properties
\begin{equation}\label{cut-off-func}
\eta_T^m(\mathbf{x}) = 
\begin{cases} 
0 &, \text{ if }\mathbf{x} \in N^{m}(T), \\ 
1 &, \text{ if }\mathbf{x} \in \Omega \setminus N^{m+1}(T). 
\end{cases}
\end{equation}
Note that $\eta_T^m\in W^{1,\infty}(\Omega)$ and $\| \nabla \eta_T^m \|_{L^{\infty}(\Omega)} \lesssim H^{-1}$.

\subsection{Localized multiscale approximation and main result}\label{Sec: loc mult app}

In this section, we introduce the localized multiscale method and present a corresponding a priori error estimate in the energy norm. 

As a first step, we define the localized correction operators.
\begin{definition}[Localized correction operators]
\label{definition:localized-corrector}
For given $T \in \mathcal{T}_H$ and a number of layers $m \in \mathbb{N}$, the localized element correction operator $\mathcal{C}_T^m: \Hzerodiv \rightarrow W_{\operatorname{div}\hspace{-1pt}0}(N^m(T)) \subset \Wdivzero$ is defined such that, for any $\mathbf{v}\in \Hzerodiv$, the function $\operatorname{\mathcal{C}}_T^m \mathbf{v} \in W_{\operatorname{div}\hspace{-1pt}0}(N^m(T))$ denotes the unique solution to the problem
\begin{equation}\label{Cmt prob}
a\left(\operatorname{\mathcal{C}}_T^m \mathbf{v}, \mathbf{w}\right)=a_T(\mathbf{v}, \mathbf{w})\qquad \mathrm{for}\; \mathrm{all}\;\; \mathbf{w} \in W_{\operatorname{div}\hspace{-1pt}0}(N^m(T)).
\end{equation}
The corresponding approximation of the global correction operator $\operatorname{\mathcal{C}}$ is given by the sum of local contributions, i.e.,
$$\operatorname{\mathcal{C}}^m:=\sum_{T \in \mathcal{T}_H} \operatorname{\mathcal{C}}_T^m.$$
\end{definition}

Just like the ideal correctors, the localized versions are still well-posed  by the Lax--Milgram theorem due to the coercivity and boundedness of $a(\cdot, \cdot)$ on the closed space $W_{\operatorname{div}\hspace{-1pt}0}(N^m(T))$. 

\begin{remark}
The implementation of \eqref{Cmt prob} can be achieved using Lagrange multipliers. Note that \eqref{Cmt prob} is equivalent to the constrained optimization problem
\begin{equation}
\min_{\mathbf{w} \in W_{\operatorname{div}\hspace{-1pt}0}(N^m(T))} \left\{\frac{1}{2}\,a( \mathbf{w}, \mathbf{w}) - a_T(\mathbf{v}, \mathbf{w})\right\}.
\end{equation}
By introducing Lagrange multipliers, we consider the Lagrangian functional
\begin{equation}
\mathcal{L}(\mathbf{w}, \lambda_1, \boldsymbol{\lambda_2}, \lambda_3): = \frac{1}{2}a( \mathbf{w}, \mathbf{w}) - a_T(\mathbf{v}, \mathbf{w}) + \left(\lambda_1, \div \mathbf{w} \right)_{L^2(N^m(T))} + \left(\boldsymbol{\lambda_2}, \operatorname{\pi}_H \mathbf{w} \right)_{L^2(N^m(T))}  + \left(\lambda_3, \lambda_1\right)_{L^2(N^m(T))}
\end{equation}
with $\mathbf{w} \in \Hzerodivomega{N^m(T)}$, $\lambda_1 \in L^2(N^m(T))$, $\boldsymbol{\lambda_2} \in V_H^k(N^m(T))$ and $\lambda_3 \in Q_H^k(N^m(T))$, where $\lambda_1$ enforces $\div \mathbf{w} = 0$, $\boldsymbol{\lambda_2}$ enforces $\operatorname{\pi}_H \mathbf{w} = 0$, and $\lambda_3$ ensures that $\lambda_1$ is orthogonal to $Q_H^k(N^m(T))$ to avoid a nontrivial kernel for the saddle point problem.  The optimality conditions for the saddle point of $\mathcal{L}(\mathbf{w}, \lambda_1, \boldsymbol{\lambda_2}, \lambda_3)$, i.e., $(\operatorname{\mathcal{C}}_T^m \mathbf{v}, \lambda_1^*, \boldsymbol{\lambda_2}^*, \lambda_3^*)$, are then given by
\begin{align*}
a(\operatorname{\mathcal{C}}_T^m \mathbf{v}, \mathbf{w})+(\lambda_1^*, \div \mathbf{w})_{L^2(N^m(T))} +(\boldsymbol{\lambda_2}^*, \operatorname{\pi}_H \mathbf{w})_{L^2(N^m(T))}& = a_T(\mathbf{v}, \mathbf{w})  \hspace{16pt} \mathrm{for}\; \mathrm{all}\;\; \mathbf{w} \in \Hzerodivomega{N^m(T)}, \\
(\div \operatorname{\mathcal{C}}_T^m \mathbf{v} + \lambda_3^*, \mu_1)_{L^2(N^m(T))} & = 0 \hspace{50pt} \mathrm{for}\; \mathrm{all}\;\; \mu_1\in L^2(N^m(T)), \\
(\operatorname{\pi}_H \operatorname{\mathcal{C}}_T^m \mathbf{v}, \boldsymbol{\mu}_2)_{L^2(N^m(T))} & = 0 \hspace{50pt} \mathrm{for}\; \mathrm{all}\;\; \boldsymbol{\mu}_2\in V_H^k(N^m(T)), \\
(\lambda_1^*, \mu_3)_{L^2(N^m(T))} & = 0 \hspace{50pt} \mathrm{for}\; \mathrm{all}\;\; \mu_3\in Q_H^k(N^m(T)).
\end{align*}
In practice, the spaces $\Hzerodivomega{N^m(T)}$ and $L^2(N^m(T))$ are discretized on a fine mesh (mesh size $h\ll H$) which resolves all fine-scale features, denoted as $V_h^k(N^m(T))$ and $Q_h^k(N^m(T))$, respectively.
\end{remark}

With the localized correctors, we can now define the localized multiscale space and the corresponding multiscale approximation.
\begin{definition}[Localized multiscale approximation]\label{def-loc-mul-ap}
Let $\operatorname{\mathcal{C}}^m: \Hzerodiv \rightarrow \Wdivzero$ denote the correction operator from Definition \ref{definition:localized-corrector}. We define the localized multiscale space by
$$
\VHkmsloc :=\left(\id - \hspace{1pt}\mathcal{C}^m\right) V_H^k.
$$
With this, the localized multiscale approximation (LOD approximation) is the solution $\left(\mathbf{u}_{H,m}^{\mathrm{ms}},p_{H,m}\right) \in \VHkmsloc \times \left(Q_H^k \cap L^2_0(\Omega) \right)$ to the problem
\begin{equation}\label{loca-mult-prob}
\begin{aligned}
a\left(\mathbf{u}_{H,m}^{\mathrm{ms}}, \mathbf{v}^{\mathrm{ms}}_H \right)+b\left( \mathbf{v}^{\mathrm{ms}}_H , p_{H,m}\right) &=0 \hspace{69pt} \mathrm{for}\; \mathrm{all}\;\; \mathbf{v}_H^{\mathrm{ms}} \in \VHkmsloc,\\ 
b\left(\mathbf{u}_{H,m}^{\mathrm{ms}}, q_H\right) &= -\left(f, q_H\right)_{\Ltwo} \quad\, \mathrm{for}\; \mathrm{all}\;\; q_H \in Q_H^k.
\end{aligned} 
\end{equation}
\end{definition}

Note that the second equation in \eqref{loca-mult-prob} still implies that $\div \mathbf{u}_{H,m}^{\mathrm{ms}} = -\operatorname{P}_{\hspace{-1pt}H} \hspace{-1pt}f$ as for the ideal method, because it holds \,\,$\div \mathbf{u}_{H,m}^{\mathrm{ms}} \in \div \VHkms \subset Q_H^k$ and
\begin{align*}
(\div \mathbf{u}_{H,m}^{\mathrm{ms}} + \operatorname{P}_{\hspace{-1pt}H} \hspace{-1pt}f, q_H )_{L^2(\Omega)} =0 \quad\;\;\, \mbox{for all } q_H \in Q_H^k.
\end{align*}

The discrete multiscale problem \eqref{loca-mult-prob} is well-posed, as shown by the following lemma that can be proved analogously to \cite[Lemma 12]{hellman2016multiscale} with the improvements suggested in \cite[Section 4.3]{hellman2016multiscale}.

\begin{lemma}[Existence and uniqueness of LOD approximations]
\label{lemma:inf-sup-VHmsloc}
The localized multiscale method \eqref{loca-mult-prob} has a unique solution. In particular, there exists an $H$-independent constant $\rho_0>0$ (which only depends on $\alpha$, $\beta$, $\rho$, and $m$) such that
$$
\inf_{q_H \in Q_H^k\cap L^2_0(\Omega)\backslash\{0\}}  \sup _{\mathbf{v}^{\mathrm{ms}}_H \in \VHkmsloc\backslash\{\mathbf{0}\}} \frac{b( \mathbf{v}^{\mathrm{ms}}_H , q_H)}{\|q_H \|_{\Ltwo}\|\mathbf{v}^{\mathrm{ms}}_H \|_{\Hdiv}} \,\,\ge \,\, \rho_0.
$$
\end{lemma}

The main result of our contribution is the following a priori error estimate for the localized multiscale method.

\begin{theorem}[Error estimate for the localized multiscale approximation]\label{main-thm}
Let Assumptions \ref{assu-mesh} and \ref{assu-stab-inte} hold and let $\Omega$ be contractible. Denote by $(\mathbf{u},p)\in \Hzerodiv \times L^2_0(\Omega)$ the exact solution to the multiscale problem \eqref{eqn-mixed-form-var}. For a coarse mesh $\mathcal{T}_H$, a polynomial degree $k \in \mathbb{N}_0$, and a given number of layers $m \in \mathbb{N}$, we let $\VHkmsloc$ denote the corresponding LOD multiscale space given by Definition \ref{def-loc-mul-ap}, and $(\mathbf{u}_{H,m}^{\mathrm{ms}}, p_{H,m}) \in \VHkmsloc \times (Q_H^k \cap L^2_0(\Omega))$ the corresponding LOD approximation given by \eqref{loca-mult-prob}. Then, it holds
\begin{eqnarray}
\label{sub-estimate-2-loc-method}
\|\div \,(\mathbf{u}-\mathbf{u}_{H,m}^{\mathrm{ms}} )\hspace{1pt}\|_{\Ltwo}  &=&
\left\| f-\operatorname{P}_H \hspace{-2pt}f \right\|_{\Ltwo},
\end{eqnarray}
and there exists a decay rate $\theta\in(0,1)$ that depends on the contrast $\beta / \alpha$, but not on $m$ or $H$, such that
\begin{eqnarray}
\label{main-estimate-loc-method}
\vertiii{\mathbf{u} - \mathbf{u}_{H,m}^{\mathrm{ms}} } &\lesssim&  H\left\|f- \operatorname{P}_H \hspace{-2pt} f\right\|_{\Ltwo} \, + \, m^{d / 2} \, \theta^{m} \, \|f\|_{\Ltwo},
\end{eqnarray}
and
\begin{eqnarray}
\label{sub-estimate-1-loc-method}
    \|p_{H,m}-p\|_{L^2(\Omega)} &\lesssim& H\left\|f- \operatorname{P}_H \hspace{-2pt} f\right\|_{\Ltwo} \,+\, \|p - \operatorname{P}_H \hspace{-2pt}p \|_{L^2(\Omega)} \, + \, m^{d / 2} \, \theta^{m} \, \|f\|_{\Ltwo}.
\end{eqnarray}
In particular, if $f \in H^{k+1}(T)$ for all $T \in \mathcal{T}_H$ and if the
number of layers $m$ is sufficiently large so that $m>\tfrac{|\log(H)|}{|\log(\theta)|}(k+2)$, then we obtain the optimal-order error bound 
\begin{equation*}
\| \mathbf{u} - \mathbf{u}_{H,m}^{\mathrm{ms}} \|_{L^2(\Omega)} \,+\, H\, \| \mathbf{u} - \mathbf{u}_{H,m}^{\mathrm{ms}} \|_{\Hdiv}  \,\,\,\lesssim\,\,\,  H^{k+2} \,\, m^{d / 2} \, \left( \sum_{T \in \mathcal{T}_H} \|f \|_{H^{k+1}(T)}^2 \right)^{1/2}.
\end{equation*}
\end{theorem}

The proof of the above theorem is postponed to Section \ref{sec-proof-main-thm}.\\[-0.8em]

The $a$-orthogonality between $\VHkms$ and $\Wdivzero$, which holds in the ideal case (cf. equation \eqref{a-orthogonality}), is no longer valid in the localized setting, i.e., we do not have $a$-orthogonality between $\VHkmsloc$ and $\Wdivzero$. Since the orthogonality was heavily exploited in the error analysis of the ideal method (cf. Theorem \ref{thm-error-ideal-problem}), the proof of Theorem \ref{main-thm} requires a different strategy based on decay estimates for functions of the form $\mathcal{C}^m_T \mathbf{v}$. The corresponding estimates are established in the next subsection.

\subsection{Exponential decay of element correctors for $d=3$}\label{Sec: exp decay 3d}

In this subsection, we establish the major step in the proof of Theorem \ref{main-thm}, that is, we show the exponential decay of correctors with decay rates measured in units of the coarse mesh size $H$. The results in this section rely fundamentally on the stability properties of the quasi-interpolation $\pi_H$ as specified in Assumption \ref{assu-stab-inte}.

The proof of the decay depends on the spatial dimension $d$ and is different for $d=2$ and $d=3$. In the following, we restrict our analysis to the more challenging case $d=3$, whereas the proof for $d=2$ is given in Appendix \ref{appendix:section:decay-2D}.

Our analysis takes inspiration from \cite{gallistl2018numerical, henning2020computational} and requires $\Hcurl$-spaces. We define
$$
\Hcurl \,\, :=\,\, \{ \, \mathbf{v} \in L^2(\Omega,\mathbb{R}^3) \, : \, \curlbf \mathbf{v} \in L^2(\Omega,\mathbb{R}^3) \, \}
$$
and the subspace of $\Hcurl$-functions with a vanishing tangential trace by
$$
\Hzerocurl \,\, := \,\, \{ \, \mathbf{v} \in \Hcurl \, : \, \mathbf{v} \times \mathbf{n} \vert_{\partial \Omega} =\mathbf{0}\,\}.
$$
Furthermore, we require the space of vector-valued $H^1$-functions with homogeneous Dirichlet boundary values. For a Lipschitz subdomain $\omega\subset \Omega$, we define
$$
\mathbf{H}^1_0(\omega)\,\, := \,\, \{ \, \mathbf{v} \in H^1(\omega,\mathbb{R}^3) \, : \, \mathbf{v}\vert_{\partial \omega} = \mathbf{0} \, \}.
$$
The Jacobian of a function $\mathbf{v} \in \mathbf{H}^1_0(\omega)$ will be denoted by $\nabla \mathbf{v}$.

With these spaces at hand, we can formulate an important auxiliary result that establishes a regular decomposition of the interpolation error $\mathbf{v}-\pi_H(\mathbf{v})$ together with corresponding local stability estimates.

\begin{lemma}[Decomposition of $(\mathrm{id}-\pi_H)$]\label{lem-stable-decomposition}
Let $d=3$, and let $\pi_H : \Hzerodiv \rightarrow \mathcal{RT}_{\hspace{-1pt}k}(\mathcal{T}_{H}) \cap \,\Hzerodiv$ be a stable quasi-interpolation operator satisfying Assumptions \ref{assu-stab-inte}. If $\Omega$ is contractible, then for any $\mathbf{v} \in \Hzerodiv$, there exist $\mathbf{r} \in \mathbf{H}_0^1(\Omega)$ and $\mathbf{q} \in \Hzerocurl$ such that
\begin{align}\label{v-piHv deco}
\mathbf{v}-\pi_H(\mathbf{v})=\mathbf{r}+\curlbf \mathbf{q},
\end{align}
and, for every $T \in \mathcal{T}_H$, it holds
\begin{align}\label{v-piHv bds}
\begin{split}
&\|\mathbf{r}\|_{L^2\left(T\right)} + H\|\nabla \mathbf{r}\|_{L^2\left(T\right)}
 + H^{-1}\|\mathbf{q}\|_{L^2\left(T\right)} + \|\curlbf \mathbf{q}\|_{L^2\left(T\right)}  \\[0.2em]
&\,\,\,\,\,\lesssim\,\,\, \|\mathbf{v}\|_{L^2\left(N^3(T)\right)}+H\|\operatorname{div} \mathbf{v}\|_{L^2\left(N^3(T)\right)},\hspace{60pt}
\end{split}
\end{align}
where the generic hidden constants  only depend on the shape of the coarse mesh.

If $\operatorname{div} \mathbf{v}=0$, then we have \eqref{v-piHv deco} and \eqref{v-piHv bds} with $\mathbf{r}=\mathbf{0}$, i.e.,
$$
\mathbf{v}-\pi_H(\mathbf{v})=\curlbf \mathbf{q}
\qquad\mathrm{with}
\qquad
\|\mathbf{q}\|_{L^2\left(T\right)} + H \|\curlbf\mathbf{q}\|_{L^2\left(T\right)} 
\,\lesssim\, H \, \|\mathbf{v}\|_{L^2\left(N^3(T)\right)} .
$$
\end{lemma}

\begin{proof}
Let $\mathbf{v} \in \Hzerodiv$, and let $I_H^S: \Hzerodiv \rightarrow V_H^0$ denote the quasi-interpolation operator introduced by Sch\"{o}berl in \cite{schoberl2008posteriori}. Note that the Sch\"{o}berl interpolation for $\Hdiv$ spaces is only defined for $d=3$ and for the lowest order Raviart--Thomas space $V_H^0=\mathcal{RT}_{\hspace{-1pt}0}(\mathcal{T}_{H}) \cap \,\Hzerodiv$. It is shown in \cite[Theorem 11]{schoberl2008posteriori} that there exists a decomposition
\begin{equation*}
\mathbf{v}-I_H^S(\mathbf{v})=\sum_{ z \in \mathcal{N}_H} \mathbf{v}_z
\end{equation*}
for some local patch functions $\mathbf{v}_z \in \Hzerodivomega{\omega_z}$ (extended by zero to $\Omega\setminus\omega_z$) that satisfy the local stability estimates
\begin{equation}\label{deco-vp-stab}
\|\mathbf{v}_z\|_{L^2\left(\omega_z\right)} \lesssim\|\mathbf{v}\|_{L^2\left(N^1\left(\omega_z\right)\right)},\qquad \|\operatorname{div} \mathbf{v}_z\|_{L^2\left(\omega_z\right)} \lesssim\|\operatorname{div} \mathbf{v}\|_{L^2\left(N^1\left(\omega_z\right)\right)},
\end{equation}
where $N^1(\omega_z):=\bigcup\left\{T \in \mathcal{T}_H: T \cap \omega_z\neq \emptyset\right\}$ denotes the patch neighborhood of $\omega_z$. Since the quasi-interpolation operator $\operatorname{\pi}_H: \Hzerodiv \rightarrow V_H^k$ is a projection and $V_H^0\subset V_H^k$, we have that $\pi_H \circ I_H^S = I_H^S$, and hence,
\begin{equation*}
\mathbf{v}-\operatorname{\pi}_H(\mathbf{v})=\mathbf{v}-I_H^S(\mathbf{v})-\operatorname{\pi}_H\left(\mathbf{v}-I_H^S \mathbf{v}\right)=\sum_{z \in \mathcal{N}_H}\mathbf{\tilde{v}}_z,\qquad\text{where}\quad \mathbf{\tilde{v}}_z := \left(\id-\operatorname{\pi}_H\right)\mathbf{v}_z.
\end{equation*}
Recalling from \eqref{h-div-inte-stab} that $\operatorname{\pi}_H$ is a local operator, we observe that $\mathbf{\tilde{v}}_z \in \Hzerodivomega{N^1(\omega_z)}$. Further, exploiting the local stability bound \eqref{h-div-inte-stab} for $\operatorname{\pi}_H$, the fact that $\left.\mathbf{v}_z\right\rvert_{\Omega\setminus \omega_z} = \mathbf{0}$, and the local stability estimate \eqref{deco-vp-stab}, we obtain that
$$
\begin{aligned}
\|\mathbf{\tilde{v}}_z\|_{L^2\left(\omega_z\right)} + 
H\|\operatorname{div}\mathbf{\tilde{v}}_z\|_{L^2\left(\omega_z\right)} 
& \lesssim\|\mathbf{v}\|_{L^2\left(N^1\left(\omega_z\right)\right)}+H\|\operatorname{div} \mathbf{v}\|_{L^2\left(N^1\left(\omega_z\right)\right)}.
\end{aligned}
$$

Next, in view of the fact that $N^1(\omega_z)$ is contractible, it follows from \cite[Lemma 3.8]{hiptmair2007nodal} that for $d=3$ there exists a continuous linear operator $\mathbf{R}:\Hzerodivomega{N^1(\omega_z)}\rightarrow \mathbf{H}_0^1\left(N^1\left(\omega_z\right)\right)$ such that $\div(\mathbf{R}(\mathbf{w})) = \div(\mathbf{w})$ for any $\mathbf{w}\in \Hzerodivomega{N^1(\omega_z)}$, and we have (cf. \cite[Corollary 3.9]{hiptmair2007nodal}) the bound 
\begin{align}\label{opR bd}
    \|\mathbf{R}(\mathbf{w})\|_{H^1(N^1(\omega_z),\mathbb{R}^3)} \lesssim \|\div \mathbf{w}\|_{L^2(N^1(\omega_z))}\qquad\text{for all }\mathbf{w}\in \Hzerodivomega{N^1(\omega_z)}.
\end{align}
We set $\mathbf{r}_z:= \mathbf{R}(\mathbf{\tilde{v}}_z)\in \mathbf{H}_0^1\left(N^1\left(\omega_z\right)\right)$. Since $N^1(\omega_z)$ is simply-connected and observing that $\left(\mathbf{\tilde{v}}_z - \mathbf{r}_z\right)\in \Hzerodivzeroomega{N^1(\omega_z)}$, we have by \cite[Theorem 3.6]{girault2012finite} that there exists a function $\mathbf{q}_z\in \Hzerocurlomega{N^1(\omega_z)}$ such that
\begin{align}\label{dec}
    \mathbf{\tilde{v}}_z=\mathbf{r}_z+\curlbf \mathbf{q}_z,\qquad \div \mathbf{q}_z = 0.
\end{align}
Using the Poincar\'e-inequality for functions in $\mathbf{H}^1_0(N^1(\omega_z))$ and the bound \eqref{opR bd}, we find that
\begin{align}\label{rbd}
    H^{-1}\|\mathbf{r}_z\|_{L^2\left(N^1\left(\omega_z\right)\right)}\,\,\lesssim\,\,\|\nabla \mathbf{r}_z\|_{L^2\left(N^1\left(\omega_z\right)\right)} \,\,\lesssim\,\, \|\operatorname{div}\mathbf{\tilde{v}}_z\|_{L^2\left(N^1\left(\omega_z\right)\right)}.
\end{align}
In view of the fact that $\mathbf{q}_z\in \Hzerocurlomega{N^1(\omega_z)}$ and $(\mathbf{q}_z,\nabla \phi )_{L^2(N^1(\omega_z))} = 0$ for all $\phi \in H^1_0(N^1(\omega_z))$, we use the Poincar\'e-type inequality from \cite[Lemma 44.4]{ern2021finite}, the decomposition \eqref{dec}, and the bound \eqref{rbd} to obtain that
\begin{align}\label{qbd}
    H^{-1}\|\mathbf{q}_z\|_{L^2(N^1(\omega_z))} \lesssim  \|\curlbf \mathbf{q}_z\|_{L^2(N^1(\omega_z))}  \lesssim \|\mathbf{\tilde{v}}_z\|_{L^2(N^1(\omega_z))} + H\|\operatorname{div}\mathbf{\tilde{v}}_z\|_{L^2\left(N^1\left(\omega_z\right)\right)}.
\end{align}

We can now collect the local contributions and define
\begin{align}\label{sum loc cont}
    \mathbf{r}\,\,:=\,\,\sum_{z \in \mathcal{N}_H} \mathbf{r}_z \qquad\mbox{and} \qquad \mathbf{q}:=\sum_{z \in \mathcal{N}_H} \mathbf{q}_z.
\end{align}
In view of \eqref{rbd} and \eqref{qbd}, we obtain the desired decomposition \eqref{v-piHv deco} of $\mathbf{v}-\operatorname{\pi}_H(\mathbf{v}) = \sum_{z \in \mathcal{N}_H}\mathbf{\tilde{v}}_z$ together with the bound
\begin{equation*}
\begin{aligned}
 & H^{-1}\|\mathbf{r}\|_{L^2\left(T\right)}+\|\nabla \mathbf{r}\|_{L^2\left(T\right)}
\leq \sum_{z \in\mathcal{N}_H\cap N^1(T)} \left(H^{-1}\|\mathbf{r}_z\|_{L^2(T)}+\|\nabla \mathbf{r}_z\|_{L^2(T)}\right) \\
&\leq \sum_{z \in\mathcal{N}_H\cap N^1(T)} \left(H^{-1}\|\mathbf{r}_z\|_{L^2(N^1(\omega_z))}+\|\nabla \mathbf{r}_z\|_{L^2(N^1(\omega_z))}\right) \\
& \lesssim \sum_{z \in\mathcal{N}_H\cap N^1(T)} \|\operatorname{div}\mathbf{\tilde{v}}_z\|_{L^2\left(N^1\left(\omega_z\right)\right)} \lesssim H^{-1}\|\mathbf{v}\|_{L^2\left(N^3(T)\right)}+\|\operatorname{div} \mathbf{v}\|_{L^2\left(N^3(T)\right)},
\end{aligned}
\end{equation*}
and the bound
\begin{equation*}
\begin{aligned}
& H^{-1}\|\mathbf{q}\|_{L^2(T)} + \|\curlbf\mathbf{q}\|_{L^2(T)} 
\leq \sum_{z \in\mathcal{N}_H\cap N^1(T)} \left( H^{-1}\| \mathbf{q}_z\|_{L^2(T)} + \|\curlbf \mathbf{q}_z\|_{L^2(T)}\right) \\
& \lesssim \sum_{z \in\mathcal{N}_H\cap N^1(T)}\left( \|\mathbf{\tilde{v}}_z\|_{L^2\left(N^1(\omega_z)\right)} + H\|\operatorname{div} \mathbf{\tilde{v}}_z \|_{L^2\left(N^1\left(\omega_z\right)\right)} \right) \lesssim \|\mathbf{v}\|_{L^2\left(N^3(T)\right)}+H\|\operatorname{div} \mathbf{v}\|_{L^2\left(N^3(T)\right)} .
\end{aligned}
\end{equation*}
Finally, if $\div \mathbf{v}=0$, then due to \eqref{deco-vp-stab} we have that $\operatorname{div} \mathbf{v}_z = 0$ for any $z \in\mathcal{N}_H$. In view of \eqref{rbd}, we then find that $\mathbf{r}_z = \mathbf{0}$ for any $z \in\mathcal{N}_H$, and hence, $\mathbf{r} = \mathbf{0}$ by \eqref{sum loc cont}.
\end{proof}

We are now prepared to establish the exponential decay property for solutions in $\Wdivzero$ to problems with local source terms. This result is applied to quantify the decay of the element correctors.

\begin{lemma}[Exponential decay of correctors in $\Wdivzero$]\label{lem-expo-deca-phi}
For $d\in \{2,3\}$ and contractible $\Omega$, let $T \in \mathcal{T}_H$ be a coarse element and $F_T \in \left( \Wdivzero \right)^{\prime}$ a given local source functional in the sense that $F_T(\mathbf{w})=0$ whenever $\mathbf{w} \in \Wdivzero$ with $\left.\mathbf{w}\right\rvert_T = \mathbf{0}$. In other words, $F_T$ is localized to the coarse element $T$. If $\boldsymbol{\varphi}_T \in \Wdivzero$ denotes the unique solution to the problem
\begin{align}
\label{def-sol-varphi-T}
a\left(\boldsymbol{\varphi}_T, \mathbf{w}\right) =F_T(\mathbf{w}) \qquad \mathrm{for}\; \mathrm{all}\;\; \mathbf{w} \in \Wdivzero,
\end{align}
then there exists a constant $\theta\in (0,1)$, independent of $H$, $T$, $m$, and $F_T$, such that we have the bound
$$
\vertiii{\boldsymbol{\varphi}_T}_{ \Omega \backslash \mathrm{N}^m(T)} \,\lesssim \,\theta^m\vertiii{\boldsymbol{\varphi}_T} \qquad\mathrm{for}\; \mathrm{all}\;\; m \in \mathbb{N}. 
$$ 
\end{lemma}
\begin{proof}
Let $d=3$. The proof for $d=2$ is given in the appendix, Section \ref{appendix:section:decay-2D}. Note that it is sufficient to prove the result for all sufficiently large $m$. In the following, we assume that $m \in \mathbb{N}$ with $m\geq 7$ is fixed. With this, let $\eta:= \eta^{m-4}_T$ denote the (continuous and piecewise linear) cut-off function introduced in \eqref{cut-off-func} with $m$ replaced $m-4$. In particular, it holds 
$$
\eta=0 \quad \text {in} \quad \mathrm{N}^{m-4}(T), \qquad \eta=1 \quad \text{in} \quad \Omega \backslash \mathrm{N}^{m-3}(T),\qquad  \|\,\lvert \nabla \eta\rvert\, \|_{L^{\infty}(\Omega)}\lesssim H^{-1}.
$$ 
For brevity, we will write $R^m_T:= \overline{N^{m-3}(T)\backslash N^{m-4}(T)}$.

In view of the non-negativity of $\eta$ and the uniform ellipticity of $\mathbf{A}^{-1}$, we have that
\begin{equation}\label{coer-bili}
\vertiii{\boldsymbol{\varphi}_T}_{\Omega\backslash N^m(T)}^2 = (\mathbf{A}^{-1}\boldsymbol{\varphi}_T,\boldsymbol{\varphi}_T)_{L^2(\Omega\backslash N^m(T))} \leq \left(\eta \, \mathbf{A}^{-1} \boldsymbol{\varphi}_T, \boldsymbol{\varphi}_T\right)_{L^2(\Omega)}.
\end{equation}
Noting that $\operatorname{div}\boldsymbol{\varphi}_T = 0$ and $\operatorname{\pi}_H(\boldsymbol{\varphi}_T) = \mathbf{0}$, we can apply Lemma \ref{lem-stable-decomposition} to deduce that there exists a function $\mathbf{q}_T \in \Hzerocurl$ such that
\begin{equation}\label{w-curl-q}
\boldsymbol{\varphi}_T = \boldsymbol{\varphi}_T - \operatorname{\pi}_H(\boldsymbol{\varphi}_T) = \curlbf \mathbf{q}_T.
\end{equation}
Then, using the definition of $a(\cdot,\cdot)$ and the fact that $\curlbf\left(\eta\,\mathbf{q}_T\right) = \eta\, \curlbf\mathbf{q}_T + \nabla\eta \times  \mathbf{q}_T$, we obtain
\begin{equation}\label{bili-spli}
\begin{aligned}
& \left(\eta \, \mathbf{A}^{-1} \boldsymbol{\varphi}_T, \boldsymbol{\varphi}_T\right)_{L^2(\Omega)} = a(\eta\, \curlbf\mathbf{q}_T, \boldsymbol{\varphi}_T) = a\left( \curlbf\left(\eta\,\mathbf{q}_T\right), \boldsymbol{\varphi}_T\right) - a\left(\nabla\eta \times  \mathbf{q}_T, \boldsymbol{\varphi}_T\right)\\
= &\,\, a\left( \operatorname{\pi}_H(\curlbf\left(\eta\,\mathbf{q}_T)\right), \boldsymbol{\varphi}_T\right) + a\left( \left(\id-\operatorname{\pi}_H\right)\curlbf\left(\eta\,\mathbf{q}_T\right), \boldsymbol{\varphi}_T\right) - a\left(\nabla\eta \times  \mathbf{q}_T, \boldsymbol{\varphi}_T\right).
\end{aligned}
\end{equation}
Observing that $\mathbf{r} :=\left(\id-\operatorname{\pi}_H\right)\curlbf\left(\eta\,\mathbf{q}_T\right) \in  \Wdivzero$ and $\left.\mathbf{r}\right\rvert_T = \mathbf{0}$, we see that
\begin{equation}\label{bili-spli-1}
a\left( \left(\id-\operatorname{\pi}_H\right)\curlbf\left(\eta\,\mathbf{q}_T\right), \boldsymbol{\varphi}_T\right) = a\left(\boldsymbol{\varphi}_T,\mathbf{r}\right) = F_T(\mathbf{r}) = 0.
\end{equation}
Since we have on $\Omega \backslash \mathrm{N}^{m-2}(T)$ that $\operatorname{\pi}_H(\curlbf\left(\eta\,\mathbf{q}_T\right))=\operatorname{\pi}_H(\curlbf\mathbf{q}_T) = \operatorname{\pi}_H (\boldsymbol{\varphi}_T) =\mathbf{0}$, and since $\eta=0$ on $\mathrm{N}^{m-4}(T)$, we have that $\operatorname{supp} (\operatorname{\pi}_H(\curlbf\left(\eta\,\mathbf{q}_T\right))) \subset N^1(R_T^m)$, and hence,
\begin{equation}
\label{est-piH-curl}
a\left( \operatorname{\pi}_H(\curlbf\left(\eta\,\mathbf{q}_T\right)), \boldsymbol{\varphi}_T\right) \,\,\lesssim \,\,  \beta \,\|\curlbf\left(\eta\,\mathbf{q}_T\right)\|_{L^2\left(N^2(R_T^m)\right)} \|\boldsymbol{\varphi}_T\|_{L^2\left(N^1(R_T^m)\right)}.
\end{equation}
Using the properties of $\eta$, the identity \eqref{w-curl-q}, and Lemma \ref{lem-stable-decomposition}, we obtain that
\begin{equation}\label{control-by-phi-T}
\begin{aligned}
& \|\curlbf\left(\eta\,\mathbf{q}_T\right)\|_{L^2\left(N^2(R^m_T)\right)} \,\,\leq \,\,\|\eta\,\boldsymbol{\varphi}_T\|_{L^2\left(N^2(R^m_T)\right)}+ \|\nabla \eta \times \mathbf{q}_T\|_{L^2\left(N^2(R^m_T)\right)}\\
\lesssim \, &\,\, \|\boldsymbol{\varphi}_T\|_{L^2\left(N^{m-1}(T)\backslash N^{m-4}(T)\right)} +  H^{-1}\|\mathbf{q}_T\|_{L^2\left(N^{m-3}(T)\backslash N^{m-4}(T)\right)} \,\,\lesssim \,\,\|\boldsymbol{\varphi}_T\|_{L^2\left(N^{m}(T)\backslash N^{m-7}(T)\right)}.
\end{aligned}
\end{equation}
Therefore, combining the bounds \eqref{est-piH-curl} and \eqref{control-by-phi-T}, we find
\begin{equation}\label{bili-spli-2}
a\left( \operatorname{\pi}_H(\curlbf\left(\eta\,\mathbf{q}_T\right)), \boldsymbol{\varphi}_T\right) \lesssim \beta \|\boldsymbol{\varphi}_T\|^2_{L^2\left(N^{m}(T)\backslash N^{m-7}(T)\right)}.
\end{equation}
With similar arguments, we have that
\begin{equation}\label{bili-spli-3}
a\left(\nabla\eta \times  \mathbf{q}_T, \boldsymbol{\varphi}_T\right) \,\,\leq\,\, \beta\, \|\nabla\eta \times  \mathbf{q}_T\|_{L^2(R^m_T)} \|\boldsymbol{\varphi}_T\|_{L^2(R^m_T)} \lesssim \beta \|\boldsymbol{\varphi}_T\|^2_{L^2\left(N^{m}\backslash N^{m-7}\right)}.
\end{equation}
Combining the estimates \eqref{coer-bili}, \eqref{bili-spli}, \eqref{bili-spli-1}, \eqref{bili-spli-2}, and \eqref{bili-spli-3} altogether, we conclude that for some constant $C>0$ independent of $T$, $m$, $H$, and $\mathbf{A}$, it holds
\begin{equation*}
\begin{aligned}
  \vertiii{\boldsymbol{\varphi}_T}_{\Omega\backslash N^m(T)}^2 &\leq C \beta \|\boldsymbol{\varphi}_T\|_{L^2\left(N^m(T)\backslash N^{m-7}(T)\right)}^2 
 \\ &\leq  C \frac{\beta}{\alpha} \vertiii{\boldsymbol{\varphi}_T}_{N^m(T)\backslash N^{m-7}(T)}^2 
 \leq  C \frac{\beta}{\alpha} 
 \left(\vertiii{\boldsymbol{\varphi}_T}_{\Omega\backslash N^{m-7}(T)}^2  -\vertiii{\boldsymbol{\varphi}_T}_{\Omega\backslash N^m(T)}^2 \right). 
\end{aligned}
\end{equation*}
Setting $\theta := \frac{C \frac{\beta}{\alpha}}{1+C \frac{\beta}{\alpha}} < 1$, we find that $\vertiii{\boldsymbol{\varphi}_T}_{\Omega\backslash N^m(T)}^2 \leq \theta \vertiii{\boldsymbol{\varphi}_T}_{\Omega\backslash N^{m-7}(T)}^2$. A recursive application of this estimate yields
\begin{equation*}
\vertiii{\boldsymbol{\varphi}_T}_{\Omega\backslash N^m(T)}^2 \lesssim \theta^{\lfloor m / 7\rfloor}\vertiii{\boldsymbol{\varphi}_T}^2,
\end{equation*}
which completes the proof.
\end{proof}
Next, we quantify the error between the ideal solution $\boldsymbol{\varphi}_T \in \Wdivzero$ to problem \eqref{def-sol-varphi-T} and its corresponding approximation $\boldsymbol{\varphi}_T^m$ obtained by truncating the problem to the subset $N^m(T)$.

\begin{lemma}\label{lem-expo-deca-phi-phi-l}
For $d\in\{2,3\}$ and $\Omega$ contractible, let $\boldsymbol{\varphi}_T \in \Wdivzero$ denote the solution to the global problem \eqref{def-sol-varphi-T}, and let $\boldsymbol{\varphi}_T^m \in W_{\operatorname{div}\hspace{-1pt}0}(N^m(T))$ denote the solution to the corresponding truncated problem 
\begin{equation*}
a\left(\boldsymbol{\varphi}_{T}^m, \mathbf{w}\right) = F_T\left(\mathbf{w}\right) \qquad \mathrm{for}\; \mathrm{all}\;\; \mathbf{w} \in W_{\operatorname{div}\hspace{-1pt}0}(N^m(T)),
\end{equation*}
where $m\in \mathbb{N}$, and $F_T \in \left(W_{\operatorname{div}\hspace{-1pt}0}(N^m(T))\right)'$ is again a local source functional with $F_T(\mathbf{w})=0$ whenever $\mathbf{w} \in W_{\operatorname{div}\hspace{-1pt}0}(N^m(T))$ with $\left.\mathbf{w}\right\rvert_T = \mathbf{0}$. Then, there exists a decay rate $\theta\in (0,1)$, independent of $H$, $T$, $m$, and $F_T$, such that
$$
\vertiii{\boldsymbol{\varphi}_T-\boldsymbol{\varphi}_{T}^m} \,\,\lesssim \,\,\theta^m\vertiii{\boldsymbol{\varphi}_T}.
$$
\end{lemma}
\begin{proof}
Let $d=3$. The proof for $d=2$ is given in the appendix, Section \ref{appendix:section:decay-2D}.

Let $m\in \mathbb{N}$ with $m\geq 5$. As $\boldsymbol{\varphi}_{T}^m$ is the Galerkin approximation of $\boldsymbol{\varphi}_{T}$ in the closed subspace $W_{\operatorname{div}\hspace{-1pt}0}(N^m(T))\subset \Wdivzero$, we have the best-approximation bound
\begin{equation}
\label{best-approx-phiTm}
\vertiii{\boldsymbol{\varphi}_T-\boldsymbol{\varphi}_{T}^m} \,\,\leq\,\, \inf_{\mathbf{w} \in W_{\operatorname{div}\hspace{-1pt}0}(N^m(T))} \vertiii{\boldsymbol{\varphi}_T-\mathbf{w} }.
\end{equation}
Consider the cut-off function $\eta:= 1-\eta_{m-2}^T$, where $\eta_{m-2}^T$ is the function defined in \eqref{cut-off-func} with $m$ replaced by $m-2$. In particular,
$$
\eta=1 \quad \text {in} \quad \mathrm{N}^{m-2}(T), \qquad\quad \eta=0 \quad \text{in} \quad \Omega \backslash \mathrm{N}^{m-1}(T).
$$ 
By Lemma \ref{lem-stable-decomposition}, we have $\boldsymbol{\varphi}_T = (\id-\operatorname{\pi}_H)\boldsymbol{\varphi}_T  = \curlbf \mathbf{q}_T = (\id-\operatorname{\pi}_H)\curlbf \mathbf{q}_T$ for some $\mathbf{q}_T \in \Hzerocurl$. With the choice $\mathbf{w}=\tilde{\boldsymbol{\varphi}}_T^m := (\id-\operatorname{\pi}_H)\curlbf(\eta\,\mathbf{q}_T) \in W_{\operatorname{div}\hspace{-1pt}0}(N^m(T))$ in the best-approximation estimate \eqref{best-approx-phiTm}, we obtain that
\begin{equation*}
\vertiii{\boldsymbol{\varphi}_T-\tilde{\boldsymbol{\varphi}}_T^m} 
\,=\,
\vertiii{ (\id-\operatorname{\pi}_H)\curlbf \mathbf{q}_T -(\id-\operatorname{\pi}_H)\curlbf(\eta\,\mathbf{q}_T) }
\,\overset{\eqref{h-div-inte-stab}}{\lesssim} \, \sqrt{\beta} \, \|\curlbf\left( \left( 1- \eta\right)\mathbf{q}_T \right) \|_{L^2(\Omega)}.
\end{equation*}
Using arguments similar to \eqref{control-by-phi-T}, we have
\begin{equation*}
\|\curlbf\left( \left( 1- \eta\right)\mathbf{q}_T \right) \|_{L^2(\Omega)} \lesssim  \frac{1}{\sqrt{\alpha}}\vertiii{\boldsymbol{\varphi}_T}_{ \Omega\backslash N^{m-5}(T)}.
\end{equation*}
By Lemma \ref{lem-expo-deca-phi}, we also have that
\begin{equation*}
\vertiii{\boldsymbol{\varphi}_T}_{\Omega\backslash N^{m-5}(T)} \lesssim \, \theta^m \vertiii{\boldsymbol{\varphi}_T}.
\end{equation*}
Combining the previous estimates concludes the proof.
\end{proof}

We are now ready to quantify the error that arises from the approximation of the ideal corrector $\mathcal{C}$ by localized (truncated) element contributions. 

\begin{lemma}[Exponential decay of localization error]\label{lem-expo-deca-loca-erro}
For $d\in\{2,3\}$ and $\Omega$ contractible, we let $\mathcal{C}:\Hzerodiv \rightarrow \Wdivzero$ denote the ideal corrector from Definition \ref{def:ideal-corrector}, and $\mathcal{C}^m :\Hzerodiv \rightarrow \Wdivzero$ denote the corresponding localized corrector from Definition \ref{definition:localized-corrector} for a given number of layers $m\in \mathbb{N}_0$. Then, there exists a decay rate $\theta\in (0,1)$, independent of $m$ and $H$, such that
\begin{equation*}
\vertiii{\left(\operatorname{\mathcal{C}}-\operatorname{\mathcal{C}}^m\right)\mathbf{v}}
\,\, \lesssim \,\, \left( m^{\frac{d}{2}} + 1 \right)\theta^m \, \vertiii{\mathbf{v}}
\qquad \mathrm{for}\; \mathrm{all}\;\; \mathbf{v}\in \Hzerodiv.
\end{equation*}
\end{lemma}
\begin{proof}
Let $d=3$. The proof for $d=2$ is given in the appendix, Section \ref{appendix:section:decay-2D}. 

Let $\mathbf{v}\in \Hzerodiv$ and denote the error by $\mathbf{e}:=\left(\operatorname{\mathcal{C}}-\operatorname{\mathcal{C}}^m\right)\mathbf{v} \in \Wdivzero$. For each $T \in \mathcal{T}_H$ consider the cut-off $\eta_T:=\eta_{m+1}^T$ as defined in \eqref{cut-off-func} with $m$ replaced by $m+1$. It holds
\begin{equation*}
\eta_T=0 \quad \text {in} \quad \mathrm{N}^{m+1}(T), \qquad\quad \eta_T=1 \quad \text {in} \quad \Omega \backslash \mathrm{N}^{m+2}(T).
\end{equation*}
Using Lemma \ref{lem-stable-decomposition}, we represent the error as 
\begin{align*}
    \mathbf{e} =\left( \id-\operatorname{\pi}_H \right) \mathbf{e} = \curlbf \mathbf{q} = \left( \id-\operatorname{\pi}_H \right)\curlbf \mathbf{q}
\end{align*}
for some $\mathbf{q} \in \Hzerocurl$. With this, consider the local error contributions $\mathbf{e}_T: = (\id-\operatorname{\pi}_H)\curlbf(\eta_T \mathbf{q})$ which vanish in $\mathrm{N}^{m}(T)$. Hence,
\begin{equation*}
a(\mathbf{e}_T, \left(\operatorname{\mathcal{C}}_T-\operatorname{\mathcal{C}}_{T}^m\right) \mathbf{v} ) = a(\mathbf{e}_T, \operatorname{\mathcal{C}}_T \mathbf{v}) = (\mathbf{A}^{-1} \mathbf{e}_T,  \mathbf{v})_{L^2(T)} = 0,
\end{equation*}
and 
\begin{equation*}
\mathbf{e} - \mathbf{e}_T = \left( \id - \operatorname{\pi}_H \right)\curlbf\left(\left(1-\eta_T\right)\mathbf{q}\right) \in W_{\operatorname{div}\hspace{-1pt}0}(N^{m+3}(T)).
\end{equation*}
Collecting the above identities and applying Lemma \ref{lem-expo-deca-phi-phi-l}, we find that
\begin{eqnarray*}
\vertiii{\mathbf{e}}^2 
&=& \sum_{T \in \mathcal{T}_H} a(\mathbf{e}, \left(\operatorname{\mathcal{C}}_T-\operatorname{\mathcal{C}}_{T}^m\right) \mathbf{v} )  
\,\,=\,\, \sum_{T \in \mathcal{T}_H} a(\mathbf{e}-\mathbf{e}_T, \left(\operatorname{\mathcal{C}}_T-\operatorname{\mathcal{C}}_{T}^m\right) \mathbf{v} ) \\
&= & \sum_{T \in \mathcal{T}_H} a(\left(\id - \operatorname{\pi}_H \right)\curlbf\left(\left(1-\eta_T\right)\mathbf{q}\right), \left(\operatorname{\mathcal{C}}_T-\operatorname{\mathcal{C}}_{T}^m\right) \mathbf{v} ) \\
&\lesssim& \theta^m  \sum_{T \in \mathcal{T}_H} \vertiii{ \mathbf{e} }_{N^{m+5}(T)} \vertiii{ \mathbf{v} }_{T} \\
&\lesssim & \sqrt{C_m} \,\, \theta^m \,\,
\vertiii{ \mathbf{e} }\,\, \vertiii{ \mathbf{v} }.
\end{eqnarray*}
Here, introducing $\chi_{m,T}\in L^{\infty}(\Omega)$ with $\left.\chi_{m,T}\right\rvert_{N^{m+5}(T)} = 1$ and $\left.\chi_{m,T}\right\rvert_{\Omega\backslash N^{m+5}(T)} = 0$, we have used the Cauchy--Schwarz inequality together with the bound $\sum_{T\in \mathcal{T}_H} \vertiii{\mathbf{e}}_{N^{m+5}(T)}^2 \leq C_m \vertiii{\mathbf{e}}^2$, where
\begin{align*}
    C_m := \left\|\sum_{T\in \mathcal{T}_H} \chi_{m,T}\right\|_{L^{\infty}(\Omega)} \leq \frac{\max_{T\in \mathcal{T}_H}\lvert N^{m+5}(T)\rvert}{\min_{T\in\mathcal{T}_H}\lvert T\rvert} \leq C( m^d + 1)
\end{align*}
for some constant $C>0$ depending only on $\cuniformity,\cshape$, and $d$.
\end{proof}

\subsection{Proof of Theorem \ref{main-thm}}\label{sec-proof-main-thm}

We are now ready to prove the main result stated in Theorem \ref{main-thm}. For that, we combine the error estimate for the ideal multiscale method in Theorem \ref{thm-error-ideal-problem} with the decay results from Lemma \ref{lem-expo-deca-loca-erro}.

\begin{proof}[Proof of Theorem \ref{main-thm}]
We observe that the identity \eqref{sub-estimate-2-loc-method} readily follows from the fact that $\div \mathbf{u} = -f$ and $\div \mathbf{u}_{H,m}^{\mathrm{ms}} = -\operatorname{P}_{\hspace{-1pt}H} \hspace{-1pt}f$; see the paragraph below Definition \ref{def-loc-mul-ap}. 

In order to show the estimate \eqref{main-estimate-loc-method}, we start as in the proof of \cite[Theorem 16]{hellman2016multiscale} and introduce the function $\tilde{\mathbf{u}}_{H,m}^{\mathrm{ms}} := (\id - \mathcal{C}^m)\pi_H(\mathbf{u}_{H}^{\mathrm{ms}} ) \in \VHkmsloc$, where $\mathbf{u}_{H}^{\mathrm{ms}} \in \VHkms$ denotes the solution to the ideal method \eqref{ideal-mult-prob}. Noting that 
\begin{align*}
    \div  \tilde{\mathbf{u}}_{H,m}^{\mathrm{ms}} = \div\left(\pi_H(\mathbf{u}_{H}^{\mathrm{ms}} )\right) = P_H(\div \mathbf{u}_{H}^{\mathrm{ms}}) = P_H(-P_H f) = -P_H f = \div  \mathbf{u}_{H,m}^{\mathrm{ms}},
\end{align*}
we can use the Galerkin orthogonality for the exact solution $\mathbf{u}$ and the multiscale approximation $\mathbf{u}_{H,m}^{\mathrm{ms}}$ on $\Hzerodivzero \cap \VHkmsloc$ to find that
\begin{eqnarray*}
a( \mathbf{u} - \mathbf{u}_{H,m}^{\mathrm{ms}}, \mathbf{u} - \mathbf{u}_{H,m}^{\mathrm{ms}} ) &=&
a( \mathbf{u} - \mathbf{u}_{H,m}^{\mathrm{ms}}, \mathbf{u} - \mathbf{u}_{H,m}^{\mathrm{ms}} )
+ a( \mathbf{u} - \mathbf{u}_{H,m}^{\mathrm{ms}}, \mathbf{u}_{H,m}^{\mathrm{ms}} - \tilde{\mathbf{u}}_{H,m}^{\mathrm{ms}} ) \\
&=& a( \mathbf{u} - \mathbf{u}_{H,m}^{\mathrm{ms}}, \mathbf{u} - \tilde{\mathbf{u}}_{H,m}^{\mathrm{ms}} ),
\end{eqnarray*}
and hence, using that $\tilde{\mathbf{u}}_{H,m}^{\mathrm{ms}} - \mathbf{u}_{H}^{\mathrm{ms}} = (\mathcal{C} - \mathcal{C}^m)\pi_H(\mathbf{u}_{H}^{\mathrm{ms}} )$, we obtain
\begin{align}\label{err pf bd1}
\begin{split}
\vertiii{  \mathbf{u} - \mathbf{u}_{H,m}^{\mathrm{ms}} } &\le 
\vertiii{  \mathbf{u} - \tilde{\mathbf{u}}_{H,m}^{\mathrm{ms}} }
\,\,\, \le\,\,\,
\vertiii{  \mathbf{u} - \mathbf{u}_{H}^{\mathrm{ms}} }
+ \vertiii{ (\mathcal{C} - \mathcal{C}^m)\pi_H(\mathbf{u}_{H}^{\mathrm{ms}} )} \\
&\overset{\eqref{main-estimate-ideal-method}}{\lesssim} H \| f - \operatorname{P}_{\hspace{-1pt}H}f \|_{L^2(\Omega)} + \vertiii{ (\mathcal{C} - \mathcal{C}^m)\pi_H(\mathbf{u}_{H}^{\mathrm{ms}})}.
\end{split}
\end{align}
In view of Lemma \ref{lem-expo-deca-loca-erro}, we have that
\begin{align}\label{err pf bd2}
     \vertiii{ (\mathcal{C} - \mathcal{C}^m)\pi_H(\mathbf{u}_{H}^{\mathrm{ms}})}  \lesssim \left(m^{\frac{d}{2}}+1\right)\theta^m \vertiii{\pi_H(\mathbf{u}_{H}^{\mathrm{ms}})}  
     \lesssim m^{\frac{d}{2}}\theta^m \|\mathbf{u}_{H}^{\mathrm{ms}}\|_{\Hdiv},
\end{align}
where the final inequality follows from global stability of $\pi_H$ implied by \eqref{h-div-inte-stab}. In view of the bounds \eqref{err pf bd1} and \eqref{err pf bd2}, and noting that by \eqref{u stabbd} and Theorem \ref{thm-error-ideal-problem} we have that
\begin{align*}
    \|\mathbf{u}_{H}^{\mathrm{ms}}\|_{\Hdiv} \leq \|\mathbf{u}  - \mathbf{u}_{H}^{\mathrm{ms}}\|_{\Hdiv} + \|\mathbf{u} \|_{\Hdiv} \lesssim \|f\|_{L^2(\Omega)},
\end{align*} 
we conclude that
\begin{align*}
    \vertiii{  \mathbf{u} - \mathbf{u}_{H,m}^{\mathrm{ms}} } \lesssim H \| f - \operatorname{P}_{\hspace{-1pt}H}f \|_{L^2(\Omega)} + m^{\frac{d}{2}}\theta^m\|f\|_{L^2(\Omega)}.
\end{align*}
Finally, to prove the remaining bound \eqref{sub-estimate-1-loc-method} we can use the inf-sup stability in Lemma \ref{lemma:inf-sup-VHmsloc} and, with the same arguments as in the proof of Theorem \ref{thm-error-ideal-problem}, we obtain
\begin{align*}
    \|p_{H,m}-p\|_{L^2(\Omega)} \,\,\leq\,\, \|p_{H,m}-\operatorname{P}_{\hspace{-1pt}H}p \|_{L^2(\Omega)} + \|p - \operatorname{P}_{\hspace{-1pt}H} p \|_{L^2(\Omega)} \,\,\lesssim\,\, \vertiii{\mathbf{u} - \mathbf{u}_{H,m}^{\mathrm{ms}}} + \|p - \operatorname{P}_{\hspace{-1pt}H} p\|_{L^2(\Omega)},
\end{align*}
which completes the proof.
\end{proof}

\section{Numerical Experiments}\label{Sec: num exp}
In this section, we present numerical experiments to illustrate the theoretical results. Since the exact solution $(\mathbf{u},p)\in \Hzerodivomega{\Omega}\times L^2_0(\Omega)$ to \eqref{eqn-mixed-form-var} is typically unknown, we define the relative approximation errors for the velocity $\mathbf{u}$ and the pressure $p$ by
\begin{equation*}
\operatorname{error}_{m}(\mathbf{u}):=\frac{\vertiii{\mathbf{u}_{H, m}^{\mathrm{ms}}-\mathbf{u}_h^{\mathrm{ref}}}}{\vertiii{\mathbf{u}_h^{\mathrm{ref}}}}, \qquad \operatorname{error}_{m}(p):=\frac{\|p_{H,m} - p_h^{\mathrm{ref}}\|_{\Ltwo}}{\|p_h^{\mathrm{ref}}\|_{\Ltwo}},
\end{equation*}
where a computed reference solution $(\mathbf{u}_h^{\mathrm{ref}}, p_h^{\mathrm{ref}})$ takes the place of the unknown true solution $(\mathbf{u},p)$. The reference solution is a fine approximation of \eqref{eqn-mixed-form-var} obtained by using lowest-order Raviart--Thomas elements on a fine mesh with mesh size $h\ll H$, which will be specified in each experiment. The code accompanying this manuscript is available at \href{https://github.com/HaoLi-HL/LOD-Hdiv}{https://github.com/HaoLi-HL/LOD-Hdiv}.

\subsection{Experiment 1: Convergence test}\label{sec-example-1}
For the first experiment, we consider the problem \eqref{elliptic-equation-strong-form} on the unit square $\Omega:=(0,1)^2$ in dimension $d=2$. 
To introduce multiscale features, we define the diffusion coefficient $\mathbf{A}$ and the source term $f$ as 
\begin{align*}
  \mathbf{A}(x,y) := \kappa(x,y)\,\mathbf{I}_2,\qquad  f(x, y):=2 \pi^2 \cos (\pi x) \cos (\pi y)\qquad \text{for }\; (x,y)\in \Omega,
\end{align*}
where $\mathbf{I}_2\in \mathbb{R}^{2\times 2}$ denotes the identity matrix and $\kappa:\Omega\rightarrow \mathbb{R}$ denotes the checkerboard coefficient displayed in Figure \ref{rand_coef} that takes the value $1$ in the black squares and $0.001$ in the white squares. The size of each checkerboard square (i.e., the block size) is chosen to be twice the size of the fine mesh, ensuring that the reference solution resolves the coefficients. The magnitude of a typical reference solution is shown in Figure \ref{u-ref-norm}, clearly illustrating the multiscale structure. The reference solution is computed on a uniform fine triangular mesh with mesh size $h=\sqrt{2} \cdot 2^{-7}$. 
\begin{figure}[h!]
\centering
\begin{subfigure}[t]{0.49\textwidth}
\centering
\includegraphics[scale=0.5]{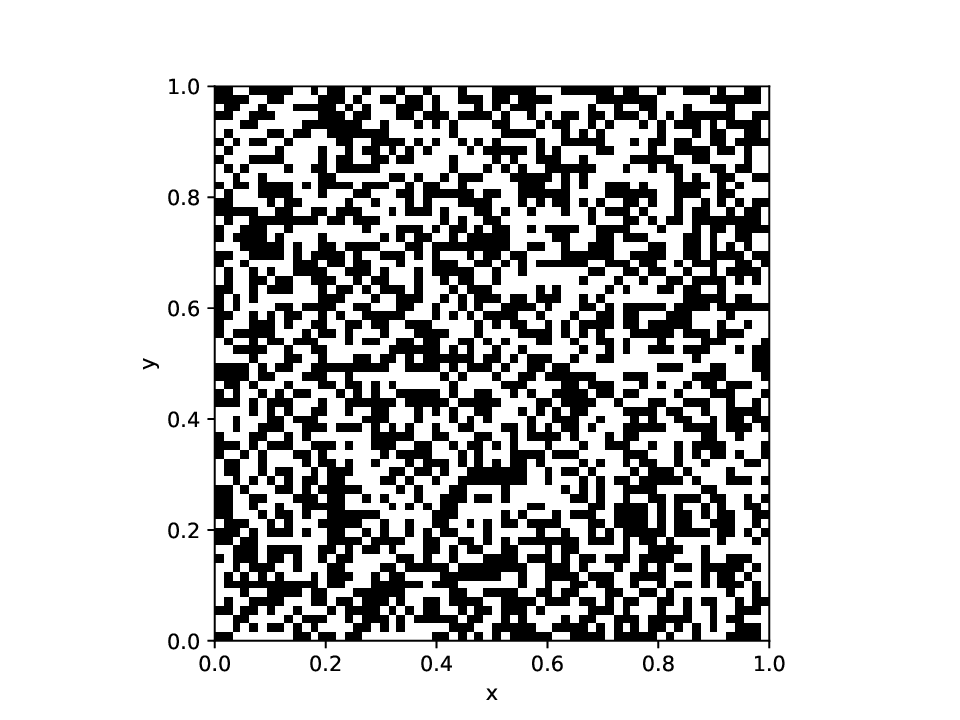}
\caption{Checkerboard coefficient $\kappa:(0,1)^2 \rightarrow \mathbb{R}$ with \\values $1$ (black) and $0.001$ (white)}
\label{rand_coef}
\end{subfigure}
\begin{subfigure}[t]{0.49\textwidth}
\centering
\includegraphics[scale=0.5]{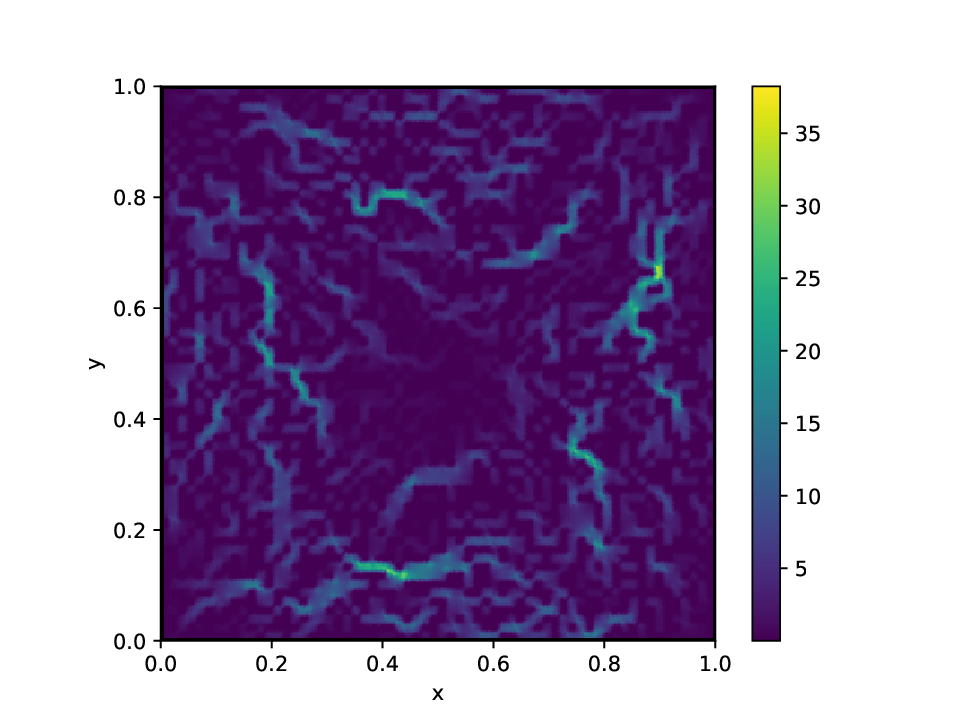}
\caption{Magnitude of the reference solution $\mathbf{u}_h^{\mathrm{ref}}$}
\label{u-ref-norm}
\end{subfigure}
\caption{Plots of the coefficient and the magnitude of the reference flux for Experiment 1.}
\end{figure}

The results are shown in Figure \ref{fig-convergence-rate}. For the approximation of the velocity in the energy norm, we observe second-order convergence for $k=0$ and third-order convergence for $k=1$. For the approximation of the pressure in the $L^2(\Omega)$-norm, we observe first-order convergence for both $k=0$ and $k = 1$. These observations are consistent with Theorem \ref{main-thm}. Also note that the number of layers required to observe the quadratic convergence for the velocity is larger than the number of layers required for linear convergence in the pressure. This is consistent with the theory since quadratic convergence constrains $m$ by $m^{d / 2} \, \theta^{m}\lesssim H^2$, whereas linear convergence only requires $m^{d / 2} \, \theta^{m}\lesssim H$.
\begin{figure}[h!]
\centering
\begin{subfigure}{0.49\textwidth}
\centering
\includegraphics[scale=0.3]{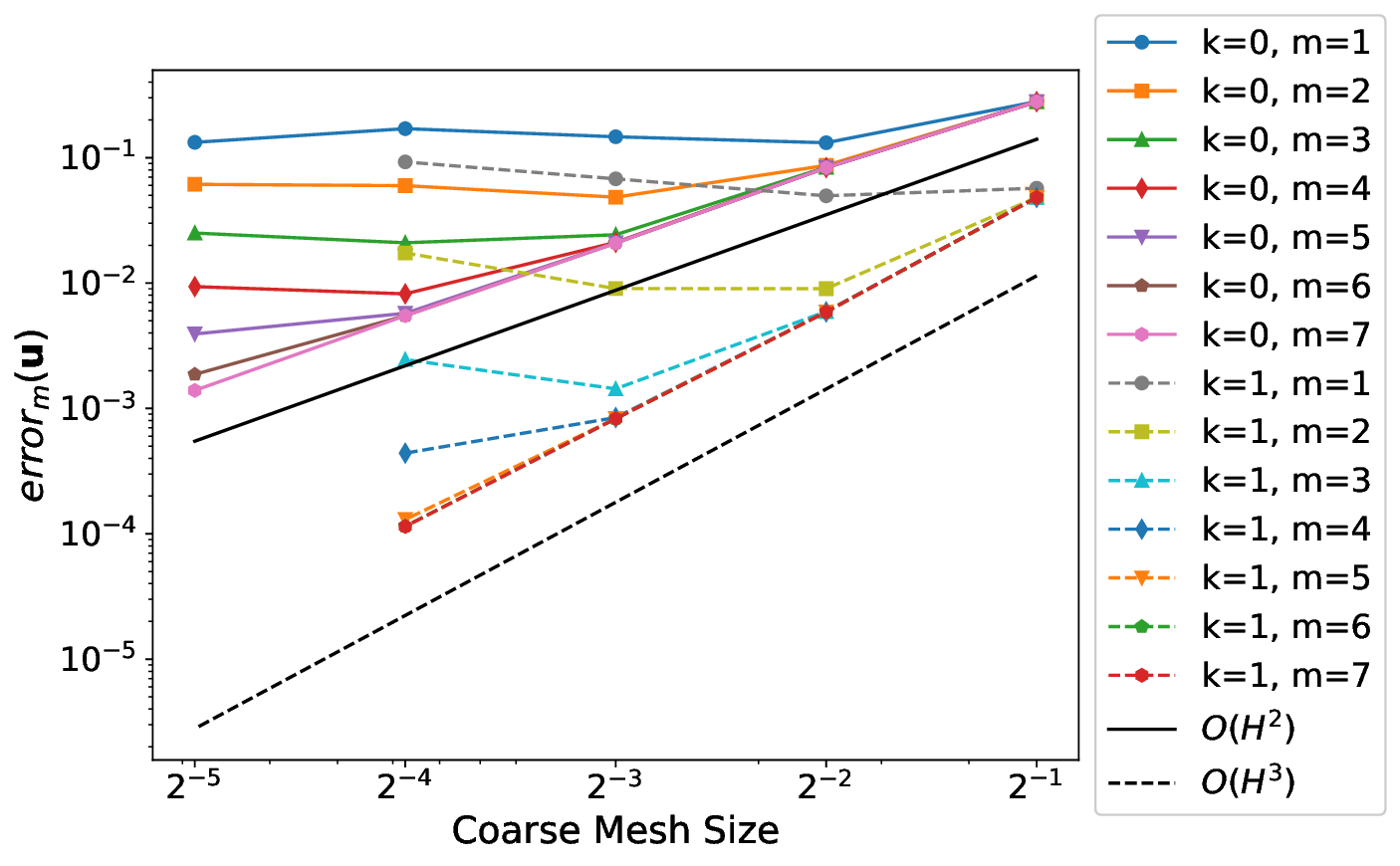}
\caption{Convergence test for flux $\mathbf{u}_{H, m}^{\mathrm{ms}}$}
\end{subfigure}
\begin{subfigure}{0.49\textwidth}
\centering
\includegraphics[scale=0.3]{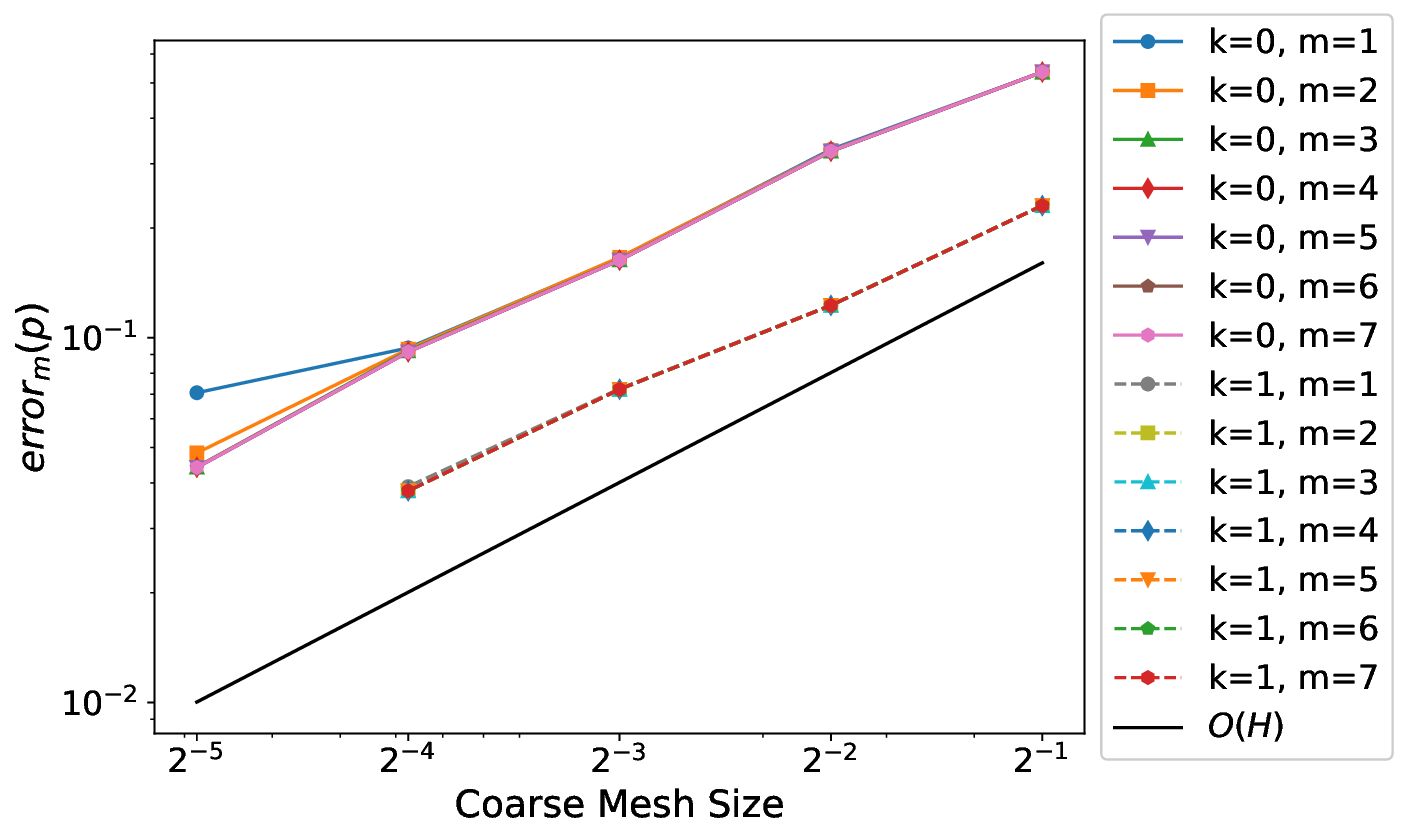}
\caption{Convergence test for pressure $p_{H,m}$}
\end{subfigure}
\caption{Convergence Test. Relative errors for the LOD approximations of Experiment 1.}
\label{fig-convergence-rate}
\end{figure}

\subsection{Experiment 2: SPE10-85}

For the second experiment, we consider the problem \eqref{elliptic-equation-strong-form} on the rectangle $\Omega := (0,\frac{6}{5})\times (0,\frac{11}{5})$ in dimension $d=2$. We use the 85th permeability layer from Model 2 of the SPE10 benchmark data set \cite{christie2001tenth} of the Society of Petroleum Engineers (SPE) as a highly heterogeneous realistic test coefficient $\tilde{\kappa}:\Omega\rightarrow\mathbb{R}$. The permeability coefficient $\mathbf{A} := \tilde{\kappa}\mathbf{I}_2$ is provided on a uniform $60 \times 220$ rectangular grid and is visualized in Figure~\ref{permeability-field-kx}. 

\begin{figure}
\centering
\begin{subfigure}[t]{0.49\textwidth}
\centering
\includegraphics[trim={0 4.5cm 0 3cm},clip,scale=0.4]{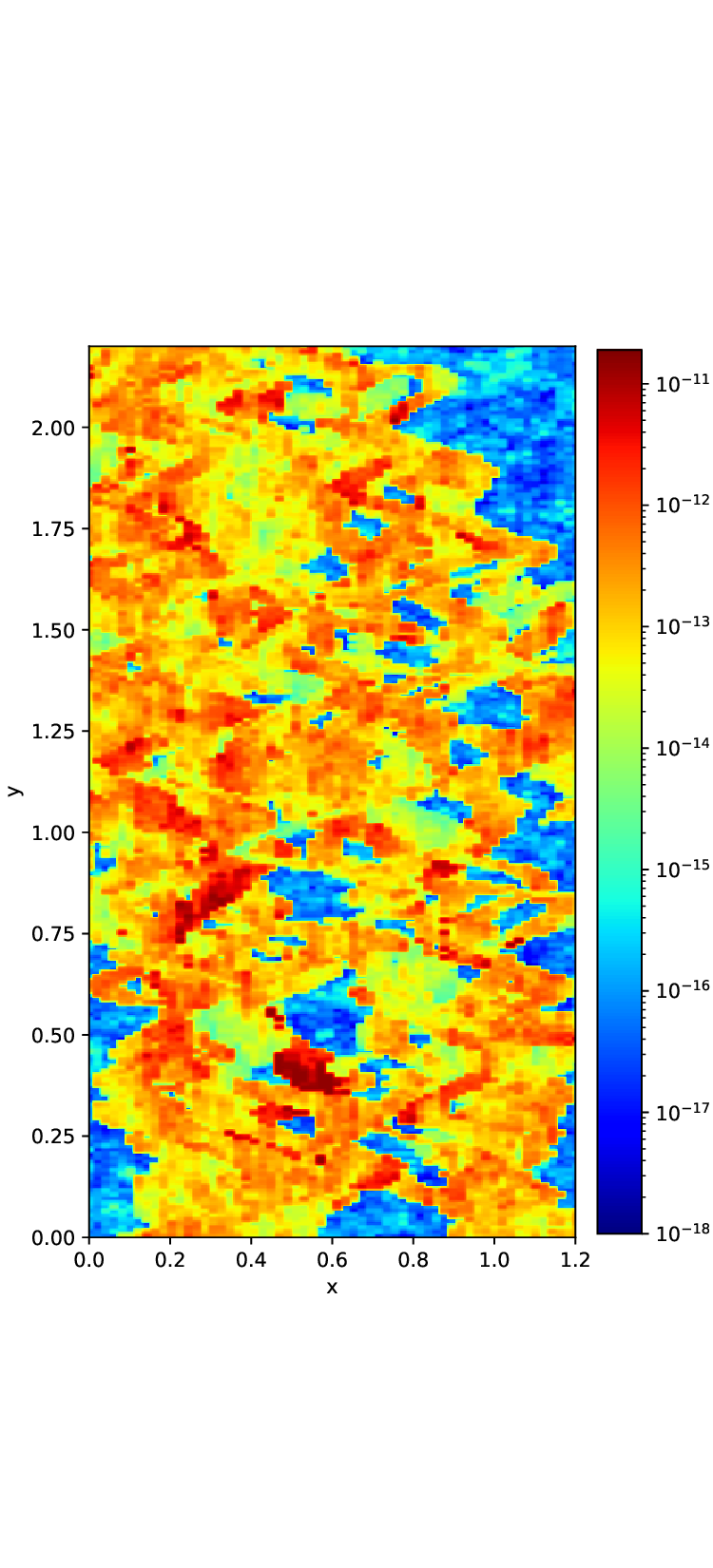}
\caption{SPE10-85 coefficient $\tilde{\kappa}:(0,\frac{6}{5})\times (0,\frac{11}{5})\rightarrow \mathbb{R}$}
\label{permeability-field-kx}
\end{subfigure}
\begin{subfigure}[t]{0.49\textwidth}
\centering
\includegraphics[trim={0 2cm 0 3cm},clip,scale=0.35]{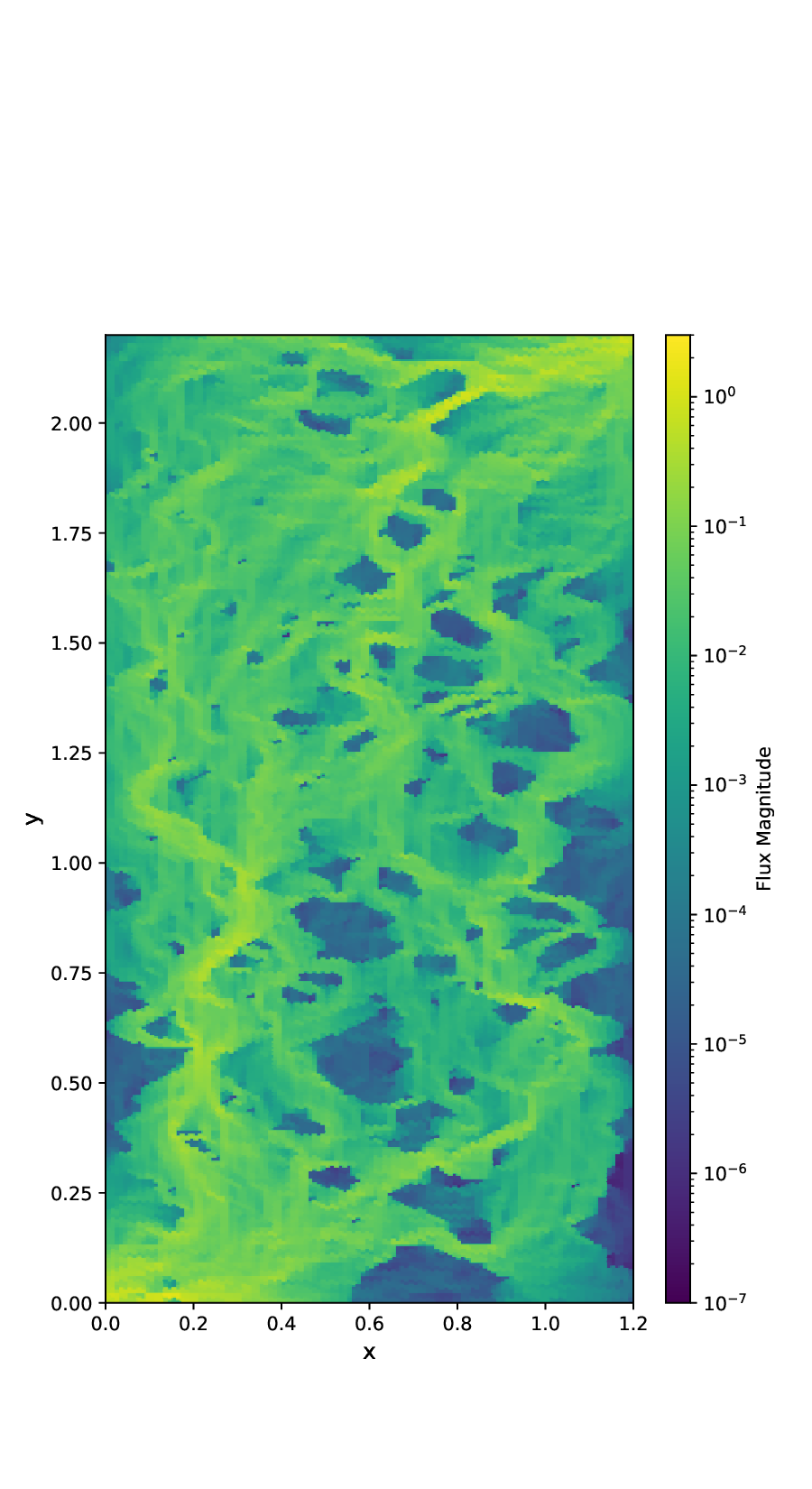}
\caption{Magnitude of the reference solution $\mathbf{u}_h^{\mathrm{ref}}$}
\end{subfigure}
\caption{Plots of the coefficient and the magnitude of the reference flux for Experiment 2.}
\end{figure}

The fine mesh is generated by subdividing each rectangle of the uniform $60 \times 220$ quadrilateral mesh, which is aligned with the permeability data, into two triangles. Similarly, the coarse mesh is constructed by dividing each rectangle of a uniform $6 \times 22$ quadrilateral mesh into two triangles. The resulting fine and coarse meshes are conforming, with each coarse element further partitioned into $10 \times 10$ fine elements. The quasi-singular source term $f$ is set to $100$ in the lower-left element and $-100$ in the upper-right element of the $60 \times 220$ quadrilateral mesh, representing a discretized Dirac delta function to model production wells in hydrological simulations. Note that a delta function source formally belongs to $W^{-s, 2}(\Omega)$ for $s>\frac{d}{2}$ rather than $L^{2}(\Omega)$, which is the setting of this work. 
Although the discrete load $f_h$ used in this experiment is piecewise constant on the
fine grid and therefore belongs to the discrete $L^{2}$-space, it represents a
discretization of a Dirac-type well source and its effective support is far below the resolution of the coarse mesh. From the coarse-scale perspective, the load is thus highly localized (quasi-singular), and the assumptions underlying the standard LOD construction for
$L^{2}$-right-hand sides do not apply (in fact, the error contribution $\| f_h -\operatorname{P}_H \hspace{-2pt}f_h\|_{\Ltwo}$ becomes large in this case). 
Following \cite{Mal11}, and as in \cite{hellman2016multiscale}, we therefore employ localized
source correctors to incorporate the fine-scale influence of the unresolved load
into the multiscale space. These localized source corrections are computed on $\ell$-coarse-layer patches with $\ell=m+1$, where $m$ is the number of coarse layers used in the multiscale correction. See \cite{hellman2016multiscale} for further details on source correction procedures. The flux solutions are plotted in Figure \ref{fig-spe-10}. A reference solution $\mathbf{u}_h^{\mathrm{ref}}$ was computed on the fine mesh. The relative errors $\operatorname{error}_{m}(\mathbf{u})$ for $m=2,3,4$ are equal to 2.42e-2, 1.07e-2, and 1.29e-3, respectively. The corresponding numerical solutions are shown in Figures~\labelcref{spe10-m2-l3,spe10-m3-l4,spe10-m4-l5} and show that LOD approximations capture the correct fine-scale behavior with very high accuracy.

\begin{figure}[h!]
\centering
\begin{subfigure}{0.245\textwidth}
\centering
\includegraphics[width=\linewidth,trim={0 2cm 0.7cm 3cm},clip]{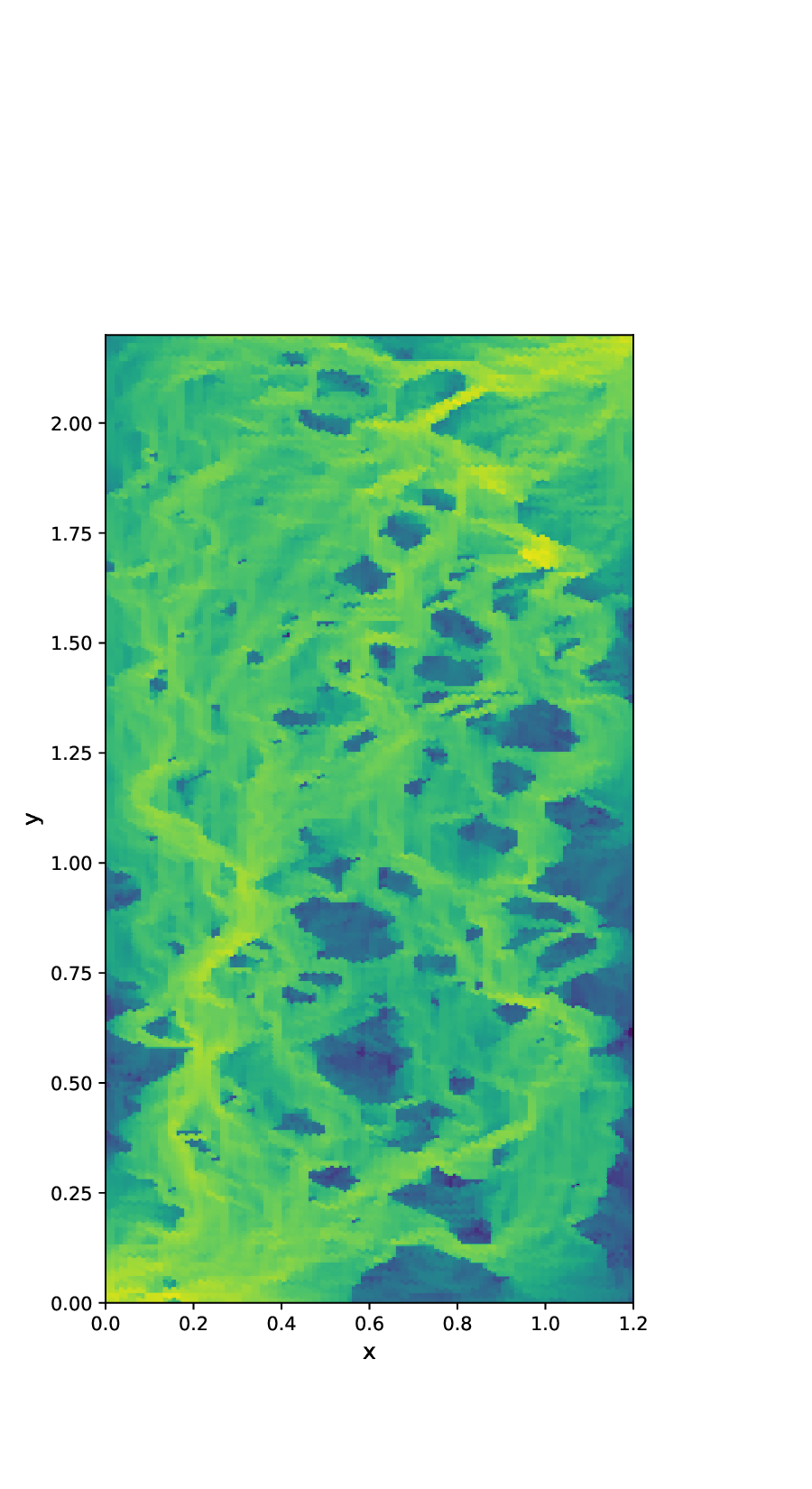}
\caption{$m=2$, $\ell =3$}
\label{spe10-m2-l3}
\end{subfigure}
\begin{subfigure}{0.245\textwidth}
\centering
\includegraphics[width=\linewidth,trim={0 2cm 0.7cm 3cm},clip]{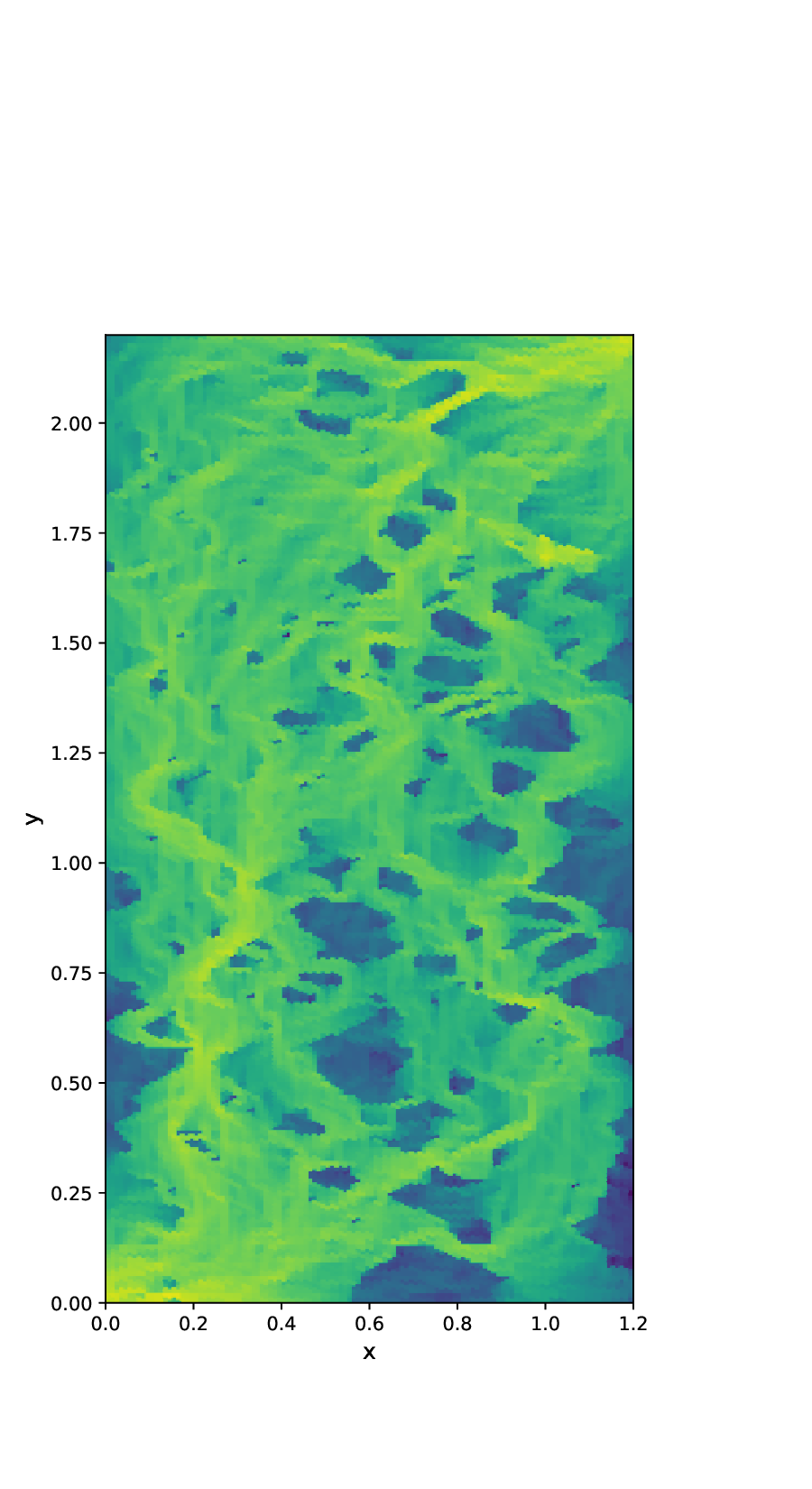}
\caption{$m=3$, $\ell =4$}
\label{spe10-m3-l4}
\end{subfigure}
\begin{subfigure}{0.245\textwidth}
\centering
\includegraphics[width=\linewidth,trim={0 2cm 0.7cm 3cm},clip]{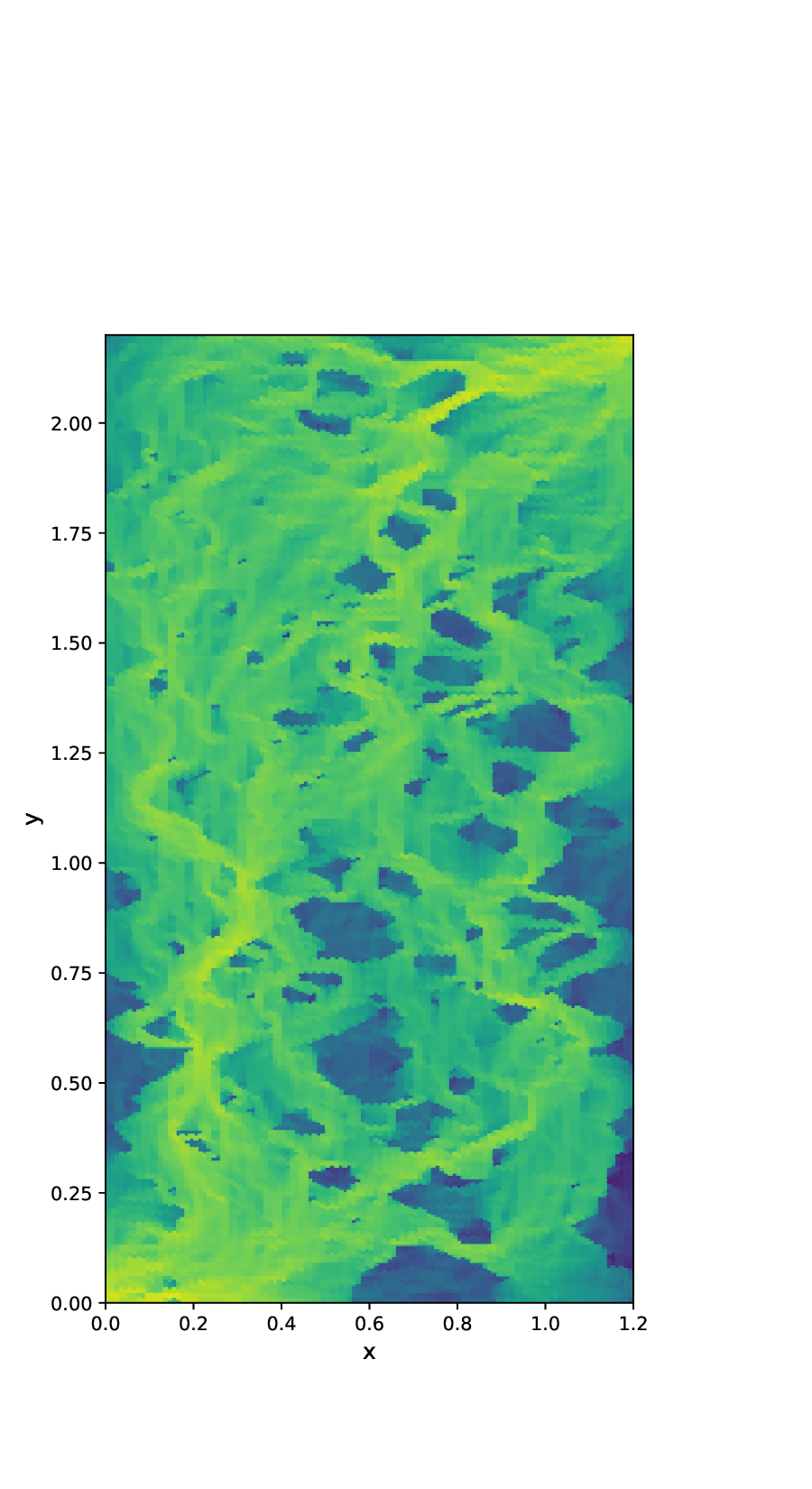}
\caption{$m=4$, $\ell =5$}
\label{spe10-m4-l5}
\end{subfigure}
\begin{subfigure}{0.245\textwidth}
\centering
\includegraphics[width=\linewidth,trim={0 2cm 0.7cm 3cm},clip]{figs/anisotropic_ref.eps}
\caption{Reference solution}
\label{spe10-ref-sol}
\end{subfigure}
\caption{SPE10-85 test. Figure~\ref{spe10-ref-sol} shows the magnitude of the reference flux solution.  Figures~\labelcref{spe10-m2-l3,spe10-m3-l4,spe10-m4-l5} display the magnitudes of the multiscale flux solutions for $m=2,3,4$ with $\ell=m+1$, respectively.}
\label{fig-spe-10}
\end{figure}

\subsection{Experiment 3: Comparison test}\label{sec-example-3}
In this experiment, we compare the LOD method using the stable interpolator presented in this paper with the method employing the nodal Raviart--Thomas interpolator from \cite{hellman2016multiscale} to demonstrate the stability of the new method.

We consider the unit square $\Omega := (0,1)^2$ with the diffusion coefficient defined as
$$
\mathbf{A}(x, y):= \begin{cases}\exp (1) & \text { if } y<\frac{1}{2} \text { or } (x, y) \in\left[\frac{1}{2}-2^{-5}, \frac{1}{2}+2^{-5}\right] \times\left[\frac{1}{2}, \frac{1}{2}+2^{-5}\right], \\ \frac{1}{10} & \text { otherwise}, \end{cases}
$$
and the source term chosen as
$$
f(x, y):= \begin{cases}-1 & \text { if } y<\frac{1}{2}, \\ 1 & \text { otherwise }\end{cases}
$$
for $(x,y)\in\Omega$. We fix $H = 2^{-2}$ so that $f$ is resolved on the coarse scale. Consequently, $f \in Q_H^0$, and all errors stem from localization (see Theorem \ref{main-thm}). We set the resolution $h$ of the fine mesh to $h = 2^{-i}$ for $i = 5,6,7,8$. With $m = 2$, we compute the localized multiscale solution $\mathbf{u}_{H, m}^{\mathrm{ms}}(h)$ and the reference solution $\mathbf{u}_h$ for these values of $h$.
\begin{figure}[h!]
\centering
\includegraphics[width=0.7\linewidth]{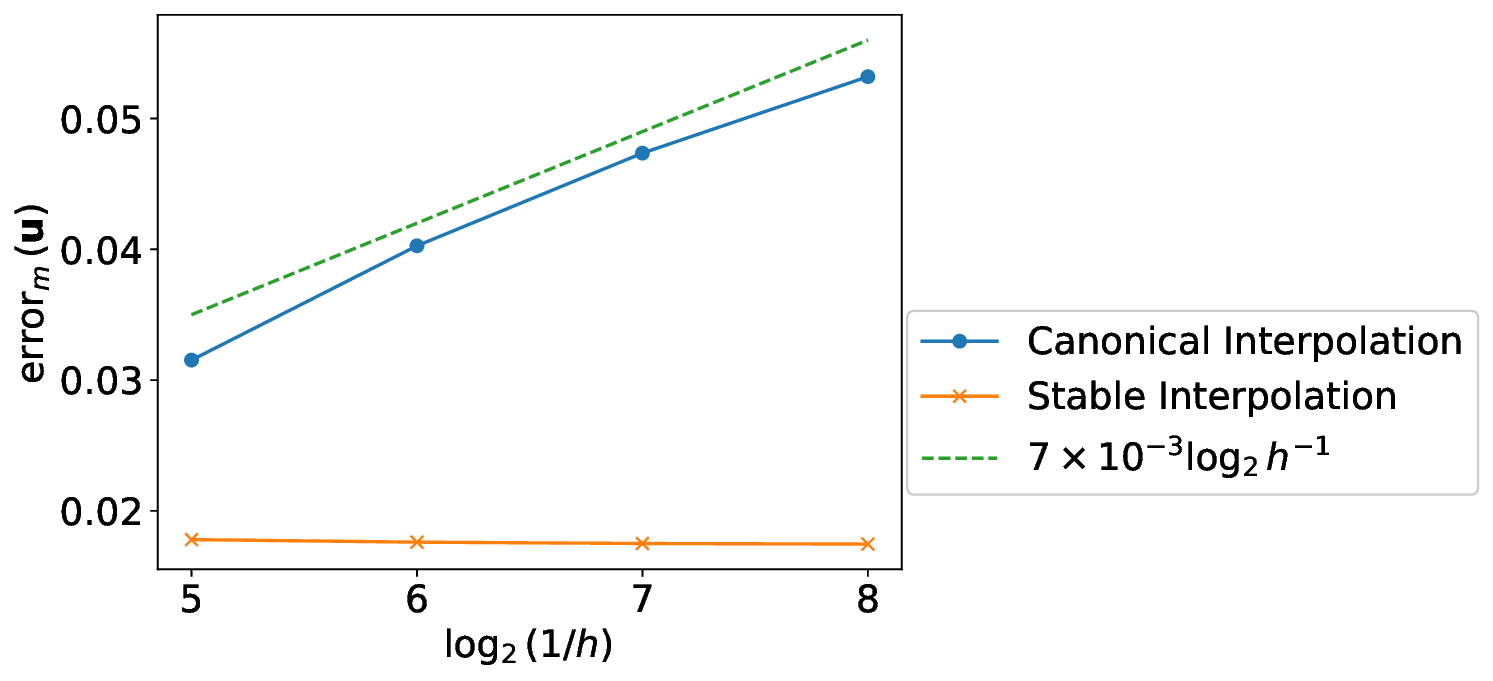}
\caption{Energy norm of the localization errors
of methods with different interpolators as $h$ decreases.}
\label{fig-instability-comparison}
\end{figure}
The energy norm of the error is plotted in Figure~\ref{fig-instability-comparison}. For this specific problem and range of $h$, the error of the method in \cite{hellman2016multiscale} increases as $h$ decreases, at a rate of $\log(h^{-1})$, as predicted by the theory. On the contrary, the error of the new method remains stable and there is no logarithmic pollution. This shows the computational advantage of the new method.

\subsection{Experiment 4: 3D convergence test}\label{sec-example-4}
We consider the problem \eqref{elliptic-equation-strong-form} on the unit cube $\Omega := (0,1)^3$ in dimension $d=3$. To introduce multiscale features, we define the diffusion coefficient $\mathbf{A}$ and the source term $f$ as
$$\mathbf{A}(x, y, z) := \kappa(x, y, z)\, \mathbf{I}_3, \qquad f(x, y, z) := 3 \pi^2 \cos(\pi x) \cos(\pi y) \cos(\pi z) \quad \text{for } (x, y, z) \in \Omega,$$
where $\mathbf{I}_3 \in \mathbb{R}^{3 \times 3}$ denotes the identity matrix, and $\kappa: \Omega \to \mathbb{R}$ is the checkerboard coefficient displayed in Figure \ref{rand_coef_3D}, taking the value $1$ in the black squares and $0.01$ in the white squares. The size of each checkerboard square (i.e., the block size) is chosen to be twice the fine mesh size, ensuring that the reference solution resolves the coefficients. The reference solution is computed on a uniform fine tetrahedral mesh with mesh size $h = \sqrt{3} \cdot 2^{-4}$. The LOD approximations are computed for $H = \sqrt{3} \cdot 2^{-j}$, $j=1, 2, 3$, with $m$ taken as 2, 3, 3, respectively. The results are shown in Figure \ref{fig-convergence-rate-3d}. We observe second-order convergence for the approximation of the velocity in the energy norm and first-order convergence for the approximation of the pressure in the $L^2(\Omega)$-norm, as predicted by Theorem \ref{main-thm}.

\begin{figure}[h!]
\centering
\includegraphics[scale=0.45]{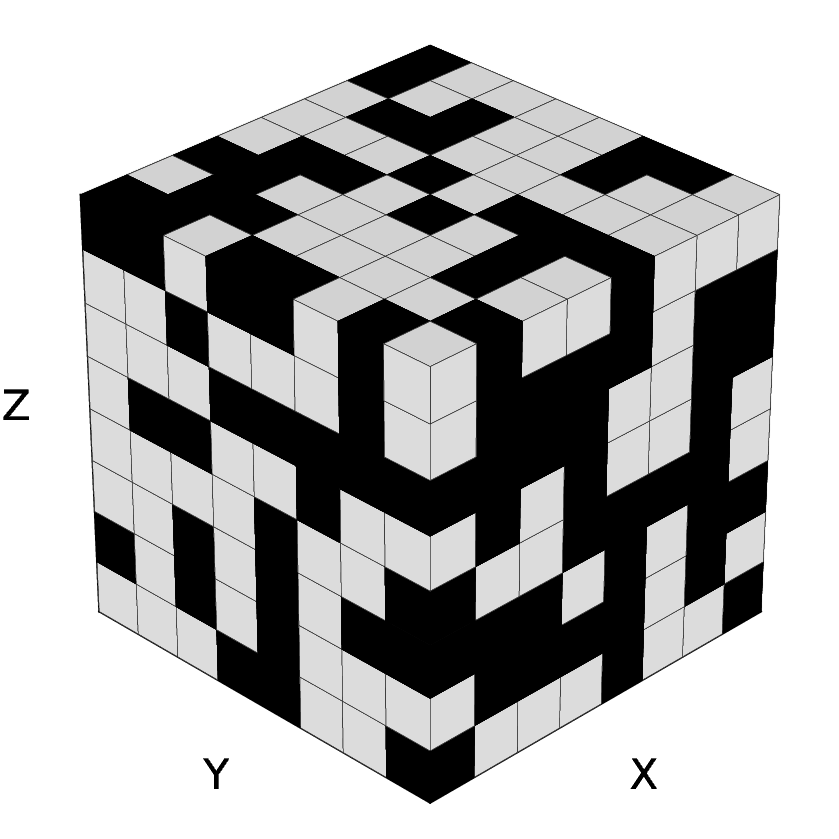}
\caption{Checkerboard coefficient $\kappa:(0,1)^3 \rightarrow \mathbb{R}$ with values $1$ (black) and $0.01$ (white)}
\label{rand_coef_3D}
\end{figure}

\begin{figure}[h!]
\centering
\begin{subfigure}{0.49\textwidth}
\centering
\includegraphics[scale=0.4]{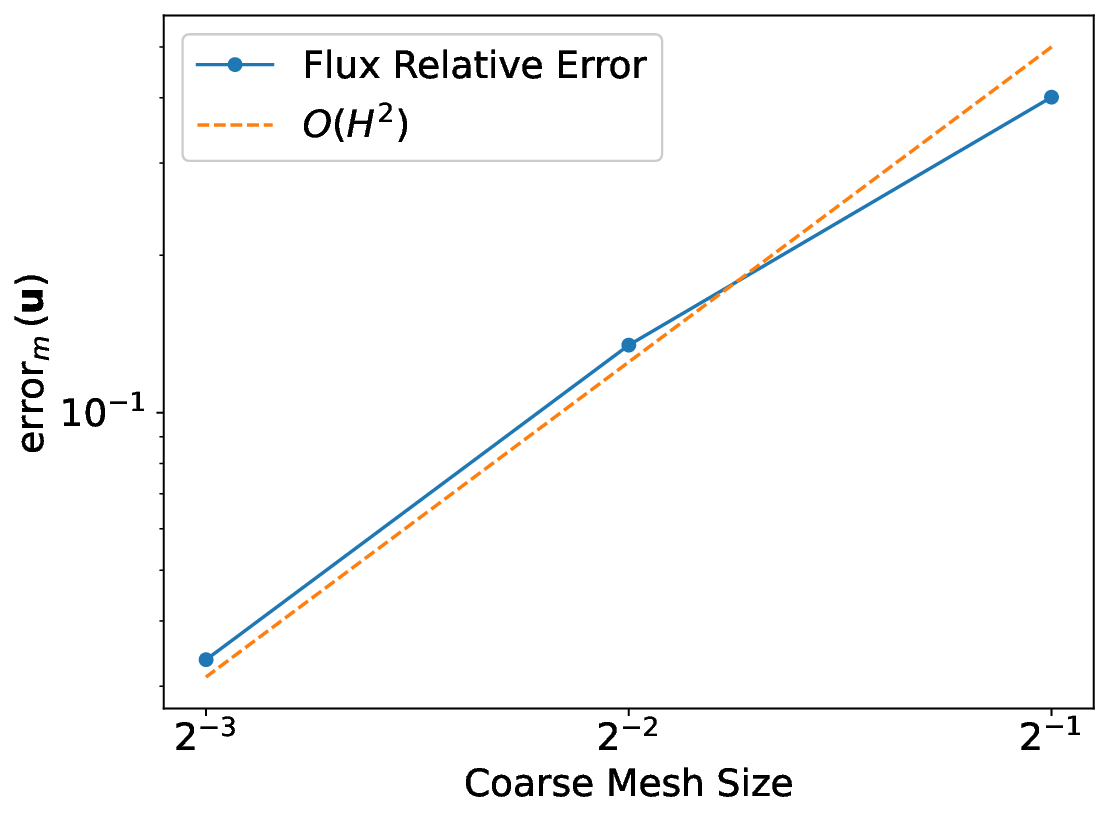}
\caption{3D convergence test for flux $\mathbf{u}_{H, m}^{\mathrm{ms}}$}
\end{subfigure}
\begin{subfigure}{0.49\textwidth}
\centering
\includegraphics[scale=0.4]{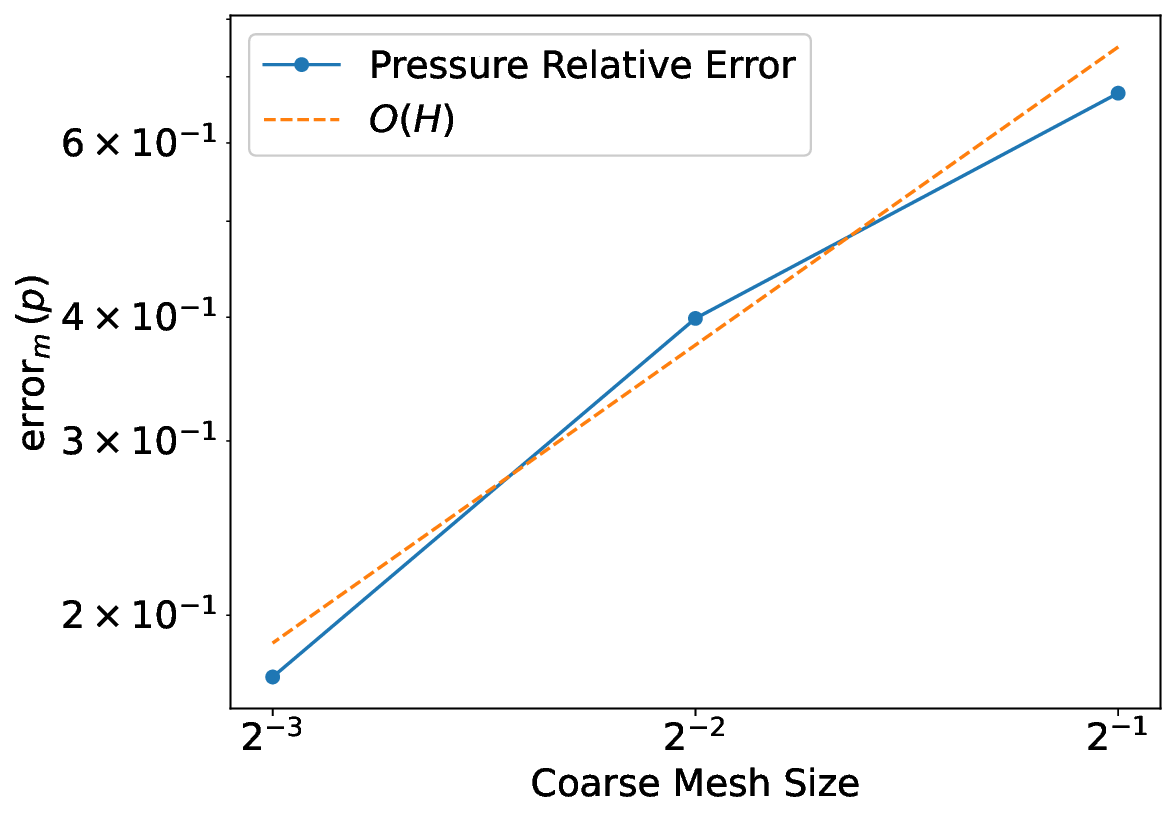}
\caption{3D convergence test for pressure $p_{H,m}$}
\end{subfigure}
\caption{3D Convergence Test. Relative errors for the LOD approximations of Experiment 4.}
\label{fig-convergence-rate-3d}
\end{figure}

\section*{Acknowledgement}
Patrick Henning acknowledges the support by the Deutsche Forschungsgemeinschaft (DFG, German Research Foundation) through the project grant 574914916. Hao Li was supported by internal grants from The Hong Kong Polytechnic University (Project IDs: P0038443 and P0058870).

\bibliographystyle{abbrv}
\bibliography{refs}

\begin{appendix}
\section{Inf-sup stability in classical Raviart--Thomas spaces}\label{appendix:A}

In the following, we give a short proof of the classical inf-sup stability in \eqref{inf-sup-classical-spaces}. Let $q_H\in Q_H^k \cap L^2_0(\Omega)\backslash\{0\}$ and let $\varphi_{q_H}\in H^1(\Omega)\cap L^2_0(\Omega)$ be the unique solution to the Neumann problem
\begin{align*}
    (\nabla \varphi_{q_H},\nabla \psi)_{L^2(\Omega)} = -(q_H,\psi)_{L^2(\Omega)}\qquad\mbox{for all } \psi\in H^1(\Omega).
\end{align*}
Then, $\mathbf{\tilde v}_{q_H} := \nabla \varphi_{q_H}\in \Hzerodiv\backslash\{\mathbf{0}\}$ with $\div(\mathbf{\tilde v}_{q_H})=q_H$. Setting $\mathbf{v}_{q_H} := \pi_H \mathbf{\tilde v}_{q_H}\in V_H^k\backslash\{\mathbf{0}\}$ for any stable and commuting quasi-interpolation operator $\pi_H : \Hzerodiv \rightarrow V_H^k$ (in the sense of Assumption \ref{assu-stab-inte}), we have that 
\begin{align*}
    \div(\mathbf{v}_{q_H}) = \div(\pi_H \mathbf{\tilde v}_{q_H}) = \PH(\div(\mathbf{\tilde v}_{q_H})) = \PH q_H = q_H,
\end{align*}
and hence,
\begin{align*}
    \sup _{\mathbf{v}_H \in V_H^k\backslash\{\mathbf{0}\} } \frac{b(\mathbf{v}_H, q_H)}{\|q_H\|_{\Ltwo}\|\mathbf{v}_H\|_{\Hdiv}} \geq \frac{b(\mathbf{v}_{q_H}, q_H)}{\|q_H\|_{\Ltwo}\|\mathbf{v}_{q_H}\|_{\Hdiv}} = \frac{\|q_H\|_{L^2(\Omega)}}{\|\mathbf{v}_{q_H}\|_{\Hdiv}}.
\end{align*}
Next, we note that $\|\div(\mathbf{v}_{q_H})\|_{L^2(\Omega)} = \|q_H\|_{L^2(\Omega)}$ and, using global stability of $\pi_H$ implied by \eqref{h-div-inte-stab}, we find that there exists a constant $C_{\mathrm{\pi_H,stab}}>0$ such that
\begin{align*}
    \|\mathbf{v}_{q_H}\|_{L^2(\Omega)}^2 = \|\pi_H \mathbf{\tilde v}_{q_H}\|_{L^2(\Omega)}^2 &\leq C_{\mathrm{\pi_H,stab}}\left( \|\mathbf{\tilde v}_{q_H}\|_{L^2(\Omega)}^2 + \|\div(\mathbf{\tilde v}_{q_H})\|_{L^2(\Omega)}^2\right)\\ &=C_{\mathrm{\pi_H,stab}}\left( \frac{\lvert(q_H, \varphi_{q_H})_{L^2(\Omega)}\rvert^2}{\|\mathbf{\tilde v}_{q_H}\|_{L^2(\Omega)}^2} + \|q_H\|_{L^2(\Omega)}^2 \right)\\ &\leq C_{\mathrm{\pi_H,stab}}(1+C_{\Omega}^2) \|q_H\|_{L^2(\Omega)}^2,
\end{align*}
where $C_{\Omega}>0$ is the Poincar\'e--Wirtinger constant for the domain $\Omega$ from \eqref{P-Wirtinger}. It follows that 
\begin{align*}
   \sup _{\mathbf{v}_H \in V_H^k\backslash\{\mathbf{0}\} } \frac{b(\mathbf{v}_H, q_H)}{\|q_H\|_{\Ltwo}\|\mathbf{v}_H\|_{\Hdiv}} \geq \left( 1+C_{\mathrm{\pi_H,stab}}(1+C_{\Omega}^2)\right)^{-\frac{1}{2}} =: \rho.
\end{align*}
This proves \eqref{inf-sup-classical-spaces}.

\section{Exponential decay of element correctors for $d=2$}
\label{appendix:section:decay-2D}

In the case $d=3$, the crucial result in the proof of the exponential decay is Lemma \ref{lem-stable-decomposition}, which allows us to express the divergence-free interpolation error $\mathbf{v}-\pi_H \mathbf{v} \in \Wdivzero$ as the curl of some $\mathbf{q} \in \Hzerocurl$ together with corresponding stability estimates. 
In the case $d=2$, a similar decomposition of $\mathbf{v}-\pi_H \mathbf{v} \in \Wdivzero$ holds using the 2D equivalent of the curl operator. However, the overall strategy needs further modifications since we can no longer rely on the Sch\"{o}berl splitting \cite{schoberl2008posteriori} which is only proved for $d=3$.

Before we start, we need to introduce the corresponding curl definitions for $d=2$. For $v \in H^1(\Omega)$ and $\mathbf{v} \in H^1(\Omega,\mathbb{R}^2)$, we recall the respective 2D curl-operators as
\begin{align*}
\curlbf \hspace{1pt}v \,:=\, 
\left(\begin{matrix} 
\,\,\,\,\,\partial_{x_2} v \\
- \partial_{x_1} v
\end{matrix}\right)
\qquad
\mbox{and}
\qquad
\curl \hspace{1pt}\mathbf{v}
 \,:=\, \partial_{x_1} \mathbf{v}_2 - \partial_{x_2} \mathbf{v}_1. 
\end{align*}
The corresponding function spaces are given by
\begin{align*}
H(\curlbf) \, &:=\, \{ \, v \in L^2(\Omega) \, : \, \curlbf v \in L^2(\Omega,\mathbb{R}^2) \, \},
\quad\mbox{and} \\
\mathbf{H}(\curl) \, &:=\, \{ \, \mathbf{v} \in L^2(\Omega,\mathbb{R}^2) \, : \, \curl \mathbf{v} \in L^2(\Omega) \, \}.
\end{align*}
Furthermore, it holds
\begin{align}
\label{inclusion-curl-2D}
\curlbf H^1_0(\Omega) \, \subset \, \Hzerodivzero. 
\end{align}
Recall here that the normal trace of $\curlbf v$ for some function $v\in H^1(\Omega)$ is equal to the tangential trace of $\nabla v$ on $\partial \Omega$.

We will use the following representation of $\Hzerodivzero$-functions that can be found in \cite[Theorem 3.1 and Corollary 3.1.]{girault2012finite}.
\begin{lemma}
\label{lemma:decomposition:2d}
Let $\Omega \subset \mathbb{R}^2$ be a simply-connected bounded Lipschitz domain. Then, for any $\mathbf{w} \in \Hzerodivzero$, there exists a stream function $q \in H^1_0(\Omega)$ such that
\begin{align*}
\mathbf{w} = \curlbf\hspace{1pt} q.
\end{align*}
The stream function can be characterized as the unique solution $q\in H^1_0(\Omega)$ to the problem
\begin{align*}
(\curlbf q , \curlbf v )_{L^2(\Omega)} \,\,=\,\, 
(\mathbf{w} , \curlbf v )_{L^2(\Omega)}
\qquad \mbox{for all } v\in H^1_0(\Omega),
\end{align*}
or equivalently, $- \Delta q = \curl \mathbf{w}$ weakly in $\Omega$. Further, by construction, $|\mathbf{w}(x)| = |\nabla q(x)|$ for a.e. $x\in \Omega$.
\end{lemma}

If $\Omega$ is not simply connected, the above representation of $\mathbf{w}$ as a curl of an $H^1$-function is still valid, thanks to the vanishing normal trace on $\partial \Omega$. In fact, as long as $\Omega$ is a Lipschitz domain, the Neumann condition $\mathbf{w}\cdot \mathbf{n}\vert_{\partial \Omega} = 0$ ensures that $\mathbf{w}$ can be canonically extended by zero to an $\Hdivzero$ function on the whole $\mathbb{R}^2$ (cf. \cite[Proposition 3.8]{Wen09}). For such an extension, the representation in Lemma \ref{lemma:decomposition:2d} for simply-connected domains can now be applied to any sufficiently large ball that contains $\Omega$. The only difference is that we can still ensure $q \in H^1(\Omega)$, but not necessarily that $q$ vanishes on $\partial \Omega$. 

With the above result, we can turn to the proof of Lemma \ref{lem-expo-deca-phi} for $d=2$. Note, however, that the decomposition does not come with a local estimate that allows to (locally) bound the $L^2$-norm of $q$ by the local $L^2$-norm of $\mathbf{w}$ while gaining an $H$ (see Lemma \ref{lem-stable-decomposition} for comparison). This is a major difference to the case $d=3$ and requires modified arguments.

\begin{proof}[Proof of Lemma \ref{lem-expo-deca-phi} for $d=2$]
Let $m \in \mathbb{N}$ be fixed and $m\geq 7$. Consider the cut-off function  $\eta:= \eta_{m-4}^T$ defined in \eqref{cut-off-func} with
$$
\eta=0 \quad \text {in} \quad \mathrm{N}^{m-4}(T), \quad \eta=1 \quad \text {in} \quad \Omega \backslash \mathrm{N}^{m-3}(T).
$$
We set $R^m_T:= \operatorname{supp}(\nabla \eta) = \overline{N^{m-3}(T)\backslash N^{m-4}(T)}$. As before, it holds
\begin{equation*}
\vertiii{\boldsymbol{\varphi}_T}_{\Omega\backslash N^m(T)}^2 \leq \left( \eta \, \mathbf{A}^{-1} \boldsymbol{\varphi}_T, \boldsymbol{\varphi}_T\right)_{L^2(\Omega)}.
\end{equation*}
Applying Lemma \ref{lemma:decomposition:2d} to $\boldsymbol{\varphi}_T \in \Wdivzero$ we obtain the existence of some $q_T \in H^1_0(\Omega)$ such that
\begin{equation*}
\boldsymbol{\varphi}_T = \curlbf \, q_T.
\end{equation*}
Since $\curlbf (\eta q_T) = q_T \, \curlbf\, \eta \,+\, \eta \,\, \curlbf \,q_T$
we obtain that
\begin{equation}\label{bili-spli-new}
\begin{aligned}
& \left( \eta \, \mathbf{A}^{-1} \boldsymbol{\varphi}_T, \boldsymbol{\varphi}_T\right)_{L^2(\Omega)} = a(  \eta \,\, \curlbf \,q_T, \boldsymbol{\varphi}_T) = a\left( \hspace{1pt}\curlbf (\eta q_T), \boldsymbol{\varphi}_T\right) - a\left( q_T \, \curlbf\, \eta , \boldsymbol{\varphi}_T\right)\\
= &\,\, a\left( \operatorname{\pi}_H \curlbf (\eta q_T), \boldsymbol{\varphi}_T\right) + a\left( \left(\id-\operatorname{\pi}_H\right) \curlbf (\eta q_T), \boldsymbol{\varphi}_T\right) - a\left( q_T \, \curlbf\, \eta  , \boldsymbol{\varphi}_T\right).
\end{aligned}
\end{equation}
Using \eqref{inclusion-curl-2D} we have $\curlbf (\eta q_T) \in \Hzerodivzero$, and hence, $\mathbf{r} :=\left(\id-\operatorname{\pi}_H\right)\curlbf (\eta q_T) \in  \Wdivzero$. Noting that $\mathbf{r} = \mathbf{0}$ in $T$, we find that
\begin{equation*}
a\left( \left(\id-\operatorname{\pi}_H\right) \curlbf (\eta q_T) , \boldsymbol{\varphi}_T\right) = a\left(\boldsymbol{\varphi}_T,\mathbf{r}\right) = F_T(\mathbf{r}) = 0.
\end{equation*}
Hence, \eqref{bili-spli-new} reduces to
\begin{eqnarray}
\label{bili-spli-new-2}
\left( \eta \, \mathbf{A}^{-1} \boldsymbol{\varphi}_T, \boldsymbol{\varphi}_T\right)_{L^2(\Omega)}
&=& a\left( \, \operatorname{\pi}_H \curlbf (\eta q_T)\,-\, q_T \hspace{1pt}\curlbf \eta \,, \boldsymbol{\varphi}_T\right).
\end{eqnarray}
Next, note that \,\,$\div(q_T \hspace{1pt}\curlbf \eta) =\nabla q_T \cdot \curlbf \eta \,\,\in\,\, L^2(\Omega)$\,\, 
and hence $q_T \hspace{1pt}\curlbf \eta \in \Hdiv$ and we can apply $\pi_H$ to it. With this, we split
\begin{eqnarray*}
 \operatorname{\pi}_H \curlbf (\eta q_T)\,-\, q_T \hspace{1pt}\curlbf \eta 
 &=&  \operatorname{\pi}_H \left( \curlbf (\eta q_T)\,-\, q_T \hspace{1pt}\curlbf \eta \right)
 - (\id - \pi_H) (q_T \hspace{1pt}\curlbf \eta)
 \\
&=& \operatorname{\pi}_H \hspace{-1pt}\left( \eta \, \curlbf q_T\right) \,-\, (\id -\pi_H)(q_T \hspace{1pt}\curlbf \eta).
\end{eqnarray*}
Inserting this identity into \eqref{bili-spli-new-2} yields
\begin{eqnarray}
\label{bili-spli-new-3}
\left( \eta \, \mathbf{A}^{-1} \boldsymbol{\varphi}_T, \boldsymbol{\varphi}_T\right)_{L^2(\Omega)} 
&=& \underbrace{a\left( \, \operatorname{\pi}_H \hspace{-1pt}\left(\eta \, \curlbf q_T\right) \,, \boldsymbol{\varphi}_T\,\right)}_{=:\,\mbox{I}} \,-\, \underbrace{ a\left( \, (\id -\pi_H)(q_T \hspace{1pt}\curlbf \eta) \,, \boldsymbol{\varphi}_T\,\right)}_{=:\,\mbox{II}}.
\end{eqnarray}
For the first term (I) on the right-hand side in \eqref{bili-spli-new-3}, we observe that 
\begin{align*}
  \operatorname{\pi}_H(\eta \, \curlbf q_T)=\operatorname{\pi}_H  \curlbf q_T = \operatorname{\pi}_H \boldsymbol{\varphi}_T =0\quad\text{on }\Omega \backslash \mathrm{N}^{m-2}(T),  
\end{align*}
and since $\eta=0$ on $\mathrm{N}^{m-4}(T)$,
we find that $\operatorname{supp} \left( \operatorname{\pi}_H(\eta \,  \curlbf q_T) \right) \subset N^1(R_T^m)$ and 
\begin{eqnarray*}
\mbox{I} \,\,\, = \,\,\, a\left(\operatorname{\pi}_H(\eta \, \curlbf q_T), \boldsymbol{\varphi}_T\right) &\lesssim& \beta \,\| \eta \, \curlbf q_T \|_{L^2\left(N^2(R_T^m)\right)} \|\boldsymbol{\varphi}_T\|_{L^2\left(N^1(R_T^m)\right)} \\
\nonumber
&\lesssim& \beta \, \|\boldsymbol{\varphi}_T\|_{L^2\left(N^{m-1}(T)\backslash N^{m-5}(T)\right)}^2.
\end{eqnarray*}
In order to estimate the second term (II) on the right-hand side in \eqref{bili-spli-new-3}, we use the local error estimate in Remark \ref{remark-error-estimates-piH} to obtain for each $T^{\prime}\in \mathcal{T}_H$
\begin{eqnarray}
\label{proof-Lem4.2-2D-step1}
\lefteqn{ \left\| (\id -\pi_H) (q_T \hspace{1pt}\curlbf \eta) \right\|_{L^2(T^{\prime})}^2 }\\
\nonumber&\lesssim& \underset{K\subset N^2(T^{\prime})}{\sum_{K \in \mathcal{T}_H}}
\min_{\mathbf{v}_K \in \mathcal{RT}_{\hspace{-1pt}k}(K)} \| q_T \hspace{1pt}\curlbf \eta - \mathbf{v}_K \|_{L^2(K)}^2+H^2\| (\id - \operatorname{P}_H)(\div (q_T \hspace{1pt}\curlbf \eta))\|_{L^2(K)}^2 .
\end{eqnarray}
Since $\div(\curlbf \eta)=0$ and since $\eta$ is a $\mathcal{T}_H$-piecewise linear function (in particular, $\left.\eta\right\rvert_K \in \mathbb{P}_1(K)$), we note that $\tfrac{1}{|K|}(1,q_T)_{L^2(K)} \curlbf \eta \in \mathcal{RT}_{\hspace{-1pt}k}(K)$ and $ \| \hspace{1pt} |\curlbf \eta| \hspace{1pt} \|_{L^{\infty}(\Omega)} =\| \hspace{1pt} |\nabla \eta| \hspace{1pt} \|_{L^{\infty}(\Omega)} \lesssim \tfrac{1}{H}$. Therefore, we can estimate
\begin{eqnarray}
\label{proof-Lem4.2-2D-step2}
\nonumber\lefteqn{ \min_{\mathbf{v}_K \in \mathcal{RT}_{\hspace{-1pt}k}(K)} \| q_T \hspace{1pt}\curlbf \eta - \mathbf{v}_K \|_{L^2(K)}
\,\,\,\le\,\,\, \| (q_T-\tfrac{1}{|K|}(1,q_T)_{L^2(K)}) \hspace{1pt}\curlbf \eta \|_{L^2(K)} } \\
&\lesssim& \frac{1}{H} \| q_T-\tfrac{1}{|K|}(1,q_T)_{L^2(K)}  \|_{L^2(K)}
\,\,\, \lesssim \,\,\, \| \nabla q_T \|_{L^2(K)}
\,\,\, = \,\,\, \| \boldsymbol{\varphi}_T  \|_{L^2(K)},
\end{eqnarray}
where we used the Poincar\'e inequality for functions with zero average on $K$, as well as $| \nabla q_T|= |\boldsymbol{\varphi}_T|$ according to Lemma \ref{lemma:decomposition:2d}. On the other hand, the local $L^2$-stability of $\operatorname{P}_H$ implies
\begin{eqnarray}
\label{proof-Lem4.2-2D-step3}
\nonumber\lefteqn{ H \| (\id - \operatorname{P}_H)(\div (q_T \hspace{1pt}\curlbf \eta))\|_{L^2(K)}
\,\,\,\lesssim \,\,\, H \| \div (q_T \hspace{1pt}\curlbf \eta)\|_{L^2(K)} }\\
&=& H  \| \nabla q_T \cdot \curlbf \eta\|_{L^2(K)}
\,\,\,\lesssim \,\,\, \| \nabla q_T \|_{L^2(K)} \,\,\, = \,\,\, \| \boldsymbol{\varphi}_T  \|_{L^2(K)}.
\end{eqnarray}
Inserting the bounds \eqref{proof-Lem4.2-2D-step2} and \eqref{proof-Lem4.2-2D-step3} into \eqref{proof-Lem4.2-2D-step1}, we obtain that
\begin{eqnarray*}
\left\| (\id -\pi_H) (q_T \hspace{1pt}\curlbf \eta) \right\|_{L^2(T^{\prime})} 
&\lesssim& \| \boldsymbol{\varphi}_T  \|_{L^2(N^2(T^{\prime}))} .
\end{eqnarray*}
Since $\mbox{supp}((\id -\pi_H) (q_T \hspace{1pt}\curlbf \eta)) \subset N^1(R_T^m)$, we can bound the second term (II) on the right-hand side in \eqref{bili-spli-new-3} as follows:
\begin{eqnarray*}
\mbox{II} &\lesssim& \beta \, \| \boldsymbol{\varphi}_T \|_{L^2(N^1(R_T^m))} \,\,\| (\id -\pi_H)(q_T\, \curlbf \eta) \|_{L^2(N^1(R_T^m))} \\
&\lesssim& \beta \, \| \boldsymbol{\varphi}_T \|_{L^2(N^1(R_T^m))} \,\, \| \boldsymbol{\varphi}_T \|_{L^2(N^3(R_T^m))} 
\,\,\,
\lesssim 
\,\,\, \beta \,\| \boldsymbol{\varphi}_T \|_{L^2(N^{m}(T)\backslash N^{m-7}(T))}^2.
\end{eqnarray*}
Combining the estimates for terms I and II  with \eqref{bili-spli-new-3}, we deduce that
\begin{eqnarray*}
 \vertiii{\boldsymbol{\varphi}_T}_{\Omega\backslash N^m(T)}^2 
 &\leq& \, C \, \beta \, \|\boldsymbol{\varphi}_T\|_{L^2\left(N^{m}(T)\backslash N^{m-7}(T)\right)}^2 \\  
  &\leq& C \, \frac{\beta}{\alpha} \,
 \left(\vertiii{\boldsymbol{\varphi}_T}_{\Omega\backslash N^{m-7}(T)}^2  -\vertiii{\boldsymbol{\varphi}_T}_{\Omega\backslash N^m(T)}^2 \right), 
\end{eqnarray*}
where $C>0$ is a constant is independent of $T$, $m$, $H$, $\mathbf{A}$.  Setting $\theta := \frac{C\tfrac{\beta}{\alpha}}{1+C\tfrac{\beta}{\alpha}} < 1$, we find that $\vertiii{\boldsymbol{\varphi}_T}_{\Omega\backslash N^m(T)}^2 
 \leq\theta \vertiii{\boldsymbol{\varphi}_T}_{\Omega\backslash N^m(T)}^2$, and hence, 
\begin{equation*}
\vertiii{\boldsymbol{\varphi}_T}_{\Omega\backslash N^m(T)}^2 \lesssim \theta^{\lfloor m / 7\rfloor}\vertiii{\boldsymbol{\varphi}_T}^2
\end{equation*}
by a recursive application of the estimate.
\end{proof}

Next, we consider Lemma \ref{lem-expo-deca-phi-phi-l} for $d=2$.

\begin{proof}[Proof of Lemma \ref{lem-expo-deca-phi-phi-l} for $d=2$]
The proof remains similar to the case $d=3$. Using Lemma \ref{lemma:decomposition:2d}, we write $\boldsymbol{\varphi}_T = \curlbf q_T = (\id-\operatorname{\pi}_H)\curlbf q_T$ for some $q_T \in H^1_0(\Omega)$. Now, consider the cut-off function $\eta:= 1-\eta_{m-2}^T$ with $\eta_{m-2}^T$ defined as in \eqref{cut-off-func} with $m$ replaced by $m-2$. In particular, $\eta=1$ in $\mathrm{N}^{m-2}(T)$ and $\eta=0$ in $\Omega \backslash \mathrm{N}^{m-1}(T)$. Since $\boldsymbol{\varphi}_{T}^m$ is the Galerkin best-approximation to $\boldsymbol{\varphi}_{T}$ in $W_{\operatorname{div}\hspace{-1pt}0}(N^m(T))\subset \Wdivzero$, and since $(\id-\operatorname{\pi}_H)\curlbf(\eta\, q_T) \in W_{\operatorname{div}\hspace{-1pt}0}(N^m(T))$, we obtain that
\begin{equation*}
\vertiii{\boldsymbol{\varphi}_T-\boldsymbol{\varphi}_{T}^m} 
\,\le\,
\vertiii{ (\id-\operatorname{\pi}_H)\curlbf ((1-\eta)q_T) }
\, = \,\vertiii{ (\id-\operatorname{\pi}_H)\curlbf ((1-\eta)q_T) }_{ \Omega \setminus N^{m-3}(T)}.
\end{equation*}
For $T^{\prime}\in \mathcal{T}_H$ with $T^{\prime} \subset \Omega \setminus N^{m-3}(T)$, we apply the error estimate from Remark \ref{remark-error-estimates-piH} to obtain with the usual arguments that
\begin{eqnarray*}
 \lefteqn{ \left\| (\id-\operatorname{\pi}_H)\curlbf ((1-\eta)q_T) \right\|_{L^2(T^{\prime})}^2 
\,\,\,\lesssim\,\,\, 
\underset{K\subset N^2(T^{\prime})}{\sum_{K \in \mathcal{T}_H}}
\min_{\mathbf{v}_K \in \mathcal{RT}_{\hspace{-1pt}k}(K)} \|\curlbf ((1-\eta)q_T) - \mathbf{v}_K \|_{L^2(K)}^2 }\\
&\lesssim& 
 \| (1-\eta) \curlbf q_T \|_{L^2(N^2(T^{\prime}))}^2 + 
 \underset{K\subset N^2(T^{\prime})}{\sum_{K \in \mathcal{T}_H}}
 \| (q_T- \tfrac{1}{|K|} (1,q_T)_{L^2(K)}) \, \curlbf (1-\eta)\|_{L^2(K)}^2 \\
&\lesssim& 
 \| \curlbf q_T \|_{L^2(N^2(T^{\prime}))}^2 + 
 \| \nabla q_T \|_{L^2(N^2(T^{\prime}))}^2 
 \,\,\, \lesssim \,\,\,  \| \boldsymbol{\varphi}_T \|_{L^2(N^2(T^{\prime}))}^2.
\end{eqnarray*}
We conclude that
\begin{equation*}
\vertiii{\boldsymbol{\varphi}_T-\boldsymbol{\varphi}_{T}^m} \,\,\, \lesssim  \,\,\, \vertiii{\boldsymbol{\varphi}_T}_{ \Omega\backslash N^{m-5}(T)},
\end{equation*}
and Lemma \ref{lem-expo-deca-phi} for $d=2$ completes the proof.
\end{proof}

It remains to prove Lemma \ref{lem-expo-deca-loca-erro} for $d=2$.

\begin{proof}[Proof of Lemma \ref{lem-expo-deca-loca-erro} for $d=2$]
We apply the same strategy as before. Let $\mathbf{e}:=\left(\operatorname{\mathcal{C}}-\operatorname{\mathcal{C}}^m\right)\mathbf{v} \in \Wdivzero$ and for each $T \in \mathcal{T}_H$ consider the cut-off functions $\eta_T:=\eta_T^{m+1}$ defined as in \eqref{cut-off-func} with $m$ replaced by $m+1$. In particular, $\eta_T=0$ in $\mathrm{N}^{m+1}(T)$ and $\eta_T=1$ in $\Omega \backslash \mathrm{N}^{m+2}(T)$. Using Lemma \ref{lemma:decomposition:2d}, we write $\mathbf{e} = \mathbf{e} - \operatorname{\pi}_H \mathbf{e} = \curlbf q$ for some $q \in H^1_0(\Omega)$. Now consider $\mathbf{e}_T: = (\id-\operatorname{\pi}_H)\curlbf(\eta_T \,q)$ with support in $\overline{\Omega \setminus \mathrm{N}^{m}(T)}$. Then, we have $a(\mathbf{e}_T, \mathbf{e})= a(\mathbf{e}_T, \left(\operatorname{\mathcal{C}}_T-\operatorname{\mathcal{C}}_{T}^m\right) \mathbf{v} ) = a(\mathbf{e}_T, \operatorname{\mathcal{C}}_T \mathbf{v}) = ( \mathbf{A}^{-1} \mathbf{e}_T ,  \mathbf{v} )_{L^2(T)} = 0$, and hence,
\begin{eqnarray*}
\vertiii{\mathbf{e}}^2 
&=& \sum_{T \in \mathcal{T}_H} a(\mathbf{e}-\mathbf{e}_T, \left(\operatorname{\mathcal{C}}_T-\operatorname{\mathcal{C}}_{T}^m\right) \mathbf{v} ) 
\,\,\,=\,\,\, \sum_{T \in \mathcal{T}_H} a(\left(\id - \operatorname{\pi}_H \right)\curlbf\left(\left(1-\eta_T\right)\mathbf{q}\right), \left(\operatorname{\mathcal{C}}_T-\operatorname{\mathcal{C}}_{T}^m\right) \mathbf{v} ).
\end{eqnarray*}
The contribution $\left(\id - \operatorname{\pi}_H \right)\curlbf\left(\left(1-\eta_T\right)\mathbf{q}\right)$ is estimated as in the proof of Lemma \ref{lem-expo-deca-phi-phi-l} for $d=2$ and the contribution $\left(\operatorname{\mathcal{C}}_T-\operatorname{\mathcal{C}}_{T}^m\right) \mathbf{v}$ is estimated directly estimated with Lemma \ref{lem-expo-deca-phi-phi-l}. Combining everything yields $\vertiii{\mathbf{e}}^2 \,\lesssim\, (m^d+1)\, \theta^m\,
\vertiii{ \mathbf{e} }\, \vertiii{ \mathbf{v} }$ which completes the proof.
\end{proof}

\end{appendix}

\end{document}